\documentclass[a4paper,notitlepage,twoside,reqno,10pt]{amsart}

\usepackage{bbm,pifont,latexsym}

\usepackage{dcolumn,indentfirst}
\usepackage[hypertexnames=false,bookmarksopen=true,linktocpage=true,pdfstartview={XYZ null null 1.25}]{hyperref}
\usepackage{amsmath,amssymb,amscd,amsthm,amsfonts,mathrsfs,appendix}
\usepackage{color,graphicx,xcolor,graphics,subfigure,extarrows,caption2}
\usepackage{pdfpages,titletoc}

\usepackage{autonum}
\newcommand{\retainlabel}[1]{\label{#1}\sbox0{\ref{#1}}}

\newtheorem{thm}{Theorem}[section]
\newtheorem{lem}[thm]{Lemma}

\newtheorem{prop}[thm]{Proposition}

\newtheorem*{mainthm}{Main Theorem}
\newtheorem*{cor1}{Corollary}

\theoremstyle{definition}
\newtheorem*{defi}{Definition}

\newtheorem*{rmk}{Remark}

\newtheorem*{note}{Notations}

\newcommand{\ly}{\underset{*}{<}}
\newcommand{\gy}{\underset{*}{>}}
\newcommand{\Aeq}[1]{{\,(\doteqdot{#1}\ldots)}}
\newcommand{\tuta}[1]{{\widetilde{#1}}}

\newcommand{\EC}{\widehat{\mathbb{C}}}
\newcommand{\C}{\mathbb{C}}
\newcommand{\D}{\mathbb{D}}
\newcommand{\BH}{\mathbb{H}}
\newcommand{\N}{\mathbb{N}}

\newcommand{\R}{\mathbb{R}}

\newcommand{\V}{\mathbb{V}}
\newcommand{\Z}{\mathbb{Z}}
\newcommand{\OD}{\overline{\mathbb{D}}}
\newcommand{\OV}{\overline{\mathbb{V}}}

\newcommand{\MA}{\mathcal{A}}

\newcommand{\MF}{\mathcal{F}}
\newcommand{\MH}{\mathcal{H}}

\newcommand{\MMR}{\mathcal{R}}

\newcommand{\MU}{\mathcal{U}}

\newcommand{\ii}{\textup{i}}
\newcommand{\re}{\textup{Re}\,}
\newcommand{\im}{\textup{Im}\,}

\newcommand{\Exp}{\textup{Exp}}

\newcommand{\Mod}{\textup{mod}}

\newcommand{\SD}{\textsf{D}}
\newcommand{\Phia}{\Phi_{attr}}

\usepackage{enumerate,enumitem}

\renewcommand\theequation{\thesection.\arabic{equation}} 
\makeatletter\@addtoreset{equation}{section}\makeatother %

\begin{document}

\author{FEI YANG}
\address{Department of Mathematics, Nanjing University, Nanjing 210093, P. R. China}
\email{yangfei@nju.edu.cn}

\title[Parabolic renormalization for local degree three]{Parabolic and near-parabolic renormalizations for local degree three}

\begin{abstract}
The invariant class under parabolic and near-parabolic renormalizations constructed by Inou and Shishikura has been proved to be extremely useful in recent years. It leads to several important progresses on the dynamics of certain holomorphic maps with critical points of local degree two. In this paper, we construct a new class consisting of holomorphic maps with critical points of local degree \textit{three} which is invariant under parabolic and near-parabolic renormalizations.
As potential applications, some results of cubic unicritical polynomials can be obtained similarly as the quadratic case. For example, the existence of cubic unicritical Julia sets with positive area, the characterizations of the topology and geometry of cubic irrationally indifferent attractors etc.
\end{abstract}



\date{\today}

\thanks{This work was supported by the NSFC (Nos.\,12222107 and 12071210), the NSF of Jiangsu Province (No.\,BK20191246) and the CSC program (2014/2015).}

\maketitle


\section{Introduction}\label{introduction}

Let $f(z)$ be a non-linear holomorphic function defined in a neighborhood of a fixed point $z_0\in\C$. The number $f'(z_0)$ is called the \textit{multiplier} of $z_0$. If $f'(z_0)$ is a root of the unity, then $z_0$ is called a \emph{parabolic fixed point} of $f$. In particular, $z_0$ is called $1$-\emph{parabolic} if $f'(z_0)=1$ and it is called \emph{non-degenerate} if further $f''(z_0)\neq 0$.
The dynamical behavior in the parabolic basin is simple but it may become extremely complicated once a perturbation on the map $f$ is made. The main tools to analyze such complicated phenomena are Fatou coordinates and horn maps, which were developed by Douady-Hubbard \cite{DH84, DH85b}, Lavaurs \cite{Lav89}, Inou-Shishikura \cite{IS08} and others.

The definitions of parabolic and near-parabolic renormalizations were introduced firstly in \cite{Shi98}, where the Hausdorff dimensions of some quadratic Julia sets and the Mandelbrot set were studied. Near-parabolic renormalization is also called \emph{cylinder renormalization}, which was introduced by Yampolsky in the study of analytic circle homeomorphisms with one critical point \cite{Yam03} (see also \cite{LY14}). In \cite{Shi98}, an invariant class under the parabolic renormalization operator was given. However, in order to iterate the near-parabolic renormalization operator infinitely many times, the class defined in \cite{Shi98} cannot serve for the purpose anymore. In 2008, Inou and Shishikura introduced an infinite-dimensional class of complex analytic parabolic maps, and proved that such class is invariant under the parabolic renormalization operator and a perturbation of this class is invariant under the near-parabolic renormalization operator \cite{IS08}.

Each element in the Inou-Shishikura class is a holomorphic map having exactly one critical point and one critical value. Moreover, the critical point is simple (i.e., with local degree two). This invariant class has led to many important progresses on the dynamics of quadratic polynomials. We will mention them in a moment. In order to study the dynamics of unicritical holomorphic maps with higher degrees, a natural problem is to extend Inou-Shishikura's class to higher degrees. Our main goal in this paper is to do this for the cubic unicritical case.

\subsection{The  invariant class and the main result}

The most difficult part to extend Inou-Shishikura's results to higher degrees is to define new invariant classes (especially the choice of the domains of definitions of the maps in the class). Therefore, before stating our main result, we first introduce a class.

\begin{defi}[{The class $\MF_1$}]
Let $P$ be a degree $5$ rational map:
\begin{equation}\label{equ:P}
P(z)=\frac{z\,(1+\frac{z}{\mu})^4}{(1+z)^2}, \text{\quad where }\mu=11-4\sqrt{6}>1.
\end{equation}
Then $P$ has a parabolic fixed point at $0$ with multiplier $1$, and $4$ critical points at $-\mu$, $-1$, $\infty$ and $cp_P=1-\tfrac{2}{3}\sqrt{6}\in(-1,0)$, where $cp_P$ has local degree \emph{three} (see \S \ref{subsec:P}). It is easy to see that $P:\C\setminus\{-1\}\to\C$ has two finite critical values $0=P(-\mu)$ and $cv_P=P(cp_P)$.
Let $V$ be a domain in $\C\setminus\{-1\}$ containing $0$. Define
\begin{equation}
\mathcal{F}_1:=
\left\{f=P\circ\varphi^{-1}:\varphi(V)\to\C
\left|
\begin{array}{l}
\varphi:V\to\C \text{ is univalent},\\
\varphi(0)=0\text{ and }\varphi'(0)=1
\end{array}
\right.
\right\}.
\end{equation}
If $f\in\mathcal{F}_1$, then $0$ is a $1$-parabolic fixed point of $f$. If $cp_P\in V$, then $cp_f:=\varphi(cp_P)$ is a critical point of $f$ with local degree \textit{three} and $cv:=cv_P$ is a critical value of $f$.
\end{defi}

\begin{mainthm}[{Invariance of $\mathcal{F}_1$}]\label{thm-main}
There exists a Jordan disk $V\subset\C$ containing $0$ and $cp_P$ such that $\mathcal{F}_1$ is invariant under the \emph{parabolic renormalization} $\MMR_0$:
\begin{equation}
\mathcal{R}_0(\mathcal{F}_1)\subset\mathcal{F}_1.
\end{equation}
That is, for any $f\in\mathcal{F}_1$, $\mathcal{R}_0 f$ is well-defined so that $\mathcal{R}_0 f=P\circ\psi^{-1}\in\mathcal{F}_1$. Moreover, there exists a simply connected domain $V'\subset\C$ (independent of $f$ and $\psi$) containing $\overline{V}$ such that $\psi$ extends to a univalent function from $V'$ to $\C$.
\end{mainthm}

The explicit definitions of $V$ and $V'$ will be given in \S\ref{sec-definitions} (see Figure \ref{Fig-V}). The property $\overline{V}\subset V'$ is important. This implies that the new domain of definition of $\psi$ is strictly larger than that of the original $\varphi$.
For the definition of parabolic renormalization $\MMR_0$, see \S\ref{subsec-ParaRenorm}.

If we require further that every normalized univalent map $\varphi:V\to\C$ in the definition of $\MF_1$ has a quasiconformal extension to $\C$ (and the new class is denoted by $\widetilde{\MF}_1$), then one can designate a one-to-one correspondence between $\widetilde{\MF}_1$ and the Teichm\"{u}ller space of $\C\setminus\overline{V}$ as in \cite[\S 6]{IS08}.
Let $d(\cdot,\cdot)$ be the distance on $\widetilde{\MF}_1$ induced from the Teichm\"{u}ller distance. By following \cite[\S 6]{IS08}, we have the following immediate consequence of the Main Theorem.

\begin{cor1}
The parabolic renormalization operator $\MMR_0$ is a uniform contraction:
\begin{equation}
d(\MMR_0(f),\MMR_0(g))\leq\lambda\, d(f,g) \text{\quad for~} f,g\in\widetilde{\MF}_1,
\end{equation}
where $\lambda=e^{-2\pi\,\Mod(V'\setminus\overline{V})}<1$.
\end{cor1}

Indeed, Inou and Shishikura proved the uniform contraction by using Royden-Gardiner's non-expanding theorem on Teichm\"{u}ller distance, and the proof is unrelated to the specific formulas of the maps in the invariant class. The unique requirement is $\overline{V}\subset V'$, which is stated in the Main Theorem.

By using a completely similar argument of the continuity of the horn maps to \cite[\S\,7]{IS08}, the \textit{near-parabolic renormalization} $\MMR$ can be defined on the class
\begin{equation}
e^{2\pi\ii\,\MA(\alpha_*)}\times\MF_1=\{e^{2\pi\ii\alpha}h(z) \,|\, \alpha\in\MA(\alpha_*) \text{~and~} h\in\MF_1\},
\end{equation}
where $\MA(\alpha_*)=\{\alpha\in\C\,|\, 0<|\alpha|<\alpha_*, |\arg\alpha|<\pi/4 \text{~or~} |\arg(-\alpha)|<\pi/4\}$
and $\alpha_*>0$ is a small number.
In particular, $\MMR$ can be expressed as a skew product $\MMR:(\alpha,h)\mapsto(-1/\alpha,\MMR_\alpha h)$, where $\MMR_\alpha$ is the renormalization in the \textit{fiber direction}. The class $\MF_1$ is invariant under the fiber renormalization, and the near-parabolic renormalization can be acted infinitely many times on any map $e^{2\pi\ii\alpha}h\in e^{2\pi\ii\,\MA(\alpha_*)}\times\MF_1$ if $\alpha$ is of \textit{high type}, i.e., the coefficients of the continued fraction expansion of $\alpha$ are all larger than $1/\alpha_*$.
For more details on the near-parabolic renormalization $\MMR$ and the renormalization $\MMR_\alpha$ in the fiber direction, we refer to \cite{IS08}.

The proof of the Main Theorem is strongly inspired by the method of Inou and Shishikura \cite{IS08}. We will follow the framework of their paper and do some numerical calculations.
We need to check a number of inequalities by computer. These inequalities are all elementary functions evaluated at \textit{explicit} values. We first check them by \textsf{Maple} and \textsf{Mathematica} numerically, and then use \textsf{INTLAB}\footnote{Inou and Shishikura also used INTLAB to check the estimates rigorously. This tool can be found in \url{https://www.tuhh.de/ti3/rump/intlab/}.} on \textsf{MATLAB} to verify the estimates rigorously with interval arithmetics.
See \cite{Yan24M} for specific calculations.

Compared to \cite{IS08}, the main difficulty in the calculations here is the appearance of a pole in the model map $P$ (note that the corresponding model in \cite{IS08} is a cubic polynomial). More differences between the proofs of the Main Theorem and \cite[Main Theorem 1]{IS08} can be found as the reader gets into the details.

In our previous version \cite{Yan16}, an invariant class of holomorphic maps with local degree three has already been constructed and the proof also relies on numerical estimates. However, because of the complexity of the model map there, some numerical calculations need to compute the maximum and minimum of continuous functions on closed intervals, which cause the results there not rigorous. The class we choose in this paper is simpler and the numerical calculations can be checked rigorously. A key idea behind the construction of this new model map $P$ is to put all the preimages of the parabolic fixed point on the real axis for which all calculations can be carried out precisely.

\subsection{Applications}

The construction of the Inou-Shishikura class aims to study the dynamics of holomorphic maps whose critical points have local degree two. As the first remarkable application, Buff and Ch\'{e}ritat proved the existence of quadratic Julia sets with positive area and one important ingredient in the proof is the fine control of the post-critical sets \cite{BC12}.

Cheraghi developed several elaborate analytic techniques and carried out a quantitative analysis of Inou-Shishikura's renormalization scheme in \cite{Che13} and \cite{Che19}. In the past few years, these techniques led to some of the recent major progresses on the dynamics of quadratic polynomials. For examples:
\begin{itemize}
\item The Marmi-Moussa-Yoccoz conjecture for rotation numbers of high type has been proved in \cite{CC15};
\item The Feigenbaum Julia sets with positive area (which is very different from the examples in \cite{BC12}) have been found in \cite{AL22};
\item The local connectivity of the Mandelbrot set at some infinitely satellite renormalizable points was proved in \cite{CS15};
\item The statistical properties of the dynamics of certain quadratic polynomials was characterized in \cite{AC18};
\item The topological structure and the Hausdorff dimension of the post-critical sets of the maps in Inou-Shishikura's class with an irrationally indifferent fixed point have been studied in \cite{Che22b}, \cite{Che23}, \cite{SY24} and \cite{CDY20} respectively;
\item The topological structure of the post-critical sets of some Feigenbaum maps was characterized in \cite{CP22} and \cite{CP23}.
\end{itemize}

Some other consequences of Inou-Shishikura's near-parabolic renormalization theory include: Self-similarity of the boundaries of high type Siegel disks in the quadratic family \cite{CY16}, the existence of non-renormalizable Cremer cubic polynomials having Julia sets with positive area \cite{QQ20}, and the existence of cubic rational maps having smooth degenerate Herman rings \cite{Yan24} etc.

As one potential application of the class $\MF_1$ defined in this paper, it is likely that one can prove the existence of Julia sets with positive area for cubic unicritical polynomials by following Buff and Ch\'{e}ritat. We also believe that the techniques developed by Cheraghi in \cite{Che13} and \cite{Che19} based on Inou-Shishikura's class can be transplanted to the class $\MF_1$ without difficulty. Therefore, the results mentioned above may be applied to cubic unicritical polynomials.

Recently, Ch\'{e}ritat extended the parabolic and near-parabolic renormalization theory to all finite degrees with several critical points and only one critical value (see \cite{Che22}). In the case of degree two, every map in Inou-Shishikura's class has exactly one critical point. By using the language in \cite{Che22}, the ``structure" of every map in the Inou-Shishikura class is the restriction of the ``structure" of some map in Ch\'{e}ritat's class. Hence, the class (with local degree two) found in \cite{Che22} is smaller than Inou-Shishikura's. For the same reason, the class $\MF_1$ defined in this paper is larger than the class in \cite{Che22} (the local degree three case). This implies that our results can be applied to more holomorphic maps.

On the other hand, both the class defined in this paper and the class defined by Ch\'{e}ritat can be used to study the cubic unicritical polynomials $z\mapsto z^3+c$. Although these cubic polynomials are not contained in $\MF_1$, but their first parabolic renormalizations are (see \S\ref{sec-subcover}).

We would like to mention that recently, Kapiamba proved that for all roots of unity, there exist corresponding classes which are invariant under the parabolic or near-parabolic renormalizations \cite{Kap23}. The related renormalization theory can be used to study the quadratic polynomials with an irrationally indifferent fixed point whose rotation number is of \textit{valley-type}.

\begin{note}
We use $\N$, $\R$ and $\C$ to denote the set of all natural numbers, real numbers and complex numbers, respectively. The Riemann sphere, the upper (lower) half plane and the unit disk are denoted by $\EC=\C\cup\{\infty\}$, $\BH_\pm=\{z\in\C:\pm\,\im z>0\}$ and $\D=\{z\in\C:|z|<1\}$ respectively. A round disk is denoted by $\D(a,r)=\{z\in\C:|z-a|<r\}$ and $\OD(a,r)$ is its closure. We use $\C^*=\C\setminus\{0\}$ and $\R_\pm=\{x\in\R:\pm\,x>0\}$ to denote the punctured complex plane and the positive (negative) real line respectively. We use $d_X(\cdot,\cdot)$ to denote the Poincar\'{e} distance in the hyperbolic Riemann surface $X$ and $\D_X(a,r)=\{z\in X:d_X(z,a)<r\}$ to denote the hyperbolic disk in $X$. For a complex number $z\neq 0$, $\arg\,z$ and $\log\,z$ denote a suitably chosen branch of the argument and logarithm respectively.
\end{note}

\section{Parabolic renormalization and preliminaries of $P$}\label{sec:ParaRenorm}

In this section, we first review the definitions of Fatou coordinates, horn maps and parabolic renormalization. For more details, see \cite{Shi00a}. After that we give some preliminary dynamical properties of the rational map $P$.

\subsection{Fatou coordinates and horn maps}\label{subsec-ParaRenorm}

Let $f(z)=z+a_2 z^2+\mathcal{O}(z^3)$ be a holomorphic function with $a_2\neq 0$, i.e., $f$ has a non-degenerate $1$-parabolic fixed point at the origin. After a coordinate change $w=-\frac{1}{a_2 z}$, the parabolic fixed point $0$ moves to $\infty$ and the dynamics of $f$ in this new coordinate is
\begin{equation}
F(w)=w+1+\tfrac{b_1}{w}+\mathcal{O}(\tfrac{1}{w^2})
\end{equation}
near $\infty$, where $b_1\in\C$ is a constant\footnote{\,If $f(z)=z+a_2 z^2+a_3 z^3 + \mathcal{O}(z^4)$ with $a_2\neq 0$, then a direct calculation shows that $b_1=1-{a_3}/{a_2^2}$. This constant is called the \textit{iterative residue} of $f$.} depending on $f$.

\begin{thm}[{and the definitions of Fatou coordinates and horn maps, \cite{Shi00a}}]\label{thm-Fatou-horn}
For a sufficiently large $L>0$, there are univalent maps $\Phi_{attr}=\Phi_{attr,F}:\{w\in\C:\re w-L>-|\im w|\}\to\C$ and $\Phi_{rep}=\Phi_{rep,F}:\{w\in\C:\re w+L<|\im w|\}\to\C$ such that
\begin{equation}
\Phi_s(F(w))=\Phi_s(w)+1 ~~(s=attr,rep)
\end{equation}
holds in the region where both sides are defined. These two maps $\Phi_{attr}$ and $\Phi_{rep}$ are unique up to an additive constant. They are called the \emph{attracting} and \emph{repelling Fatou coordinates} respectively.

In the region $V_\pm=\{w\in\C:\pm\,\im w>|\re w|+L\}$, both Fatou coordinates are defined. The \emph{horn map} $E_F$ in $\Phi_{rep}(V_\pm)$ is defined as
\begin{equation}
E_F:=\Phi_{attr}\circ\Phi_{rep}^{-1}.
\end{equation}
There exists a more larger $L'>0$ such that $E_F$ extends holomorphically to $\{z\in\C:|\im z|> L'\}$ and satisfies $E_F(z+1)=E_F(z)+1$. Moreover, there are constants $c_{upper}$ and $c_{lower}$ such that
\begin{equation}
\lim_{\im z\to+\infty}E_F(z)-z= c_{upper}  \text{\quad and\quad} \lim_{\im z\to-\infty}E_F(z)-z= c_{lower},
\end{equation}
where $c_{lower}-c_{upper}=2\pi\ii b_1$.
\end{thm}

Since the Fatou coordinates are unique up to an additive constant, it is convenient to make a normalization for them. In this paper, the attracting Fatou coordinate is \emph{normalized} by $\Phi_{attr}(cv)=1$ for a marked critical value $cv$ (we will see later such a critical value is always defined). For the repelling Fatou coordinate, we normalize it such that $c_{upper}=0$, i.e.,
\begin{equation}
E_F(z)=z+o(1) \text{\quad as\quad} \im z\to+\infty.
\end{equation}

\begin{defi}[{Parabolic renormalization}]
Let $\Exp(z)=e^{2\pi\ii z}:\C/\Z\to\C^*$ be an isomorphism. The \textit{parabolic renormalization}\footnote{This is the parabolic renormalization for the upper end of the cylinder $\C/\Z$, and the parabolic renormalization  for the lower end can be defined similarly. See \cite[\S 3]{IS08}.} of $f$ is defined as
\begin{equation}
\MMR_0 f:=\Exp\circ E_f\circ (\Exp)^{-1},
\end{equation}
where $E_f=E_F$ is the horn map of $f$, defined in Theorem \ref{thm-Fatou-horn} and normalized as $E_f (z) = z+o(1)$ as $\im z\to+\infty$.
Then $\MMR_0 f$ extends holomorphically to $0$, $(\MMR_0 f)(0) = 0$ and $(\MMR_0 f)'(0) = 1$. So $\MMR_0 f$ has again a $1$-parabolic fixed point at $0$.
Note that $\MMR_0 f$ is a holomorphic germ at this moment. In the following we will extend the domain of definition of this germ, so that it contains a critical point and has the desired covering structure.
\end{defi}

\subsection{Preliminaries of $P$}\label{subsec:P}

In this subsection we give some basic properties of the degree $5$ rational map
\begin{equation}
P(z)=\frac{z(1+\frac{z}{\mu})^4}{(1+z)^2},
\end{equation}
where\footnote{In this paper, the approximate value of a real number $x$ is indicated as $x\doteqdot{0.1234}\ldots$, which means that $x\in[0.1234,0.1235]$.} $\mu=11-4\sqrt{6}=(2\sqrt{2}-\sqrt{3})^2\,\Aeq{1.2020}$.

\begin{prop}[see Figure \ref{Fig-P:chessboard}]\label{prop:P}
The map $P$ has the following properties:
\begin{enumerate}
\item $P$ has exactly $4$ critical points $-\mu$, $-1$, $\infty$, $cp_P$ (whose local degrees are $4$, $2$, $3$ and $3$ respectively) and $3$ critical values $0=P(-\mu)$, $\infty=P(-1)=P(\infty)$, $cv_P=P(cp_P)$, where
\begin{equation}
cp_P=-\sqrt{\tfrac{\mu}{3}}=1-\tfrac{2}{3}\sqrt{6}\Aeq{-0.6329} \text{ and } cv_P=-\tfrac{16}{3(8\sqrt{6}+3)}\Aeq{-0.2360};
\end{equation}
\item $P^{-1}(cv_P)=\{cp_P,\nu_1^P,\nu_2^P\}$, where
\renewcommand\theequation{\thesection.\arabic{equation}*}
\begin{align}
\nu_1^P&=9\sqrt{6}-\tfrac{47}{2}-\tfrac{1}{2}\sqrt{996\sqrt{6}-2439}\,\Aeq{-1.8704},  \text{ and}\\
\nu_2^P&=9\sqrt{6}-\tfrac{47}{2}+\tfrac{1}{2}\sqrt{996\sqrt{6}-2439}\,\Aeq{-1.0387};
\end{align}
\item $-2<\nu_1^P<-\mu<\nu_2^P<-1<cp_P<cv_P<0$;
\item $P$ has a super-attracting fixed point $\infty$ and a $1$-parabolic fixed point $0$;
\item $P$ has exactly two periodic Fatou components $B_\infty$ and $B_0$, which contain the critical points $\infty$ and $cp_P$ respectively;
\item $P^{-1}(B_\infty)$ consists of two connected components and one of them is a bounded Fatou component containing the pole $-1$;  and $P^{-1}(B_0)$ consists of three connected components: one is $B_0$ and the rest two are Fatou components containing $\nu_1^P$ and $\nu_2^P$ respectively.
\end{enumerate}
\end{prop}

\begin{figure}[!htpb]
  \setlength{\unitlength}{1mm}
  \centering
  \includegraphics[width=0.97\textwidth]{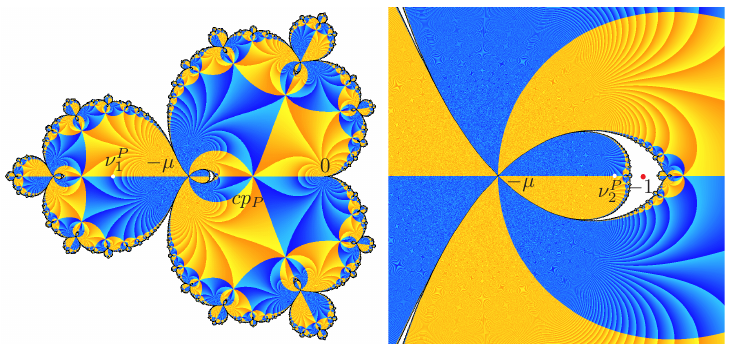}
  \caption{Left: The immediate parabolic basin of $P$ contains the critical point $cp_P$ of local degree $3$, where the chessboard structure and some special points are marked. Right: A zoom of the left near the critical point $-\mu$ and the pole $-1$.}
  \label{Fig-P:chessboard}
\end{figure}

\begin{proof}
(a) The statement follows from the following calculation:
\begin{equation}
P'(z)=\frac{3}{\mu}\left(\frac{1+\tfrac{z}{\mu}}{1+z}\right)^3\left(z+\sqrt{\tfrac{\mu}{3}}\,\right)^2.
\end{equation}

(b) Consider the equation $P(z)=P(cp_P)=cv_P$ of degree $5$ and note that such equation has the $3$ same roots $cp_P$ (counting the multiplicity). The rest two roots can be obtained by comparing the coefficients of the terms of orders $3$ and $4$.

\medskip
(c) and (d) are immediate consequences of (a) and (b).

\medskip
(e) Note that every immediate attracting or parabolic basin of a rational map with degree at least two must contain at least one critical point. Since $P(-\mu)=0$ and $P(-1)=\infty$ are fixed, it follows that the critical point $cp_P$ is contained in the immediate parabolic basin of $0$. Note that $B_\infty$ contains exactly one critical value at $\infty$. Thus $B_\infty$ is simply connected. Since the Julia set $J(P)$ is symmetric about the real axis, $-\mu<-1<0$ and $-\mu$, $0$ are contained in $J(P)$, we conclude that $-1\not\in B_\infty$. Therefore, each of $B_0$ and $B_\infty$ contains exactly one critical point (without counting multiplicity). Hence $B_0$ is also simply connected. This implies that all the Fatou components of $P$ are simply connected and $J(P)$ is connected.

\vskip0.1cm
(f) Let $B_1$ be the Fatou component containing $-1$. Then $B_1$ is the bounded connected component of $P^{-1}(B_\infty)$ which is different from $B_\infty$. Since the degree of the restriction of $P$ on $B_0$ is $3$, by (c) there exist other two Fatou components containing $\nu_1^P$ and $\nu_2^P$ respectively, where $P(\nu_1^P)=P(\nu_2^P)=cv_P\in B_0$.
\end{proof}

\section{Outline of the proof of the Main Theorem}\label{sec-definitions}

For the proof of the Main Theorem, the main goal is to find $\psi$ such that
\begin{equation}
\MMR_0 f=\Psi_0\circ E_f\circ\Psi_0^{-1}=P\circ\psi^{-1},
\end{equation}
where $\Psi_0(z)=cv_P\,\Exp (z)$ is the modified exponential map. We will adapt the main strategy in \cite[\S 5.A]{IS08} to find $\psi$ and work it in details in the following sections. In this section we give an outline of the proof of the Main Theorem.


\begin{defi}[{Some intermediate mappings and the mapping $Q$}]
Define
\begin{equation}\label{equ-Q-defi}
Q(\zeta)=\zeta\,\frac{(1+\frac{1}{\zeta})^6(1-\frac{1}{\zeta})^4}{\big(1+(2-\frac{4}{\mu})\frac{1}{\zeta}+\frac{1}{\zeta^2}\big)^4},
\quad \psi_1(\zeta)=-\frac{4\zeta}{(1+\zeta)^2},
\quad \psi_0(\zeta)=-\frac{4}{\zeta},
\end{equation}
where $\mu=11-4\sqrt{6}$. One can obtain that (see \S\ref{sec-P-to-Q}) $Q$ is related to $P$ by
\begin{equation}\label{equ:Q-relate-P}
Q=\psi_0^{-1}\circ P\circ\psi_1.
\end{equation}
The map $Q$ has a $1$-parabolic fixed point at $\infty$. Note that $\psi_1:\EC\setminus\overline{\D}\to\C\setminus(-\infty,-1]$ is conformal. We are mainly interested in the properties of $Q$ in $\EC\setminus\overline{\D}$. The motivation of introducing $Q$ is to use the dynamics of $Q$ to study $P$.
\end{defi}

\begin{defi}[{$V'=U_\eta^P$ and $U_\eta^Q$, see Figure \ref{Fig-V}}]
Let $\eta>0$ and $cv_P=-\tfrac{16}{3(8\sqrt{6}+3)}$ $\Aeq{-0.2360}$ (which is a critical value of $P$). Define
\begin{equation}\label{equ:U-eta-P}
V'=U_\eta^P:=P^{-1}\big(\D(0,|cv_P|e^{2\pi\eta})\big) \setminus \big((-\infty,-1]\cup B\big),
\end{equation}
where $B$ is the connected component of $P^{-1}\big(\OD(0,|cv_P|e^{-2\pi\eta})\big)$ containing $-\mu$.
Note that $P^{-1}\big(\D(0,|cv_P|e^{2\pi\eta})\big)$ is an annulus since $\EC\setminus \overline{\D}(0,|cv_P|e^{2\pi\eta})$ contains exactly one critical value at $\infty$ (see Proposition \ref{prop:P}(a)).
Let $U_\eta^Q:=\psi_1^{-1}(U_\eta^P)\subset\EC\setminus\overline{\D}$.
\end{defi}

\begin{defi}[{Disk $\mathscr{D}$ and domain $V$}]
Let $a_0=-0.06$ and $r_0=1.07$. We define
\begin{equation}\label{equ:disk}
\mathscr{D}:=\D(a_0,r_0).
\end{equation}
Note that $\OD\subset \mathscr{D}$. We define $V:=\psi_1\big(\EC\setminus\overline{\mathscr{D}}\big)$.
\end{defi}

\begin{prop}[{Relation between $V$ and $V'$}]\label{prop-V-V'}
Let $\eta=3$. Then $\EC\setminus \mathscr{D}\subset U_\eta^Q\subset\EC\setminus\OD$. Therefore, we have
\begin{equation}
\overline{V}\subset V'=U_\eta^P\subset\C\setminus(-\infty,-1].
\end{equation}
\end{prop}

The proof of this proposition will be given in \S\ref{sec-esti-Q-I}. The constant $\eta=3$ and the disks $\mathscr{D}$ will be used throughout the paper.

\begin{figure}[!htpb]
  \setlength{\unitlength}{1mm}
  \centering
  \includegraphics[width=0.95\textwidth]{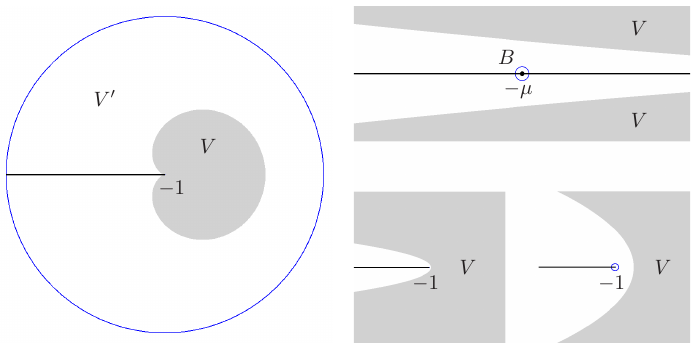}
  \caption{The domains $V$ (the gray part), $V'=U_\eta^P$ with $\eta=3$ and their zooms near $-\mu=4\sqrt{6}-11\Aeq{-1.2020}$ and $-1$. The outer boundary of $V'$ looks like a circle. The half widths of these pictures are $450$, $0.08$, $10^{-3}$ and $10^{-4}$ respectively.}
  \label{Fig-V}
\end{figure}

\begin{defi}[{Classes $\MF_2^P$ and $\MF_1^Q$}]
From now on, we use $\MF_1^P$ to denote $\MF_1$. We define two more classes of maps:
\begin{equation}
\begin{split}
\MF_2^P:=  &~ \big\{ f=P\circ\varphi^{-1}\,\big|\,\varphi:V'\to\C \text{~is normalized and univalent}\big\}, ~ \text{and}\\
\MF_1^Q:= &~ \big\{ F=Q\circ\varphi^{-1}\,\big|\,\varphi:\EC\setminus\overline{\mathscr{D}}\to\EC\setminus\{0\} \text{ is normalized and univalent}\big\}.
\end{split}
\end{equation}
Here the map in $\MF_2^P$ is normalized if $\varphi(0)=0$ and $\varphi'(0)=1$, and the map in $\MF_1^Q$ is normalized if $\varphi(\infty)=\infty$ and $\underset{{\zeta\to\infty}}{\lim}\frac{\varphi(\zeta)}{\zeta}=1$.
\end{defi}

\begin{prop}[{Relation between $\MF_1^P$, $\MF_2^P$ and $\MF_1^Q$}]\label{prop-relation}
We have $\MF_2^P\subset \MF_1^P \simeq \MF_1^Q$.
It is formulated more precisely as follows:
\begin{enumerate}
\item There exists a natural injection $\MF_2^P\hookrightarrow\MF_1^P$, defined by the restriction of $\varphi$ to $V$ for $f=P\circ\varphi^{-1}\in\MF_2^P$;
\item There exists a one-to-one correspondence between $\MF_1^P$ and $\MF_1^Q$, defined by
\begin{equation}
\begin{split}
\MF_1^P\ni f=P\circ\varphi^{-1}\mapsto F=&~\psi_0^{-1}\circ f\circ\psi_0=\psi_0^{-1}\circ P\circ\psi_1\circ\psi_1^{-1}\circ\varphi^{-1}\circ\psi_0\\
     =&~Q\circ\widehat{\varphi}^{-1}\in\MF_1^Q
\end{split}
\end{equation}
with the correspondence $\varphi\mapsto\widehat{\varphi}=\psi_0^{-1}\circ\varphi\circ\psi_1$;
\item $2.5\leq |f''(0)|\leq 4.7$ for all $f\in\MF_1^P$.
\end{enumerate}
\end{prop}

Part (c) implies that $f''(0)$ is uniformly bounded and uniformly away from zero for all $f\in\MF_1^P$. The proof of Proposition \ref{prop-relation} will be given in \S\ref{sec-P-to-Q}.

\medskip
In the following, suppose $F=Q\circ\varphi^{-1}\in\MF_1^Q$, where $\varphi:\EC\setminus\overline{\mathscr{D}}\to\EC\setminus\{0\}$ is a normalized univalent map.
In order to prove the Main Theorem, it is sufficient to prove that for any $F\in\MF_1^Q$, the parabolic renormalization $\MMR_0 F$ belongs to $\MF_2^P$.

As in \cite{IS08}, we now define a Riemann surface $X$ such that one can lift $F^{-1}=\varphi\circ Q^{-1}$ to a single-valued branch on $X$.

\begin{defi}[{Riemann surface $X$}]
Let $cv:=cv_Q=\frac{3(8\sqrt{6}+3)}{4}\Aeq{16.94}$ (which is a critical value of $Q$), $R=125$ and $\rho=0.03$. We define four ``sheets" by
\begin{equation}
\begin{split}
X_{1\pm}:=&~\{z\in\C: \pm\,\im z\geq 0,\,|z|>\rho \text{~and~} \tfrac{\pi}{6}<\pm\arg(z-cv)\leq\pi\}, \text{ and}\\
X_{2\pm}:=&~\{z\in\C: z\not\in\R_-,\,\pm\,\im z\geq 0,\,\rho<|z|<R \text{~and~} \tfrac{\pi}{6}<\pm\arg(z-cv)\leq\pi\}.
\end{split}
\end{equation}
These ``sheets" are considered as in different copies of $\C$ and we use $\pi_{i\pm}:X_{i\pm}\to\C$ ($i=1,2$) to denote the natural projection. The Riemann surface $X$ is constructed as follows: $X_{1+}$ and $X_{1-}$ are glued along negative real axis, $X_{1+}$ and $X_{2-}$ are glued along positive real axis and $X_{1-}$ and $X_{2+}$ are also glued along positive real axis. The projection $\pi_X:X\to\C$ is defined as $\pi_X=\pi_{i\pm}$ on $X_{i\pm}$ and the complex structure of $X$ is given by the projection. See Figure \ref{Fig-Riemann-X}.
\end{defi}

\begin{figure}[!htpb]
  \setlength{\unitlength}{1mm}
  \centering
  \includegraphics[width=0.98\textwidth]{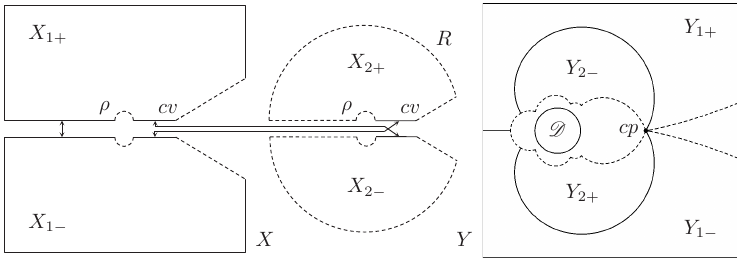}
  \caption{Riemann surface $X$ (left) and the domain $Y$ (right).}
  \label{Fig-Riemann-X}
\end{figure}

\begin{prop}[{Lifts of $Q$ and $\varphi$ to $X$}]\label{prop-lift}
There exists an open subset $Y\subset\C\setminus (\overline{\mathscr{D}}\cup\R_+)$ and a constant $b_0>7$ with the following properties:
\begin{enumerate}
\item There exists an isomorphism $\widetilde{Q}:Y\to X$ such that $\pi_X\circ\widetilde{Q}=Q$ on $Y$ and $\widetilde{Q}^{-1}(z)=\pi_X(z)-b_0+o(1)$ as $z\in X$ and $\pi_X(z)\to\infty$;
\item The normalized univalent map $\varphi$ restricted to $Y$ can be lifted to a univalent holomorphic map $\widetilde{\varphi}:Y\to X$ such that $\pi_X\circ\widetilde{\varphi}=\varphi$ on $Y$.
\end{enumerate}
\end{prop}

The constant $b_0=\tfrac{2(13+32\sqrt{6})}{25}\Aeq{7.31}$ appears in the expansion of $Q$ (see Lemma \ref{lema-Q}). Proposition \ref{prop-lift} will be proved in \S\ref{sec-lift-Q}.

\begin{defi}
Define $g:=\widetilde{\varphi}\circ \widetilde{Q}^{-1}:X\to X$. We have the commutative diagram in Figure \ref{Fig-commu-diagram}.
\begin{figure}[!htpb]
  \setlength{\unitlength}{1mm}
  \centering
  \includegraphics[width=0.5\textwidth]{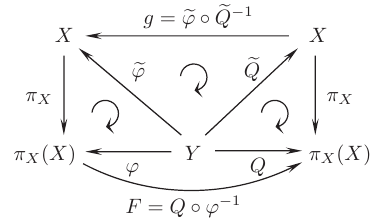}
  \caption{The commutative diagram about the maps $Q$, $\varphi$ and their lifts $\widetilde{Q}$, $\widetilde{\varphi}$.}
  \label{Fig-commu-diagram}
\end{figure}
\end{defi}

\begin{prop}[{Repelling Fatou coordinate on $X$}]\label{prop-rep-Fatou}
The map $g$ satisfies $F\circ\pi_X\circ g=\pi_X$. There exists an injective holomorphic mapping $\widetilde{\Phi}_{rep}:X\to\C$ such that $\widetilde{\Phi}_{rep}(g(z))=\widetilde{\Phi}_{rep}(z)-1$. Moreover, in $\{z\in\C:\re z<-2R\}$, $\widetilde{\Phi}_{rep}\circ\pi_X^{-1}$ is a repelling Fatou coordinate of $F=Q\circ\varphi^{-1}$.
\end{prop}

Proposition \ref{prop-rep-Fatou} will be proved in \S\ref{sec-rep}.

\begin{defi}
For $z_0\in\C$ and $\theta>0$, we denote the sector
\begin{equation}\label{equ:V-z0-theta}
\V(z_0,\theta):=\{z\in\C:\,z\neq z_0 \text{ and } |\arg(z-z_0)|<\theta\}
\end{equation}
and use $\overline{\V}(z_0,\theta)$ to denote its closure. Define
\begin{equation}
W_1:=\V(cv,\tfrac{3\pi}{4})\setminus\overline{\V}(F(cv),\tfrac{\pi}{4}).
\end{equation}
\end{defi}

We will prove in Lemma \ref{lema-F-cv}(b) that $\re F(cv)>22>cv\Aeq{16.94}$ and hence $W_1$ is connected. Let $R_1=108$ and recall that $\eta=3$.

\begin{prop}[{Attracting Fatou coordinate and the shape of $D_1$}]\label{prop-Phi}
We have
\begin{enumerate}
\item $F$ maps $\V(13,\frac{3\pi}{4})$ into itself and $\V(13,\frac{3\pi}{4})$ is contained in the immediate parabolic basin of $\infty$. There exists an attracting Fatou coordinate $\Phia:\V(13,\frac{3\pi}{4})\to\C$ such that $\Phia(F(z))=\Phia(z)+1$ and $\Phia(cv)=1$. Moreover, $\Phia(\V(13,\frac{3\pi}{4}))$ contains $\{w\in\C:\re w>1\}$;
\item There are domains $D_1$, $D_1^\pm$ in $W_1 \subset\V(13,\frac{3\pi}{4})$ such that
\begin{equation}\label{equ:PhiaD}
\begin{split}
\Phia(D_1)=&~\{w\in\C:1<\re w<2,\,-\eta<\im w<\eta\}, \text{~and}\\
\Phia(D_1^\pm)=&~\{w\in\C:1<\re w<2,\,\pm\,\im w>\eta\},
\end{split}
\end{equation}
where $D_1\subset\D(cv,R_1)$ and $D_1^\pm\subset\{z\in\C:\tfrac{\pi}{6}<\pm\,\arg(z-cv)<\tfrac{3\pi}{4}\}$.
\end{enumerate}
\end{prop}

This proposition will be proved in \S\ref{sec-Phi-attr}. The number $\eta$ in this proposition can be replaced by any number between $3$ and $8$ while using the same $R_1=108$ (see \eqref{eta-geq} and \eqref{eta-leq}).
Part (a) implies that $cv$ is contained in the immediate basin of $\infty$. 
After adding a constant, we normalize $\widetilde{\Phi}_{rep}$ such that $\widetilde{\Phi}_{rep}(z)-\Phi_{attr}(\pi_X(z))\to 0$ when $z\in X$, $\pi_X(z)\in D_1^+$ and $\im \pi_X(z)\to+\infty$.

\begin{rmk}
It turns out (by computer experiments) that there is no $\theta\in(0,\pi)$ such that $\Phia$ is injective in $\V(cv,\theta)$ and such that $\Phia(\V(cv,\theta))$ contains $\{w\in\C:\re w>1\}$. To overcome this difficulty, we consider $F(cv)$ and prove that $\Phia$ is injective in $\V(F(cv),\frac{13\pi}{20})$ and that $\Phia(\V(F(cv),\frac{13\pi}{20}))$ contains $\{w\in\C:\re w>2\}$. This is one of the main differences between our arguments and Inou-Shishikura's.
\end{rmk}

\begin{prop}[{Domains around the critical point}]\label{prop-F}
There exist disjoint Jordan domains $D_0$, $D_0'$, $D_0''$, $D_{-1}$, $D_{-1}'''$, $D_{-1}''''$ and a domain $D_0^+$ such that
\begin{enumerate}
\item  $\overline{D}_0$, $\overline{D}_0'$, $\overline{D}_0''$, $\overline{D}_{-1}$, $\overline{D}_{-1}'''$, $\overline{D}_{-1}''''$, $\overline{D}_0^+$ are contained in $Dom(F)$ $=Image(\varphi)$;
\item $F(D_0)=F(D_0')=F(D_0'')=D_1$, $F(D_{-1})=F(D_{-1}''')=F(D_{-1}'''')=D_0$ and $F(D_0^+)=D_1^+$;
\item $F$ is injective in each of these domains;
\item $cp_F=\varphi(cp_Q)\in \overline{D}_0\cap \overline{D}_0'\cap \overline{D}_0''\cap \overline{D}_{-1}\cap \overline{D}_{-1}'''\cap \overline{D}_{-1}''''$, $\overline{D}_0\cap \overline{D}_1\neq\emptyset$, $\overline{D}_0^+\cap \overline{D}_1^+\neq\emptyset$ and $\overline{D}_{-1}\cap \overline{D}_0^+\neq\emptyset$;
\item $\overline{D}_0\cup \overline{D}_0'\cup \overline{D}_0''\cup \overline{D}_{-1}\cup \overline{D}_{-1}'''\cup \overline{D}_{-1}''''\setminus\{cv\}
    \subset\pi_X(X_{2+})\cup\pi_X(X_{2-})=\D(0,R)\setminus(\overline{\D}(0,\rho)\cup\R_-\cup\OV(cv,\frac{\pi}{6}))$ and $\overline{D}_0^+\subset\pi_X(X_{1+})$.
\end{enumerate}
\end{prop}

This proposition will be proved in \S\ref{sec-bounding-dom}. The main work there is to prove $(\overline{D}_0\cup \overline{D}_0'\cup \overline{D}_0''\cup \overline{D}_{-1}\cup \overline{D}_{-1}'''\cup \overline{D}_{-1}'''')\cap\R_-=\emptyset$. See Figure \ref{Fig-chessboard} for the shape of these domains in the case of $\varphi=id$.

\begin{figure}[!htpb]
  \setlength{\unitlength}{1mm}
  \centering
  \includegraphics[width=0.98\textwidth]{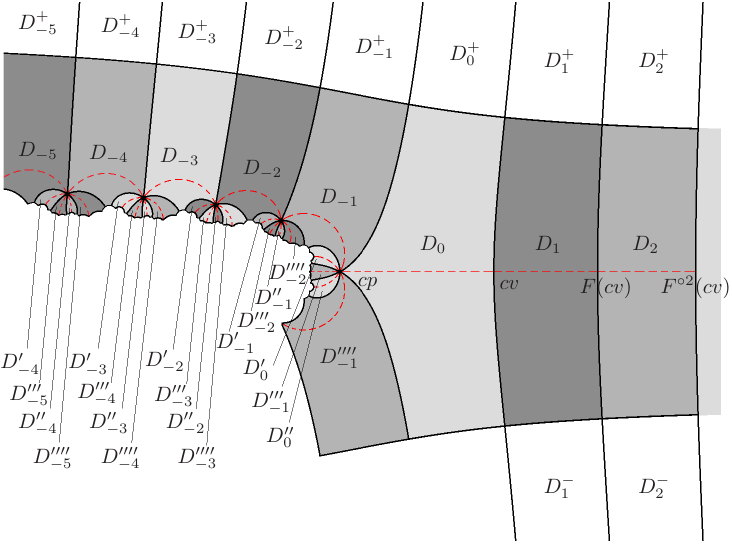}
  \caption{Various domains for $F=Q$ (i.e., $\varphi=id$). The images of $D_1$ and $D_1^+$ under $g$ (i.e., the inverse of $F$) are denoted by $D_{-n}=g^n(D_0)$, $D_{-n}'=g^n(D_0')$, $D_{-n}''=g^n(D_0'')$, $D_{-n}'''=g^{n-1}(D_{-1}''')$, $D_{-n}''''=g^{n-1}(D_{-1}'''')$, $D_{-n}^+=g^n(D_0^+)$ and their projection by $\pi_X$ are drawn. To emphasize details, we choose $\eta=1.5$ in this figure.}
  \label{Fig-chessboard}
\end{figure}

\begin{prop}[{Relating $E_F$ to $P$}]\label{prop-E_F-P}
The parabolic renormalization $\MMR_0 F$ belongs to the class $\MF_2^P$ (possibly after a linear conjugacy). In fact, suppose we regard $D_0$, $D_0'$, $D_0''$, $D_{-1}'''$, $D_{-1}''''$, $D_0^+$ as subsets of $X_{1+}\cup X_{2-}\subset X$ and let
\begin{equation}
U:=\text{the interior of }\bigcup_{n=0}^\infty g^n\big(\overline{D}_0\cup \overline{D}_0'\cup \overline{D}_0''\cup \overline{D}_{-1}'''\cup \overline{D}_{-1}''''\cup \overline{D}_0^+\big).
\end{equation}
Then there exists a surjective holomorphic mapping $\Psi_1: U\to U_\eta^P\setminus\{0\}=V'\setminus\{0\}$ such that
\begin{enumerate}
\item $P\circ\Psi_1=\Psi_0\circ\widetilde{\Phi}_{attr}$ on $U$, where $\Psi_0(z)=cv_P\,\Exp(z)=cv_P \,e^{2\pi\ii z}:\C\to\C^*$, and $\widetilde{\Phi}_{attr}:U\to\C$ is the natural extension of the attracting Fatou coordinate to $U$;
\item $\Psi_1(z)=\Psi_1(z')$ if and only if $z'=g^n(z)$ or $z=g^n(z')$ for some integer $n\geq 0$;
\item $\psi=\Psi_0\circ\widetilde{\Phi}_{rep}\circ\Psi_1^{-1}:V'\setminus\{0\}\to\C^*$ is well-defined and extends to a normalized univalent function on $V'$;
\item on $\psi(V'\setminus\{0\})$, the following holds:
\begin{equation}
P\circ\psi^{-1}=P\circ\Psi_1\circ\widetilde{\Phi}_{rep}^{-1}\circ\Psi_0^{-1}=\Psi_0\circ\widetilde{\Phi}_{attr}\circ\widetilde{\Phi}_{rep}^{-1}\circ\Psi_0^{-1}=\Psi_0\circ E_F\circ\Psi_0^{-1}.
\end{equation}
\end{enumerate}
\end{prop}


This proposition will be proved in \S\ref{sec-para-renorm}. The map $\Psi_1$ is defined by choosing an appropriate branch of $P^{-1}\circ\Psi_0\circ\widetilde{\Phi}_{attr}$ on each domain $D_{-n}=g^n(D_0)$ etc. After proving the consistency of $\Psi_1$ (compare Figures \ref{Fig-chessboard} and \ref{Fig-U-eta-P-log}), we define
\begin{equation}
\MMR_0 F:=P\circ\psi^{-1}\in\MF_2^P \text{\quad for\quad} F=Q\circ\varphi^{-1}\in\MF_1^Q \,(\simeq\MF_1^P).
\end{equation}
Then the Main Theorem holds since $\MF_2^P\hookrightarrow\MF_1^P$ by Proposition \ref{prop-relation}.

\begin{figure}[!htpb]
  \setlength{\unitlength}{1mm}
  \centering
  \includegraphics[width=0.98\textwidth]{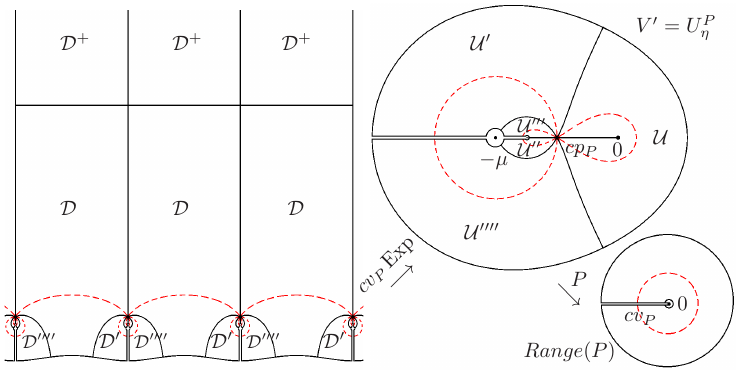}
  \caption{Domain $U_\eta^P$ and its log lift (the inverse image of $cv_P\,\Exp$). Note that these pictures are just topologically correct but not conformally precise. We choose small $\eta$ here such that the details are visible.}
  \label{Fig-U-eta-P-log}
\end{figure}

\section{Some preparations}

In this section we prepare some useful lemmas for calculations. For the proofs from Lemma \ref{lema-basic-esti} to Theorem \ref{thm-est-Fatou}, see \cite[\S 5.B]{IS08}.

\begin{lem}\label{lema-basic-esti}
\begin{enumerate}
\item If $a,b\in\C$ and $|a|>|b|$, then $|\arg(a+b)-\arg a|\leq \arcsin\big(\frac{|b|}{|a|}\big)$. In particular, if $|b|<1$, then $|\arg(1+b)|\leq \arcsin |b|$;
\item If $0\leq x\leq \tfrac{1}{2}$, then $\arcsin(x)\leq \frac{\pi}{3}\,x$.
\end{enumerate}
\end{lem}

\begin{lem}\label{lema-esti-re}
\begin{enumerate}
\item  If $\re (ze^{-\ii\theta})>t>0$ with $\theta\in\R$, then
\begin{equation}
\frac{1}{z}\in\D\left(\frac{e^{-\ii\theta}}{2t},\frac{1}{2t}\right);
\end{equation}

\item  If $H=\{z\in\C:\re (ze^{-\ii\theta})>t\}$ and $z_0\in H$ with $u=\re (z_0e^{-\ii\theta})-t$, then
\begin{equation}
\D_H(z_0,s(r))=\D\left(z_0+\frac{2ur^2e^{\ii\theta}}{1-r^2},\frac{2ur}{1-r^2}\right),
\end{equation}
where $s(r)=d_{\D}(0,r)=\log \frac{1+r}{1-r}$.
\end{enumerate}
\end{lem}

The following theorem was proved by applying the properties of cross-ratios and Golusin inequalities.

\begin{thm}[{A general estimate on Fatou coordinate}]\label{thm-est-Fatou}
Let $\Omega$ be a disk or a half plane and $f:\Omega\to\C$ a holomorphic function with $f(z)\neq z$. Suppose $f$ has a univalent Fatou coordinate $\Phi:\Omega\to\C$, i.e., $\Phi(f(z))=\Phi(z)+1$ when $z,f(z)\in\Omega$. If $z\in\Omega$ and $f(z)\in\Omega$, then
\begin{equation}
\left|\log \Phi'(z)+\log (f(z)-z)-\frac{1}{2}\log f'(z)\right|\leq \frac{1}{2}\log\frac{1}{1-r^2},
\end{equation}
where $r$ is a real number such that $0< r<1$ and $d_{\D}(0,r)=d_{\Omega}(z,f(z))$.
\end{thm}

\noindent \textbf{Computer checked inequalities.}
In this paper, the inequalities checked by computer are marked by star ``\,*\," in the equation numbers. These inequalities are all evaluated at explicit values. Recall that the approximate value is indicated as $x\doteqdot{0.1234}\ldots$, which means that $x\in[0.1234,0.1235]$.

\medskip

\noindent \textbf{List of constants.} We collect some constants that will be used throughout this paper:
\begin{itemize}
  \item The disk $\mathscr{D}$ related: $a_0=-0.06$, $r_0=1.07$, $r_1=1.2$;
  \item The map $Q$: $cv=cv_Q=\frac{3(8\sqrt{6}+3)}{4}\Aeq{16.94}$, $b_0=\tfrac{2(13+32\sqrt{6})}{25}\Aeq{7.31}$, $b_1=\frac{2029+256\sqrt{6}}{125}\Aeq{21.24}$;
  \item The univalent map $\varphi$: $c_{00}=-a_0=0.06$, $c_{01,max}=2r_0=2.14$;
  \item Attracting Fatou coordinates: $\theta_1=\frac{3\pi}{20}$, $u_{1,\theta_1}=8.5$, $u_{2,\theta_1}=6.1$, $u_3=22\cos\theta_1\Aeq{19.60}$, $u_4=17.3$; $\theta_2=\frac{\pi}{4}$, $u_{1,\theta_2}=9$, $u_{2,\theta_2}=6.6$;
  \item Riemann surface $X$: $R=125$, $\rho=0.03$;
  \item Some complex constants: $a_4=-0.22+0.69\ii$, $a_5=0.78+0.21\ii$, $\omega=\tfrac{8\sqrt{6}-3}{25}+\tfrac{6\sqrt{6}+4}{25}\ii\,(\doteqdot 0.6638\ldots+0.7478\ldots\ii)$;
  \item Other constants: $\eta=3$, $R_1=108$, $R_2=99$, $\varepsilon_1=0.128$, $\varepsilon_2=0.007$, $\varepsilon_3=0.014$, $\varepsilon_4=|a_4-\omega|$, $\varepsilon_5=|a_5-\omega|$, $\varepsilon_6=0.41$, $\varepsilon_7=0.82$.
\end{itemize}

\section{Covering property of $f\in\MF_0$ and subcover}\label{sec-subcover}

\begin{defi}[{Class $\MF_0$}]
We define a class of holomorphic maps:
\begin{equation}
\mathcal{F}_0:=
\left\{
\begin{array}{l}
f:Dom(f)\to\C\\
\text{\quad is holomorphic}
\end{array}
\left|
\begin{array}{l}
0\in Dom(f) \text{~open}\subset\C,~f(0)=0,~f'(0)=1,\\
f:Dom(f)\setminus\{0\}\to\C^* \text{ is a branched}\\
\text{covering with a unique critical value } cv_f, \\
\text{all critical points are of local degree 3}
\end{array}
\right.
\right\}.
\end{equation}
The cubic polynomial $p(z)=z+\sqrt{3}z^2+z^3$ (which is conformally conjugate to the cubic unicritical polynomial $\widetilde{p}(z)=z^3+2/(3\sqrt{3})$) and the degree $5$ rational map $P$ belong to $\MF_0$. Indeed, one can consider the restriction of the domains of definitions and set $Dom(p)=\C\setminus\{\frac{-\sqrt{3}+\ii}{2},\frac{-\sqrt{3}-\ii}{2}\}$ and $Dom(P)=\C\setminus\{-\mu,-1\}$.
\end{defi}

Identifying holomorphic maps as the projections of some Riemann surfaces onto a plane domain is a traditional way to understand the covering structures of holomorphic functions.
We will describe the covering properties of the maps in $\MF_0$. The idea is that all the maps can be regarded as a (partial) branched covering over the ranges and they have similar covering structures when the maps restrict on certain ``sheets".

\medskip
Let $f\in\MF_0$ and suppose that the critical value $cv=cv_f$ of $f$ is contained in $\R_-=(-\infty,0)$. We denote $\Gamma_a=(cv,0)$, $\Gamma_b=(-\infty,cv]$, $\Gamma_c=(0,+\infty)\subset\R$. Define $\C_{slit}=\C\setminus(\{0\}\cup\Gamma_b\cup\Gamma_c)$ and recall that $\BH_\pm=\{z\in\C:\pm\,\im z>0\}$ denote the upper and lower half planes respectively.

The description of the covering properties of $f\in\MF_0$ in the following is inspired by \cite[\S5.C]{IS08}. Since $\C_{slit}$ is simply connected and contains no critical values, the preimage $f^{-1}(\C_{slit})$ consists of several connected components $\MU_i$, where $i\in I$ and $I$ is an index set which is equal to $\N$ or a finite set $\{1,2,\ldots, n\}$. Each of these components is mapped isomorphically onto $\C_{slit}$ by $f$. Denote $\MU_{i\pm}=f^{-1}(\BH^{\pm})\cap\MU_i$, $\gamma_{ai}=f^{-1}(\Gamma_a)\cap \MU_i$, $\gamma_{bi\pm}=f^{-1}(\Gamma_b)\cap \overline{\MU}_{i\pm}$ and $\gamma_{ci\pm}=f^{-1}(\Gamma_c)\cap \overline{\MU}_{i\pm}$, where the closures are taken in $Dom(f)$. See Figure \ref{Fig-dom-P}.

\begin{figure}[!htpb]
  \setlength{\unitlength}{1mm}
  \centering
  \includegraphics[width=0.98\textwidth]{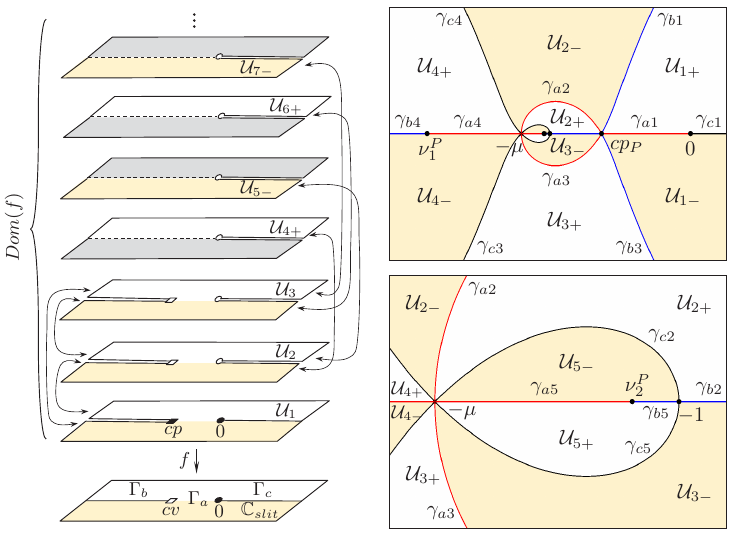}
  \caption{Left: The domain $Dom(f)$ as a Riemann surface spread over $\C$. Right: The domain $Dom(P)$ and its zoom near $[-\mu,-1]$. Some special points are marked (The superscript ``$P$" in the notations are omitted). As preimages of $cv_P$, the points $\nu_1^P$ and $\nu_2^P$ are introduced in Proposition \ref{prop:P}(b).}
  \label{Fig-dom-P}
\end{figure}

The domain $Dom(f)$ can be described as the union of $\overline{\MU}_i$'s which are glued along $\gamma_{bi\pm}$ and $\gamma_{ci\pm}$. Each $\gamma_{bi+}$ is glued with some $\gamma_{bj-}$ and vice versa. This is also true for $\gamma_{ci\pm}$. If $\gamma_{bi+}$ is glued with $\gamma_{bj-}$ and $\gamma_{bj+}$ is glued with $\gamma_{bk-}$, then $\gamma_{bk+}$ is glued with $\gamma_{bi-}$ since the critical points are of local degree $3$. Because $f$ is univalent near $0$, there must exists a component, say $\MU_1$ such that $0\in\partial\MU_1$ and $\gamma_{c1+}=\gamma_{c1-}$.

Note that $f^{-1}(cv)\cap\gamma_{b1+}\cap\gamma_{b1-}$ is a critical point of local degree $3$. We call this point the \emph{closest critical point}. 
There exist other two components $\MU_2$ and $\MU_3$ such that $\gamma_{b1+}=\gamma_{b2-}$, $\gamma_{b2+}=\gamma_{b3-}$ and $\gamma_{b3+}=\gamma_{b1-}$.
Denote $\MU_{123}:=\MU_1\cup\,\MU_2\cup\,\MU_3\cup\gamma_{b1+}\cup\gamma_{b2+}\cup\gamma_{b3+}$. Then
\begin{equation}
f|_{\MU_{123}}:\MU_{123}\to\C_{slit}\cup\Gamma_b=\C\setminus(\{0\}\cup \Gamma_c)
\end{equation} is a branched covering of degree $3$ branched over $cv$.

We already have three components $\MU_1$, $\MU_2$ and $\MU_3$. Now we consider $\gamma_{c2+}$, $\gamma_{c2-}$, $\gamma_{c3+}$ and $\gamma_{c3-}$. If $\gamma_{c2+}$ and $\gamma_{c2-}$, $\gamma_{c3+}$ and $\gamma_{c3-}$ are glued together, or $\gamma_{c2+}$ and $\gamma_{c3-}$, $\gamma_{c2-}$ and $\gamma_{c3+}$ are glued together, by removable singularity theorem, we will obtain a branched covering defined from $\EC$ to itself and this implies that $f$ is a cubic rational map.
If $f$ is neither a cubic nor quartic rational map, we obtain the following result.

\begin{prop}[Subcover like $P$]\label{prop:subcover}
Let $f\in\MF_0$. Suppose $cv_f=cv_P$ and any two of $\gamma_{c2+}$, $\gamma_{c2-}$, $\gamma_{c3+}$ and $\gamma_{c3-}$ are not glued together. Then there exists an open subset $U$ of $Dom(f)$ and a conformal mapping
$\varphi:\C\setminus(-\infty,-1]\to U$ such that $\varphi(0)=0$, $\varphi'(0)=1$ and $f=P\circ\varphi^{-1}$ on $U$.
\end{prop}

\begin{proof}
By the assumption, there exist $4$ components $\MU_4$, $\MU_5$, $\MU_6$ and $\MU_7$ such that $\gamma_{c2-}=\gamma_{c4+}$, $\gamma_{c2+}=\gamma_{c5-}$, $\gamma_{c3-}=\gamma_{c6+}$ and $\gamma_{c3+}=\gamma_{c7-}$. The further gluings for $\gamma_{c4-}$, $\gamma_{c5+}$, $\gamma_{b4\pm}$ and $\gamma_{b5\pm}$ etc., depend on particular $f$. See left picture in Figure \ref{Fig-dom-P}. Although the components $\MU_4$, $\MU_5$, $\MU_6$ and $\MU_7$ for $f$ may or may not be distinct, $f$ and $P$ have the same gluing relation up to the half components $\MU_{4+}$, $\MU_{5-}$, $\MU_{6+}$ and $\MU_{7-}$.

Denote the components and curves for $P$ by $\MU_i^P$ and $\gamma_{ai}^P$ etc., as in the right picture of Figure \ref{Fig-dom-P} (we denote $\gamma_{bi}^P=\gamma_{bi+}^P$ and $\gamma_{ci}^P=\gamma_{ci+}^P$ for simplicity). Now we define
\begin{equation}
\varphi:\C\setminus(-\infty,-1]=\C\setminus(\gamma_{a4}^P\cup\gamma_{b4}^P\cup\gamma_{a5}^P\cup\gamma_{b5}^P)\to Dom(f)
\end{equation}
by $\varphi(z)=(f|_{\MU_{i\pm}})^{-1}\circ P(z)$ on $\MU_{i\pm}^P$ for $i=1,2,3,4,5$, except in $\MU_{4-}^P$ and $\MU_{5+}^P$, where $(f|_{\MU_{7-}})^{-1}\circ P$ and $(f|_{\MU_{6+}})^{-1}\circ P$ are used respectively. Since the gluing relations are the same (if the components $\MU_{4-}^P$ and $\MU_{5+}^P$ are denoted by $\MU_{7-}^P$ and $\MU_{6+}^P$ respectively), the definition of $\varphi$ above can be extended continuously to the boundary curves $\gamma_{b1}^P$, $\gamma_{b2}^P$, $\gamma_{b3}^P$,$\gamma_{c1}^P$, $\gamma_{c2}^P$, $\gamma_{c3}^P$, $\gamma_{c4}^P$ and $\gamma_{c5}^P$. The origin is mapped to itself and $cp_P$ is mapped to the closest critical point of $f$ by $\varphi$. It is easy to see that $\varphi$ is a homeomorphism in $\C\setminus(-\infty,-1]$ and holomorphic except the union of finite number of analytic curves. By the removable singularity theorem, it follows that $\varphi$ is conformal from $\C\setminus(-\infty,-1]$ onto its image. By definition, we have $f=P\circ\varphi^{-1}$, $\varphi(0)=0$ and $\varphi'(0)=1$. The proof is finished if we set $U:=\varphi(\C\setminus(-\infty,-1])\subset Dom(f)$.
\end{proof}

By the same reason of \cite[Theorem 3.2]{IS08} (see also \cite[\S 1]{Che22}), $\MF_0$ is invariant under the parabolic renormalization $\MMR_0$.
Moreover, for any $f\in\MF_0$, its parabolic renormalization $\MMR_0(f)$ can be extended to a holomorphic map which is a branched covering of infinite degree and the local degree at each critical point is $3$.

Note that the cubic unicritical parabolic map $\widetilde{p}(z)=z^3+2/(3\sqrt{3})$ is conformally conjugate to $p(z)=z+\sqrt{3}z^2+z^3$ and $p\in\MF_0$. Hence $\MMR_0(\widetilde{p})$ can be also defined.
By Proposition \ref{prop:subcover}, the restriction of $\MMR_0(\widetilde{p})$ in some domain can be written as $P\circ\varphi^{-1}$, where $\varphi:\C\setminus(-\infty,-1]\to\C$ is a normalized univalent map. Since $V\subset \C\setminus(-\infty,-1]$ by Proposition \ref{prop-V-V'}, it follows that $\MMR_0(\widetilde{p})\in\MF_1$.
On the other hand, every $f\in\MF_1$ is a partial branched covering: some points in the range of $f$ are covered $5$ times while some points are covered less. This implies that $\widetilde{p}\not\in\MF_1$.

Consider the quartic polynomial $q(z)=z+\frac{2+b}{2} z^2+\frac{1+2b}{3} z^3+\frac{b}{4} z^4$, where $b=\frac{1}{6}(2+\sqrt{2}\ii)$ such that $q$ has a double critical point $-1$ and a simple critical point $-\frac{1}{b}$ satisfying $q(-\frac{1}{b})=0$. We have $q\in\MF_0$, $q\not\in\MF_1$ and $\MMR_0(q)\in\MF_1$.

\section{Passing from $P$ to $Q$}\label{sec-P-to-Q}

We introduce the rational map $Q$ in \eqref{equ-Q-defi} because it will be more easier to obtain the estimates of the univalent functions in the complement of a disk.
Recall that $\psi_1(\zeta)=-4\zeta/(1+\zeta)^2$ and $\psi_0(\zeta)=-4/\zeta$. In particular, $\psi_1:\EC\setminus\overline{\D}\to\C\setminus(-\infty,-1]$ is a conformal mapping. In the following, we use
\begin{equation}
\psi_1^{-1}(z)=\tfrac{1+\sqrt{z+1}}{1-\sqrt{z+1}}:\C\setminus(-\infty,-1]\to\EC\setminus\overline{\D}
\end{equation}
to denote the inverse of $\psi_1$, where $\arg\sqrt{\cdot}\in(-\frac{\pi}{2},\frac{\pi}{2}]$.
If $x\leq -1$, we unify the notation and denote $\psi_1^{-1}(x):=\frac{1+\ii\sqrt{-x-1}}{1-\ii\sqrt{-x-1}}=\frac{2+x}{-x}+\frac{2\sqrt{-x-1}}{-x}\ii$.

Recall that $\mu=11-4\sqrt{6}>1$ and note that $\mu-1=(\sqrt{6}-2)^2$. Let $cp_P=1-\tfrac{2}{3}\sqrt{6}\Aeq{-0.6329}$, $cv_P=-\tfrac{16}{3(8\sqrt{6}+3)}\Aeq{-0.2360}$, $\nu_1^P\Aeq{-1.8704}$ and $\nu_2^P\Aeq{-1.0387}$ be the numbers obtained in Proposition \ref{prop:P}.
The following result can be obtained by a direct calculation.

\begin{lem}\label{lema-cp-cv-Q}
\begin{enumerate}
\item The maps $P$ and $Q$ are related by $Q=\psi_0^{-1}\circ P\circ \psi_1$;

\item $Q$ has $6$ critical points $\omega$, $\overline{\omega}$, $1$, $-1$, $cp$ and $cp'$ whose local degrees are $4$, $4$, $4$, $6$, $3$ and $3$ respectively, where
\renewcommand\theequation{\thesection.\arabic{equation}*}
\begin{align}
\omega:=&~\psi_1^{-1}(-\mu)=\tfrac{2-\mu}{\mu}+\tfrac{2\sqrt{\mu-1}}{\mu}\ii\\
=&~\tfrac{8\sqrt{6}-3}{25}+\tfrac{6\sqrt{6}+4}{25}\ii\,(\doteqdot 0.6638\ldots+0.7478\ldots\ii)\in\partial\D, \\
cp:=&~\psi_1^{-1}(cp_P)=\tfrac{1+\sqrt{cp_P+1}}{1-\sqrt{cp_P+1}}=\tfrac{1}{5}\Big(1+4\sqrt{6}+2\sqrt{18+2\sqrt{6}}\Big)\Aeq{4.0737}, \\
cp':=&~\tfrac{1}{cp_P}=\tfrac{1-\sqrt{cp_P+1}}{1+\sqrt{cp_P+1}}=\tfrac{1}{5}\Big(1+4\sqrt{6}-2\sqrt{18+2\sqrt{6}}\Big)\Aeq{0.2454};
\end{align}

\item $Q$ has $3$ critical values: $\infty =Q(\omega)=Q(\overline{\omega})$, $0=Q(1)=Q(-1)$ and
\renewcommand\theequation{\thesection.\arabic{equation}*}
\begin{equation}
\qquad cv=cv_Q:=Q(cp)=Q(cp')=\psi_0^{-1}(cv_P)=\tfrac{3(8\sqrt{6}+3)}{4}\Aeq{16.9469};
\end{equation}

\item $Q^{-1}(cv)=\{cp,cp',\nu_1,\overline{\nu}_1,\nu_2,\overline{\nu}_2\}$, where
\renewcommand\theequation{\thesection.\arabic{equation}*}
\begin{align}
\nu_1&=\psi_1^{-1}(\nu_1^P)=\tfrac{2+\nu_1^P}{-\nu_1^P}+\tfrac{2\sqrt{-\nu_1^P-1}}{-\nu_1^P}\ii \,(\doteqdot 0.0692\ldots+0.9975\ldots\ii)\in\partial\D \text{\quad and}  \\
\nu_2&=\psi_1^{-1}(\nu_2^P)=\tfrac{2+\nu_2^P}{-\nu_2^P}+\tfrac{2\sqrt{-\nu_2^P-1}}{-\nu_2^P}\ii \,(\doteqdot 0.9254\ldots+0.3788\ldots\ii)\in\partial\D.
\end{align}
\end{enumerate}
\end{lem}

Although $Q$ is a rational map defined in the whole $\EC$, in the following we are mainly interested in the restriction of $Q$ in the outside of the unit disk.
The domain of definition of $Q$ can be decomposed as we just did for $P$.
Define $\MU_{i\pm}^Q=\psi_1^{-1}(\MU_{i\pm}^P)$, $\gamma_{ai}^Q=\psi_1^{-1}(\gamma_{ai}^P)$ and $\Gamma_a^Q=\psi_0^{-1}(\Gamma_a^P)$ etc. Then $\Gamma_a^Q=(cv,+\infty)$, $\Gamma_b^Q=(0,cv]$, $\Gamma_c^Q=(-\infty,0)$, $\gamma_{a1}^Q=(cp,+\infty)$, $\gamma_{b2}^Q=(1,cp)$, $\gamma_{c1}^Q=(-\infty,-1)$. See Figure \ref{Fig-dom-Q}.

\begin{figure}[!htpb]
  \setlength{\unitlength}{1mm}
  \centering
  \includegraphics[width=0.98\textwidth]{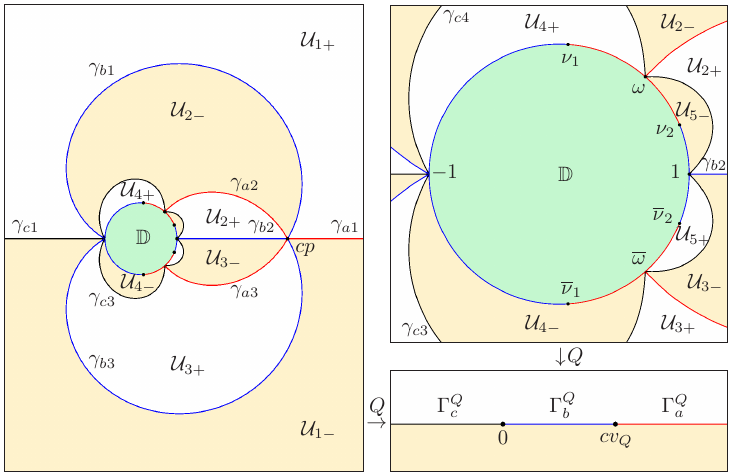}
  \caption{The domain of $Q$ with partition by curves. The top right picture is a zoom of the left picture near the unit disk. Note that the superscript ``$Q$" has been omitted in most of the notations.}
  \label{Fig-dom-Q}
\end{figure}

The map $Q$ maps each $\MU_{i\pm}^Q$ isomorphically onto $\{z\in\C:\pm\,\im z>0\}$ and maps $\gamma_{ai}^Q$ homeomorphically onto $\Gamma_a^Q$ etc. Denote $\MU_{123}^Q:=\MU_1^Q\cup\,\MU_2^Q\cup\,\MU_3^Q\cup\gamma_{b1}^Q\cup\gamma_{b2}^Q\cup\gamma_{b3}^Q=\psi_1^{-1}(\MU_{123}^P)$. Then
\begin{equation}
Q|_{\MU_{123}^Q}:\MU_{123}^Q\to\C_{slit}\cup\Gamma_b^Q=\C\setminus(\{0\}\cup \Gamma_c^Q)
\end{equation}
is a branched covering of degree $3$ branched over $cv_Q$.

\begin{proof}[Proof of Proposition \ref{prop-relation} assuming Proposition \ref{prop-V-V'}]
For (a) and (b), the reader can refer to \cite[\S5.D]{IS08} for a completely similar proof. We only prove the statement on $f''(0)$ in Part (c). Note that one can write $Q(\zeta)=\zeta+b_0+\mathcal{O}(\tfrac{1}{\zeta})$ near $\infty$, where $b_0=\tfrac{2(13+32\sqrt{6})}{25}\Aeq{7.31}$ (see Lemma \ref{lema-Q}). Since $\widehat{\varphi}(\zeta)=\zeta+c_0+\mathcal{O}(\tfrac{1}{\zeta})$ for some $c_0\in\C$ as $\zeta\to \infty$, then $F(z)=Q\circ\widehat{\varphi}^{-1}(z)=z+(b_0-c_0)+\mathcal{O}(\tfrac{1}{z})$ near $\infty$. Therefore,
\begin{equation}
f(z)=\psi_0\circ F\circ\psi_0^{-1}(z)=z+\tfrac{1}{4}(b_0-c_0)z^2+\mathcal{O}(z^3).
\end{equation}
Hence we have $f''(0)=\frac{1}{2}(b_0-c_0)$. By Lemma \ref{lema-esti-varphi}(a), we have $|c_0-c_{00}|\leq c_{01,max}$, i.e., $|\tfrac{c_0}{2}-0.03|\leq 1.07$. Therefore, we have $|f''(0)-(\tfrac{b_0}{2}-0.03)|\leq 1.07$. Note that $\tfrac{b_0}{2}\in[3.65,3.66]$. We have $2.5\leq |f''(0)|\leq 4.7$.
\end{proof}

In the following sections, if there is no confusion, we will drop the superscript ``$Q$" in the notation $\MU_i^Q$, $\gamma_{ai}^Q$ etc., and denote them by $\MU_i$, $\gamma_{ai}$ etc., for simplicity.

\section{Estimates on $Q$: Part I}\label{sec-esti-Q-I}

The numerical estimates will begin from this section. In the following two sections, we give some estimates on the map $Q$.
Let $U_\eta^Q=\psi_1^{-1}(U_\eta^P)=\psi_1^{-1}(V')\subset\EC\setminus\OD$ and $\mathscr{D}=\D(-0.06,1.07)$ be defined in \S\ref{sec-definitions}.

\begin{lem}[see Figure \ref{Fig-E} on the left]\label{lema-D-E}
Let $\eta=3$, $\varepsilon_1=0.128$, $\varepsilon_2=0.007$ and $\varepsilon_3=0.014$. Then
\begin{enumerate}
\item $\EC\setminus (U_\eta^Q\cup\overline{\D})$ is covered by $\D(-1,\varepsilon_1)\cup\D(1,\varepsilon_2)\cup\D(\omega,\varepsilon_3)\cup\D(\overline{\omega},\varepsilon_3)$;
\item The closure of $\D\cup\D(-1,\varepsilon_1)\cup\D(1,\varepsilon_2)\cup\D(\omega,\varepsilon_3)\cup\D(\overline{\omega},\varepsilon_3)$ is contained in the disk $\mathscr{D}$.
\end{enumerate}
\end{lem}

\begin{figure}[!htpb]
  \setlength{\unitlength}{1mm}
   \centering
  \includegraphics[width=0.98\textwidth]{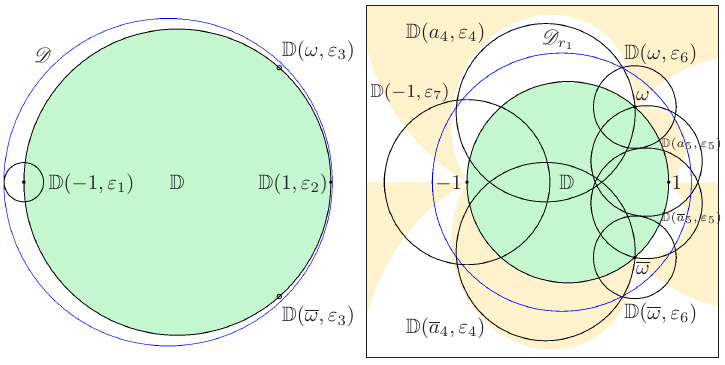}
  \caption{Left: The disk $\mathscr{D}$ covers $\OD(-1,\varepsilon_1)$, $\OD(1,\varepsilon_2)$, $\OD(\omega,\varepsilon_3)$, $\OD(\overline{\omega},\varepsilon_3)$ and $\OD$.  Right: The closure of the disk $\mathscr{D}_{r_1}$ is covered by the union of $7$ disks, which are introduced in Lemmas \ref{lema-a4} and \ref{lema-E-r1}.}
  \label{Fig-E}
\end{figure}

\begin{proof}
(a) By the definition of $U_\eta^P$ (see \S\ref{sec-definitions}) and the relation between $P$ and $Q$ (see Lemma \ref{lema-cp-cv-Q}), one can see that $\EC\setminus (U_\eta^Q\cup\OD)$ consists of $4$ connected components $W'$, $W''$, $W^+$ and $W^-$ such that $-1\in\partial W'$, $1\in\partial W''$, $\omega\in \partial W^+$ and $\overline{\omega}\in\partial W^-$ (note that $\psi_1(-1)=\infty$, $\psi_1(1)=-1$ and $\psi_1(\omega)=\psi_1(\overline{\omega})=-\mu$).
Moreover, $|Q(\zeta)|\leq cv\,e^{-2\pi\eta}$ for $\zeta\in W'\cup W''$ and $|Q(\zeta)|\geq cv\,e^{2\pi\eta}$ for $\zeta\in W^+\cup W^-$.
If we can show that $|Q(\zeta)|> cv\,e^{-2\pi\eta}$ on $\partial\D(-1,\varepsilon_1)\cup\partial\D(1,\varepsilon_2)$ and $|Q(\zeta)|< cv\,e^{2\pi\eta}$ on $\partial\D(\omega,\varepsilon_3)\cup\partial\D(\overline{\omega},\varepsilon_3)$,
then we have $\overline{W}'\subset\D(-1,\varepsilon_1)$, $\overline{W}''\subset\D(1,\varepsilon_2)$, $\overline{W}^+\subset\D(\omega,\varepsilon_3)$ and $\overline{W}^-\subset\D(\overline{\omega},\varepsilon_3)$, and the result follows.

\medskip
Note that $\omega=\tfrac{8\sqrt{6}-3}{25}+\tfrac{6\sqrt{6}+4}{25}\ii\,(\doteqdot 0.6638\ldots+0.7478\ldots\ii)\in\partial\D$. We have $|\omega-1|<1$ and $2\,\im\omega=|\omega-\overline{\omega}|>1$. From \eqref{equ-Q-defi} and Lemma \ref{lema-cp-cv-Q} we have
\begin{equation}\label{equ:Q-another-form}
Q(\zeta)=\frac{(\zeta+1)^6(\zeta-1)^4}{\zeta(\zeta-\omega)^4(\zeta-\overline{\omega})^4}.
\end{equation}
If $|\zeta+1|=\varepsilon_1$, then
\renewcommand\theequation{\thesection.\arabic{equation}*}
\begin{equation}
\begin{split}
|Q(\zeta)|
&\geq \frac{\varepsilon_1^6 (2-\varepsilon_1)^4}{(1+\varepsilon_1)(2+\varepsilon_1)^8}\,(\doteqdot 1.13\ldots\times 10^{-7}) \\
&\gy cv\, e^{-2\pi\eta} \,(\doteqdot 1.10\ldots\times 10^{-7}). \retainlabel{equ:y1}
\end{split}
\end{equation}
If $|\zeta-1|=\varepsilon_2$, then
\renewcommand\theequation{\thesection.\arabic{equation}*}
\begin{equation}
\begin{split}
|Q(\zeta)|
&\geq \frac{(2-\varepsilon_2)^6\varepsilon_2^4}{(1+\varepsilon_2)(1+\varepsilon_2)^8}\,(\doteqdot 1.41\ldots\times 10^{-7}) \\
&\gy cv\, e^{-2\pi\eta} \,(\doteqdot 1.10\ldots\times 10^{-7}).\retainlabel{equ:y2}
\end{split}
\end{equation}
If $|\zeta-\omega|=\varepsilon_3$, then
\renewcommand\theequation{\thesection.\arabic{equation}*}
\begin{equation}\label{equ:y3}
|Q(\zeta)|
\leq \frac{(2+\varepsilon_3)^6(1+\varepsilon_3)^4}{(1-\varepsilon_3)\varepsilon_3^4(1-\varepsilon_3)^4}\,(\doteqdot 1.97\ldots\times 10^9)
\ly cv\, e^{2\pi\eta} \,(\doteqdot 2.60\ldots\times 10^9).
\end{equation}
By the symmetry of $Q$, we have the same estimate of $Q$ on $\partial\D(\overline{\omega},\varepsilon_3)$ as \eqref{equ:y3}.

\medskip
(b) To prove that $\OD(\widetilde{a},\widetilde{r})$ is contained in $\mathscr{D}=\D(a_0,r_0)$, where $a_0=-0.06$ and $r_0=1.07$, it is sufficient to verify that $|\widetilde{a}-a_0|<r_0-\widetilde{r}$. By direct calculations,
\begin{align}
|0-a_0|-(r_0-1) &=0.06+1-1.07=-0.01<0, \\
|-1-a_0|-(r_0-\varepsilon_1) &=0.94+0.128-1.07=-0.002<0, \\
|1-a_0|-(r_0-\varepsilon_2) &=1.06+0.007-1.07=-0.003<0,
\end{align}
and
\renewcommand\theequation{\thesection.\arabic{equation}*}
\begin{align}
|\omega-a_0|^2-(r_0-\varepsilon_3)^2
=&~\big(\tfrac{8\sqrt{6}-3}{25}+0.06\big)^2+\big(\tfrac{6\sqrt{6}+4}{25}\big)^2-1.056^2 \\
(\doteqdot&~ -0.031\ldots)\ly 0. \retainlabel{1.056-omega}
\end{align}
The proof is finished.
\end{proof}

\begin{proof}[Proof of Proposition \ref{prop-V-V'}]
By Lemma \ref{lema-D-E}, we have
\begin{equation}
\EC\setminus \mathscr{D}\subset \EC\setminus (\OD(-1,\varepsilon_1)\cup\OD(1,\varepsilon_2)\cup\OD(\omega,\varepsilon_3)\cup\OD(\overline{\omega},\varepsilon_3)\cup\OD)\subset U_\eta^Q\subset\EC\setminus\OD.
\end{equation}
Therefore, we have $\overline{V}=\psi_1(\EC\setminus \mathscr{D})\subset \psi_1(U_\eta^Q)=U_\eta^P=V'$.
This ends the proof of Proposition \ref{prop-V-V'}.
\end{proof}

\begin{lem}\label{lema-a4}
Let $a_4=-0.22+0.69\ii$, $a_5=0.78+0.21\ii$, $\varepsilon_4=|a_4-\omega|$ and $\varepsilon_5=|a_5-\omega|$. Then
\begin{enumerate}
\item $\OD(a_4,\varepsilon_4)\setminus\OD\subset\MU_{4+}$ and $\OD(\overline{a}_4,\varepsilon_4)\setminus\OD\subset\MU_{4-}$; and
\item $\big(\partial\D(a_5,\varepsilon_5)\cap\BH_+\big)\setminus\OD\subset \MU_{2+}$ and $\big(\partial\D(\overline{a}_5,\varepsilon_5)\cap\BH_-\big)\setminus\OD\subset \MU_{3-}$.
\end{enumerate}
\end{lem}

The proof of this lemma needs elaborate estimates. We postpone the proof to Appendix \ref{sec:appendix}. See the picture in Figure \ref{Fig-E} on the right and Figure \ref{Fig-a4-a5}.

\medskip
Recall that $\MU_i=\MU_{i+}\cup\MU_{i-}\cup\gamma_{ai}$ for $i=1,2,3$ and $\MU_{123}=\MU_1\cup\,\MU_2\cup\,\MU_3\cup\gamma_{b1}\cup\gamma_{b2}\cup\gamma_{b3}$ are defined in \S\ref{sec-P-to-Q}. We denote $\MU_{12}':=\MU_1\cup\MU_{2-}\cup\MU_{3+}\cup\gamma_{b1}\cup\gamma_{b3}$. For $r>1$, we define $\mathscr{D}_r:=\D(a_0, r_0 r)$.

\begin{lem}[see Figure \ref{Fig-E} on the right]\label{lema-E-r1}
Let $R=125$, $\rho=0.03$, $\varepsilon_6=0.41$, $\varepsilon_7=0.82$ and $r_1=1.2$. We have
\begin{enumerate}
\item If $\zeta\in\OD(\omega,\varepsilon_6)\setminus\big(\D(a_4,\varepsilon_4)\cup \D(a_5,\varepsilon_5)\big)$ or $\zeta\in\OD(\overline{\omega},\varepsilon_6)\setminus\big(\D(\overline{a}_4,\varepsilon_4)\cup \D(\overline{a}_5,\varepsilon_5)\big)$, then $|Q(\zeta)|>R$;
\item If $\zeta\in\OD(-1,\varepsilon_7)\setminus\big(\D(a_4,\varepsilon_4)\cup \D(\overline{a}_4,\varepsilon_4)\big)$, then $|Q(\zeta)|<\rho$;
\item $\overline{\mathscr{D}}_{r_1}$ is covered by the union of $7$ disks: $\D(a_4,\varepsilon_4)$, $\D(\overline{a}_4,\varepsilon_4)$, $\D(a_5,\varepsilon_5)$, $\D(\overline{a}_5,\varepsilon_5)$, $\D(\omega,\varepsilon_6)$, $\D(\overline{\omega},\varepsilon_6)$ and $\D(-1,\varepsilon_7)$;
\item If $\zeta\in\overline{\MU}_1$ and $|Q(\zeta)|\geq \rho$, then $\zeta\in\EC\setminus\overline{\mathscr{D}}_{r_1}$;
\item If $\zeta\in\overline{\MU}_{12}'$ and $\rho\leq |Q(\zeta)|\leq R$, then $\zeta\in\C\setminus\overline{\mathscr{D}}_{r_1}$; and
\item If $\zeta\in\overline{\MU}_{123}\setminus\big(\OD(a_5,\varepsilon_5)\cup\OD(\overline{a}_5,\varepsilon_5)\big)$ and $\rho\leq |Q(\zeta)|\leq R$, then $\zeta\in\C\setminus\overline{\mathscr{D}}_{r_1}$.
\end{enumerate}
\end{lem}

\begin{proof}
To unify notations, we denote $a_6:=\omega$ and $a_7:=-1$. For $4\leq j\leq 7$, we parameterize the upper and lower half parts of the circle $\partial\D(a_j,\varepsilon_j)$:
\begin{equation}\label{equ:y-j-pm}
y_j^\pm(x)=\im a_j\pm\sqrt{\varepsilon_j^2-(x-\re a_j)^2}.
\end{equation}
Based on \eqref{equ:Q-another-form}, for $x,x'\in[\re a_j-\varepsilon_j, \re a_j+\varepsilon_j]$ we define
\begin{align}
\xi_{j,1}^\pm(x)&=\big((x+1)^2+(y_j^\pm(x))^2\big)^3, & \xi_{j,2}^\pm(x)&=\big((x-1)^2+(y_j^\pm(x))^2\big)^2, \\
\xi_{j,3}^\pm(x)&=\big(x^2+(y_j^\pm(x))^2\big)^{1/2}, & \xi_{j,4}^\pm(x)&=\big((x-\re\omega)^2+(y_j^\pm(x)-\im\omega)^2\big)^2, \\
&  & \xi_{j,5}^\pm(x)&=\big((x-\re\omega)^2+(y_j^\pm(x)+\im\omega)^2\big)^2,
\end{align}
and
\begin{equation}\label{equ:Xi-j-pm}
\Xi_j^\pm(x,x'):=\frac{\xi_{j,1}^\pm(x')\cdot \xi_{j,2}^\pm(x')}{\xi_{j,3}^\pm(x)\cdot\xi_{j,4}^\pm(x)\cdot\xi_{j,5}^\pm(x)}.
\end{equation}
Then $|Q(\zeta)|=\Xi_j^\pm(x,x)$ if $\zeta=x+y_j^\pm(x)\in\partial\D(a_j,\varepsilon_j)$. The function $\Xi_j^\pm(x,x')$ will be used to estimate the upper and lower bounds of $|Q(\zeta)|$.

\medskip
(a) Consider the following $4$ arcs $\ell_{6,0}$, $\ell_{6,1}$, $\ell_{6,2}^+$ and $\ell_{6,2}^-$:
\begin{align}
\ell_{6,0}&:=y_4^+(x) \text{ with } x_{6,0}\leq x\leq x_{6,1},
& \ell_{6,1}&:=y_5^+(x) \text{ with } x_{6,1}\leq x\leq x_{6,2}^+, \\
\ell_{6,2}^+&:=y_6^+(x) \text{ with } x_{6,0}\leq x\leq x_{6,3},
& \ell_{6,2}^-&:=y_6^-(x) \text{ with } x_{6,2}^-\leq x\leq x_{6,3},
\end{align}
where $x_{6,0}=0.54$,  $x_{6,1}=\re\omega=\tfrac{8\sqrt{6}-3}{25}\Aeq{0.6638}$, $x_{6,2}^+=1.07$, $x_{6,2}^-=1.067$ and $x_{6,3}=\re\omega+\varepsilon_6=\tfrac{8\sqrt{6}-3}{25}+0.41\Aeq{1.0738}$.
Note that the boundary of $\D(\omega,\varepsilon_6)\setminus\big(\D(a_4,\varepsilon_4)\cup \D(a_5,\varepsilon_5)\big)$ consists of three arcs $\widetilde{\ell}_{6,0}\subset\partial\D(a_4,\varepsilon_4)$, $\widetilde{\ell}_{6,1}\subset\partial\D(a_5,\varepsilon_5)$ and $\widetilde{\ell}_{6,2}\subset\partial\D(\omega,\varepsilon_6)$.
We have $\widetilde{\ell}_{6,0}\subset \ell_{6,0}$, $\widetilde{\ell}_{6,1}\subset \ell_{6,1}$ and $\widetilde{\ell}_{6,2}\subset \ell_{6,2}^+\cup\ell_{6,2}^-$ from the following estimates:
\renewcommand\theequation{\thesection.\arabic{equation}*}
\begin{align}
(x_{6,0}-\re\omega)^2+(y_4^+(x_{6,0})-\im\omega)^2-\varepsilon_6^2 \Aeq{0.0048} &\gy 0, \retainlabel{0.1572-x60omega}\\
(x_{6,2}^+-\re\omega)^2+(y_5^+(x_{6,2}^+)-\im\omega)^2- \varepsilon_6^2\Aeq{0.0017} & \gy 0, \retainlabel{0.1525-x62omega}\\
(x_{6,0}-\re a_4)^2+(y_6^+(x_{6,0})-\im a_4)^2-\varepsilon_4^2 \Aeq{-0.0055} &\ly 0, \retainlabel{0.7785-x60a4}\\
(x_{6,2}^--\re a_5)^2+(y_6^-(x_{6,2}^-)-\im a_5)^2-\varepsilon_5^2\Aeq{-0.0057} & \ly 0. \retainlabel{0.2933-x62a5}
\end{align}
In the following we prove that $|Q(\zeta)|>R=125$ on $\ell_{6,0}\cup\ell_{6,1}\cup\ell_{6,2}^+\cup\ell_{6,2}^-$. Then Part (a) follows immediately.

\medskip
(\textbf{Case a.1})  Let $\zeta(x)=x+y_4^+(x)\ii\in\ell_{6,0}$ with $x\in [x_{6,0}, x_{6,1}]$. The line passing through $-1$ and $a_4$ has the equation $y=\frac{23}{26}(x+1)$. The arc $\ell_{6,0}$ lies below this line. Indeed, we only need to verify the point $\zeta(x_{6,0})$:
\renewcommand\theequation{\thesection.\arabic{equation}*}
\begin{align}\retainlabel{0.2174-y4}
y_4^+(x_{6,0})-\tfrac{23}{26}(x_{6,0}+1)\Aeq{-0.2174}\ly 0.
\end{align}
Similarly, the arc $\ell_{6,0}$ lies above the line passing through $1$ and $a_4$.
Then it is easy to see that $|\zeta(x)+1|$, $|\zeta(x)-1|$, $|\zeta(x)|$, $|\zeta(x)-\omega|$ and $|\zeta(x)-\overline{\omega}|$ decrease monotonously for $x\in[x_{6,0},x_{6,1}]$. Denote $x_{6,0,1}=0.6$. If $x\in[x_{6,0},x_{6,0,1}]$, then $|Q(\zeta)|\geq\Xi_4^+(x_{6,0},x_{6,0,1})$, where $\Xi_4^+$ is defined in \eqref{equ:Xi-j-pm}. If $x\in[x_{6,0,1},x_{6,1}]$, then $|Q(\zeta)|\geq\Xi_4^+(x_{6,0,1},x_{6,1})$. One can verify that
\renewcommand\theequation{\thesection.\arabic{equation}*}
\begin{align}
\Xi_4^+(x_{6,0},x_{6,0,1})\Aeq{140.6}\gy 125, \quad \Xi_4^+(x_{6,0,1},x_{6,1})\Aeq{217.0}\gy 125.\retainlabel{217.07-xi4}
\end{align}
This implies that $|Q(\zeta)|>R=125$ for all $\zeta\in\ell_{6,0}$.

\medskip
(\textbf{Case a.2})  Let $\zeta(x)=x+y_5^+(x)\ii\in\ell_{6,1}$ with $x\in [x_{6,1}, x_{6,2}^+]$. We denote the line passing through $\overline{\omega}$ and $a_5$ by $\iota_{6,1}$ with slope $s_{6,1}=\frac{\im (a_5+\omega)}{\re(a_5-\omega)}>0$. Let $\zeta(\widetilde{x}_{6,1})$ be one of the points in $\iota_{6,1}\cap\partial\D(a_5,\varepsilon_5)$ with larger imaginary part. Then $\widetilde{x}_{6,1}=\re a_5+\varepsilon_5/\sqrt{1+s_{6,1}^2}\Aeq{0.8462}$.
It implies that $|\zeta(x)-\overline{\omega}|$ increases on $[x_{6,1},\widetilde{x}_{6,1}]$ and decreases on $[\widetilde{x}_{6,1},x_{6,2}^+]$. Note that the center $a_5$ lies below the line passing through $1$ and $\omega$. Indeed, this is true because of the following estimate:
\renewcommand\theequation{\thesection.\arabic{equation}*}
\begin{align}\retainlabel{0.8462-tilde-x}
\im a_5-\tfrac{\im\omega}{\re\omega-1}(\re a_5-1)\Aeq{-0.2794}\ly 0.
\end{align}
Hence $|\zeta(x)-1|$ decreases about $x\in [x_{6,1}, x_{6,2}^+]$. Moreover, it is easy to see that $|\zeta(x)+1|$, $|\zeta(x)|$ and $|\zeta(x)-\omega|$ increase on $[x_{6,1}, x_{6,2}^+]$.

Define $x_{6,1,1}=0.99$, $x_{6,1,2}=1.05$ and $x_{6,1,3}=1.06$. Note that $|\zeta-\overline{\omega}|\leq |a_5-\overline{\omega}|+\varepsilon_5$ for $\zeta\in\ell_{6,1}$. If $x\in[x_{6,1},x_{6,1,1}]$, then
\begin{align}\retainlabel{132.59-Q-zeta}
|Q(\zeta)|\geq \frac{\xi_{5,1}^+(x_{6,1})\cdot \xi_{5,2}^+(x_{6,1,1})}{\xi_{5,3}^+(x_{6,1,1})\cdot\xi_{5,4}^+(x_{6,1,1})\cdot (|a_5-\overline{\omega}|+\varepsilon_5)^4}
\Aeq{132.5}\gy 125.
\end{align}
Define
\begin{align}
\widehat{\Xi}_5^+(x,x'):=\frac{\xi_{5,1}^+(x)\cdot \xi_{5,2}^+(x')}{\xi_{5,3}^+(x')\cdot\xi_{5,4}^+(x')\cdot\xi_{5,5}^+(x)}.
\end{align}
Then $|Q(\zeta)|\geq \widehat{\Xi}_5^+(x,x')$ for any interval $[x,x']\subset[x_{6,1,1}, x_{6,2}^+]$. We have following estimates:
\renewcommand\theequation{\thesection.\arabic{equation}*}
\begin{align}
\widehat{\Xi}_5^+(x_{6,1,1},x_{6,1,2})\Aeq{137.6}&\gy 125, ~~ \widehat{\Xi}_5^+(x_{6,1,2},x_{6,1,3})\Aeq{141.4}  \gy 125, \quad\qquad  \retainlabel{252.32-iota}\\
\widehat{\Xi}_5^+(x_{6,1,3},x_{6,2}^+)\Aeq{126.5}&\gy 125.  \retainlabel{150.16-iota}
\end{align}
This implies that $|Q(\zeta)|>R=125$ for all $\zeta\in\ell_{6,1}$.

\medskip
(\textbf{Case a.3})  Let $\zeta(x)=x+y_6^+(x)\ii\in\ell_{6,2}^+$ with $x\in [x_{6,0}, x_{6,3}]$. Then $\xi_{6,4}^+=\xi_{6,4}^+(x)=\varepsilon_6^4$ for all $x\in [x_{6,0}, x_{6,3}]$.
Let $\iota_{6,2}$, $\iota_{6,3}$ and $\iota_{6,4}$,  be the lines passing through $\omega$ and $-1$, $\omega$ and $1$, $\omega$ and $0$,  with slopes $s_{6,2}=\frac{\im \omega}{\re\omega+1}>0$, $s_{6,3}=\frac{\im \omega}{\re\omega-1}<0$ and $s_{6,4}=\frac{\im \omega}{\re\omega}>0$ respectively. We use $\zeta(\widetilde{x}_{6,i})$, where $i=2,3,4$, to denote one of the points in $\iota_{6,i}\cap\partial\D(\omega,\varepsilon_6)$ with larger imaginary part. Then $\widetilde{x}_{6,2}=\re \omega+\varepsilon_6/\sqrt{1+s_{6,2}^2}\Aeq{1.0377}$, $\widetilde{x}_{6,3}=\re \omega-\varepsilon_6/\sqrt{1+s_{6,3}^2}\Aeq{0.4957}$ and $\widetilde{x}_{6,4}=\re \omega+\varepsilon_6/\sqrt{1+s_{6,4}^2}\Aeq{0.9360}$.
It implies that $|\zeta(x)+1|$ increases on $[x_{6,0},\widetilde{x}_{6,2}]$ and decreases on $[\widetilde{x}_{6,2},x_{6,3}]$, $|\zeta(x)-1|$ decreases on $[x_{6,0},x_{6,3}]$, and $|\zeta(x)|$ increases on $[x_{6,0},\widetilde{x}_{6,4}]$ and decreases on $[\widetilde{x}_{6,4},x_{6,3}]$. Obviously, $|\zeta(x)-\overline{\omega}|$ increases on $[x_{6,0},\re\omega]$ and decreases on $[\re\omega,x_{6,3}]$.

Define $x_{6,2,1}=\re\omega$, $x_{6,2,2}=0.99$ and $x_{6,2,3}=1.07$. If $x\in[x_{6,0},x_{6,2,1}]$, we have
\renewcommand\theequation{\thesection.\arabic{equation}*}
\begin{align}
|Q(\zeta)|\geq\frac{\xi_{6,1}^+(x_{6,0})\cdot \xi_{6,2}^+(x_{6,2,1})}{\xi_{6,3}^+(x_{6,2,1})\cdot\xi_{6,4}^+\cdot\xi_{6,5}^+(x_{6,2,1})}\Aeq{209.6}\gy 125. \retainlabel{270.52-Q-zeta}
\end{align}
If $x\in[x_{6,2,1}, x_{6,2,2}]$, we have
\renewcommand\theequation{\thesection.\arabic{equation}*}
\begin{align}
|Q(\zeta)|\geq\frac{\xi_{6,1}^+(x_{6,2,1})\cdot \xi_{6,2}^+(x_{6,2,2})}{(1+\varepsilon_6)\cdot\xi_{6,4}^+\cdot\xi_{6,5}^+(x_{6,2,1})}\Aeq{130.0}\gy 125. \retainlabel{158.40-Q-zeta}
\end{align}
If $x\in[x_{6,2,2}, x_{6,2,3}]$, we have
\renewcommand\theequation{\thesection.\arabic{equation}*}
\begin{align}
|Q(\zeta)|\geq&\,\min\left\{\frac{\xi_{6,1}^+(x_{6,2,2})\cdot \xi_{6,2}^+(x_{6,2,3})}{\xi_{6,3}^+(x_{6,2,2})\cdot\xi_{6,4}^+\cdot\xi_{6,5}^+(x_{6,2,2})},~\Xi_6^+(x_{6,2,2},x_{6,2,3})\right\} \\
(\doteqdot&\,\min\,\{130.8\ldots, 129.1\ldots\})\gy 125. \retainlabel{194.00-Q-zeta}
\end{align}
$x\in[x_{6,2,3},x_{6,3}]$, we have
\renewcommand\theequation{\thesection.\arabic{equation}*}
\begin{align}
|Q(\zeta)|\geq \Xi_6^+(x_{6,2,3},x_{6,3})\Aeq{146.2}\gy 125. \retainlabel{152.39-Q-zeta}
\end{align}
Hence we have $|Q(\zeta)|>R=125$ for all $\zeta\in\ell_{6,2}^+$.

\medskip
(\textbf{Case a.4}) Let $\zeta(x)=x+y_6^-(x)\ii\in\ell_{6,2}^-$ with $x\in [x_{6,2}^-, x_{6,3}]$. Then $\xi_{6,4}^-=\xi_{6,4}^-(x)=\varepsilon_6^4$ for all $x\in [x_{6,2}^-, x_{6,3}]$.
From (Case a.3) we know that $|\zeta(x)+1|$, $|\zeta(x)-1|$, $|\zeta(x)|$ and $|\zeta(x)-\overline{\omega}|$ increase on $[x_{6,2}^-, x_{6,3}]$. Hence we have $|Q(\zeta)|\geq \Xi_6^-(x',x)$ for any interval $[x,x']\subset[x_{6,2}^-, x_{6,3}]$. Define $x_{6,3,1}=1.069$ and $x_{6,3,2}=1.072$. Then
\renewcommand\theequation{\thesection.\arabic{equation}*}
\begin{align}
\Xi_6^-(x_{6,3,1},x_{6,2}^-)\Aeq{125.3} & \gy 125, & \Xi_6^-(x_{6,3,2},x_{6,3,1})\Aeq{126.5}&\gy 125, \quad\qquad  \retainlabel{157.60-Xi6}\\
\Xi_6^-(x_{6,3},x_{6,3,2})\Aeq{133.0} & \gy 125.  \retainlabel{157.79-Xi6}
\end{align}
Hence we have $|Q(\zeta)|>R=125$ for all $\zeta\in\ell_{6,2}^-$. The proof of Part (a) is finished.

\medskip
(b) Consider the following $3$ arcs $\ell_{7,0}$, $\ell_{7,1}$ and $\ell_{7,2}$:
\begin{align}
\ell_{7,0} &:=y_7^+(x) \text{ with } x_{7,0}\leq x\leq x_{7,1},
& \ell_{7,1} &:=y_4^+(x) \text{ with } x_{7,2}\leq x\leq x_{7,1}, \\
\ell_{7,2}&:=y_4^-(x) \text{ with } x_{7,2}\leq x\leq x_{7,3},
\end{align}
where $x_{7,0}=-1-\varepsilon_7=-1.82$, $x_{7,1}=-1.095$, $x_{7,2}=-0.22-\varepsilon_4\Aeq{-1.1057}$ and $x_{7,3}=-0.77$.
Note that the boundary of $\D(-1,\varepsilon_7)\setminus\big(\D(a_4,\varepsilon_4)\cup \D(a_4,\varepsilon_4)\big)$ consists of three arcs $\widetilde{\ell}_{7,0}\subset\partial\D(-1,\varepsilon_7)$, $\widetilde{\ell}_{7,1}\subset\partial\D(a_4,\varepsilon_4)$ and $\widetilde{\ell}_{7,2}\subset\partial\D(\overline{a}_4,\varepsilon_4)$.
We have $\widetilde{\ell}_{7,0}\subset \ell_{7,0}\cup\overline{\ell}_{7,0}$, $\widetilde{\ell}_{7,1}\subset \ell_{7,1}\cup\ell_{7,2}$ and $\widetilde{\ell}_{7,2}\subset \overline{\ell}_{7,1}\cup\overline{\ell}_{7,2}$, where $\overline{\ell}_{7,i}=\{\overline{\zeta}: \zeta\in \ell_{7,i}\}$ and $i=0,1,2$, from the following estimates:
\renewcommand\theequation{\thesection.\arabic{equation}*}
\begin{align}
(x_{7,1}-\re a_4)^2+(y_7^+(x_{7,1})-\im a_4)^2-\varepsilon_4^2\Aeq{-0.0033} &\ly 0, \retainlabel{0.7845-x71}\\
(x_{7,1}+1)^2+(y_4^+(x_{7,1}))^2-\varepsilon_7^2\Aeq{0.0212} &\gy 0, \retainlabel{0.6724-x71}\\
y_4^-(x_{7,3}) \Aeq{-0.0042}&\ly 0. \retainlabel{0.0042-y4}
\end{align}
In the following we prove that $|Q(\zeta)|<\rho=0.03$ on $\ell_{7,0}\cup\ell_{7,1}\cup\ell_{7,2}$. Then Part (b) follows immediately.

\medskip
(\textbf{Case b.1})  Let $\zeta(x)=x+y_7^+(x)\ii\in\ell_{7,0}$ with $x\in [x_{7,0}, x_{7,1}]$. Then $\xi_{7,1}^+=\xi_{7,1}^+(x)=\varepsilon_7^6$ for all $x\in [x_{7,0}, x_{7,1}]$.
We denote the line passing through $\overline{\omega}$ and $-1$ by $\iota_{7,1}$ with slope $s_{7,1}=-\frac{\im \omega}{\re\omega+1}$. Let $\zeta(\widetilde{x}_{7,1})$ be one of the points in $\iota_{7,1}\cap\partial\D(-1,\varepsilon_7)$ with larger imaginary part. Then $\widetilde{x}_{7,1}=-1-\varepsilon_7/\sqrt{1+s_{7,1}^2}\Aeq{-1.7479}$.
It implies that $|\zeta(x)-\overline{\omega}|$ increases on $[x_{7,0},\widetilde{x}_{7,1}]$ and decreases on $[\widetilde{x}_{7,1},x_{7,1}]$. Moreover, it is easy to see that $|\zeta(x)-1|$, $|\zeta(x)|$ and $|\zeta(x)-\omega|$ decrease on $[x_{7,0}, x_{7,1}]$.

Define $x_{7,0,1}=-1.3$.
If $x\in [x_{7,0}, x_{7,0,1}]$, then
\renewcommand\theequation{\thesection.\arabic{equation}*}
\begin{align}
|Q(\zeta)|\leq&\,\max\left\{\frac{\xi_{7,1}^+\cdot \xi_{7,2}^+(x_{7,0})}{\xi_{7,3}^+(x_{7,0,1})\cdot\xi_{7,4}^+(x_{7,0,1})\cdot\xi_{7,5}^+(x_{7,0})},~\Xi_7^+(x_{7,0,1},x_{7,0})\right\} \\
(\doteqdot&\,\max\,\{0.0189\ldots, 0.0227\ldots\})\ly 0.03. \retainlabel{0.0244-Q-zeta}
\end{align}
If $x\in [x_{7,0,1}, x_{7,1}]$. We have
\renewcommand\theequation{\thesection.\arabic{equation}*}
\begin{align}
|Q(\zeta)|\leq \Xi_7^+(x_{7,1},x_{7,0,1})\Aeq{0.0261}\ly 0.03. \retainlabel{0.0261-Q-zeta}
\end{align}
Hence we have $|Q(\zeta)|<\rho=0.03$ for all $\zeta\in\ell_{7,0}$.

\medskip
(\textbf{Case b.2}) Let $\zeta(x)=x+y_4^+(x)\ii\in\ell_{7,1}$ with $x\in [x_{7,2}, x_{7,1}]$.
We denote the line passing through $1$ and $a_4$ by $\iota_{7,2}$ with slope $s_{7,2}=\frac{\im a_4}{\re a_4-1}<0$. Let $\zeta(\widetilde{x}_{7,2})$ be one of the points in $\iota_{7,2}\cap\partial\D(a_4,\varepsilon_4)$ with larger imaginary part. Then $\widetilde{x}_{7,2}=\re a_4-\varepsilon_4/\sqrt{1+s_{7,2}^2}\Aeq{-0.9909}$.
It implies that $|\zeta(x)-1|$ increases on $[x_{7,2}, x_{7,1}]$. Moreover, it is easy to see that $|\zeta(x)+1|$, $|\zeta(x)|$ and $|\zeta(x)-\overline{\omega}|$ increase on $[x_{7,2}, x_{7,1}]$ while $|\zeta(x)-\omega|$ decrease on $[x_{7,2}, x_{7,1}]$.
Therefore,
\renewcommand\theequation{\thesection.\arabic{equation}*}
\begin{align}
|Q(\zeta)|\leq\frac{\xi_{4,1}^+(x_{7,1})\cdot \xi_{4,2}^+(x_{7,1})}
{\xi_{4,3}^+(x_{7,2})\cdot\xi_{4,4}^+(x_{7,1})\cdot\xi_{4,5}^+(x_{7,2})}\Aeq{0.0253}\ly 0.03. \retainlabel{0.0253-Q-zeta}
\end{align}
Hence we have $|Q(\zeta)|<\rho=0.03$ for all $\zeta\in\ell_{7,1}$.

\medskip
(\textbf{Case b.3}) Let $\zeta(x)=x+y_4^-(x)\ii\in\ell_{7,2}$ with $x\in [x_{7,2}, x_{7,3}]$.  We denote the line passing through $-1$ and $a_4$ by $\iota_{7,3}$ with slope $s_{7,3}=\frac{\im a_4}{\re a_4+1}>0$. Let $\zeta(\widetilde{x}_{7,3})$ be one of the points in $\iota_{7,3}\cap\partial\D(a_4,\varepsilon_4)$ with smaller imaginary part. Then $\widetilde{x}_{7,3}=\re a_4-\varepsilon_4/\sqrt{1+s_{7,3}^2}\Aeq{-0.8834}$. This implies that $|\zeta(x)+1|$ decreases on $[x_{7,2}, \widetilde{x}_{7,3}]$ and increases on $[\widetilde{x}_{7,3},x_{7,3}]$.
Let $\iota_{7,4}$ be the line passing through $\omega$ and $a_4$ with slope $s_{7,4}=\frac{\im (\omega-a_4)}{\re (\omega-a_4)}>0$. Let $\zeta(\widetilde{x}_{7,4})$ be one of the points in $\iota_{7,4}\cap\partial\D(a_4,\varepsilon_4)$ with smaller imaginary part. Then $\widetilde{x}_{7,4}=\re a_4-\varepsilon_4/\sqrt{1+s_{7,4}^2}\Aeq{-1.1038}$.
This implies that $|\zeta(x)-\omega|$ increases on $[x_{7,2}, \widetilde{x}_{7,4}]$ and decreases on $[\widetilde{x}_{7,4},x_{7,3}]$.
It is easy to see that $|\zeta(x)-1|$, $|\zeta(x)|$ and $|\zeta(x)-\overline{\omega}|$ decrease on $[x_{7,2}, x_{7,3}]$.

Define $x_{7,2,1}=-1$. If $x\in [x_{7,2}, x_{7,2,1}]$, we have
\renewcommand\theequation{\thesection.\arabic{equation}*}
\begin{align}
|Q(\zeta)|\leq&\,\max\left\{\frac{\xi_{4,1}^-(x_{7,2})\cdot \xi_{4,2}^-(x_{7,2})}{\xi_{4,3}^-(x_{7,2,1})\cdot\xi_{4,4}^-(x_{7,2})\cdot\xi_{4,5}^-(x_{7,2,1})},~\Xi_4^-(x_{7,2,1},x_{7,2})\right\} \\
(\doteqdot&\,\max\,\{ 0.0189\ldots,0.0207\ldots\})\ly 0.03. \retainlabel{0.0207-Q-zeta}
\end{align}
If $x\in [x_{7,2,1},x_{7,3}]$, we have
\renewcommand\theequation{\thesection.\arabic{equation}*}
\begin{align}
|Q(\zeta)|\leq&\,\max\left\{\Xi_4^-(x_{7,3},x_{7,2,1}),~\frac{\xi_{4,1}^-(x_{7,3})\cdot \xi_{4,2}^-(x_{7,2,1})}{\xi_{4,3}^-(x_{7,3})\cdot\xi_{4,4}^-(x_{7,3})\cdot\xi_{4,5}^-(x_{7,3})}\right\} \\
(\doteqdot&\,\max\,\{1.7\ldots\times 10^{-4}, 6.8\ldots\times 10^{-5}\})\ly 0.03. \retainlabel{1.7-6.8-Q-zeta}
\end{align}
Hence we have $|Q(\zeta)|<\rho=0.03$ for all $\zeta\in\ell_{7,2}$.  The proof of Part (b) is finished.

\medskip

(c) We use a such criterion: If two points $\zeta_1,\zeta_2\in\partial\mathscr{D}_{r_1}\cap\{\zeta\in\C:\im\zeta\geq 0\}$ are contained in a disk $\D(\zeta_0,r)$ with $\im \zeta_0\geq 0$, then so is the subarc of $\partial\mathscr{D}_{r_1}$ between $\zeta_1$ and $\zeta_2$ in $\{\zeta\in\C:\im\zeta\geq 0\}$. The upper half of $\partial\mathscr{D}_{r_1}$ can be parametrized by
\begin{equation}
\Gamma:y(x)=\sqrt{(r_0r_1)^2-(x-a_0)^2}, \quad x_0\leq x\leq x_4,
\end{equation}
where $a_0=-0.06$, $r_0=1.07$, $r_1=1.2$, $x_0=a_0-r_0r_1=-1.344$ and $x_4=a_0+r_0r_1=1.224$.

Let $x_1=-1.1$, $x_2=0.54$, $x_3=1.03$ and denote $\zeta_j=x_j+y(x_j)\ii$ for $j=1,2,3$. Then $\zeta_1$, $\zeta_2$ and $\zeta_3$ divide the curve $\Gamma$ into $4$ subarcs $\Gamma_1$, $\Gamma_2$, $\Gamma_3$ and $\Gamma_4$ from left to right. The end points of $\Gamma_1$: $x_0$ and $\zeta_1$, are contained in $\D(-1,\varepsilon_7)$ since $|x_0+1|=0.344<\varepsilon_7=0.82$ and
\renewcommand\theequation{\thesection.\arabic{equation}*}
\begin{align}
(x_1+1)^2+y(x_1)^2-\varepsilon_7^2&\Aeq{-0.0953}\ly 0. \retainlabel{equ:ell-1}
\end{align}
The end points of $\Gamma_2$: $\zeta_1$ and $\zeta_2$, are contained in $\D(a_4,\varepsilon_4)$, since
\renewcommand\theequation{\thesection.\arabic{equation}*}
\begin{align}
(x_1-\re a_4)^2+(y(x_1)-\im a_4)^2-\varepsilon_4^2 & \Aeq{-0.0061}\ly 0 \text{~~and}  \retainlabel{0.0061-x1}  \\
(x_2-\re a_4)^2+(y(x_2)-\im a_4)^2-\varepsilon_4^2 & \Aeq{-0.0087}\ly 0.  \retainlabel{0.0087-x2}
\end{align}
The end points of $\Gamma_3$: $\zeta_2$ and $\zeta_3$, are contained in $\D(\omega,\varepsilon_6)$, since
\renewcommand\theequation{\thesection.\arabic{equation}*}
\begin{align}
(x_2-\re\omega)^2+(y(x_2)-\im\omega)^2-\varepsilon_6^2 & \Aeq{-0.0027}\ly 0 \text{~~and}  \retainlabel{0.0027-x2}  \\
(x_3-\re\omega)^2+(y(x_3)-\im\omega)^2-\varepsilon_6^2 & \Aeq{-0.0292}\ly 0.  \retainlabel{0.0292-x3}
\end{align}
The end points of $\Gamma_4$: $\zeta_3$ and $x_4$, are contained in $\D(a_5,\varepsilon_5)$, since
\renewcommand\theequation{\thesection.\arabic{equation}*}
\begin{align}
(x_3-\re a_5)^2+(y(x_3)-\im a_5)^2-\varepsilon_5^2&\Aeq{-0.0206}\ly 0 \text{~and~} \retainlabel{0.0206-x3} \\
(x_4-\re a_5)^2+(\im a_5)^2-\varepsilon_5^2&\Aeq{-0.0615}\ly 0. \label{0.0615-x4a5}
\end{align}
Therefore the upper half of $\partial \mathscr{D}_{r_1}$ is contained in the union of $\D$, $\D(a_4,\varepsilon_4)$, $\D(a_5,\varepsilon_5)$, $\D(\omega,\varepsilon_6)$ and $\D(-1,\varepsilon_7)$.

Moreover, we need to verify that the interval $[x_0, x_4]$ is also covered by the union of these disks. Define $x_1'=-0.5$ and $x_2'=0.3$. Then
$[x_0,x_1']$ is contained in $\D(-1,\varepsilon_7)$ since $x_0\in\D(-1,\varepsilon_7)$ and $|x_1'+1|=0.5<\varepsilon_7=0.82$.
The interval $[x_1', x_2']$ is contained in $\D(a_4,\varepsilon_4)$, since
\renewcommand\theequation{\thesection.\arabic{equation}*}
\begin{align}
(x_1'-\re a_4)^2+(\im a_4)^2-\varepsilon_4^2 & \Aeq{-0.2300}\ly 0 \text{~~and}  \retainlabel{0.2300-x1p}  \\
(x_2'-\re a_4)^2+(\im a_4)^2-\varepsilon_4^2 & \Aeq{-0.0380}\ly 0.  \retainlabel{0.0380-x2p}
\end{align}
The interval $[x_2', x_4]$ is contained in $\D(a_5,\varepsilon_5)$, since \eqref{0.0615-x4a5} and
\renewcommand\theequation{\thesection.\arabic{equation}*}
\begin{align}
(x_2'-\re a_5)^2+(\im a_5)^2-\varepsilon_5^2&\Aeq{-0.0283}\ly 0. \retainlabel{equ:ell-4}
\end{align}
This implies that the closure of $\mathscr{D}_{r_1}\cap\BH_+$ is covered by the union of $\D(a_4,\varepsilon_4)$, $\D(a_5,\varepsilon_5)$, $\D(\omega,\varepsilon_6)$ and $\D(-1,\varepsilon_7)$.
By the symmetry we conclude that $\overline{\mathscr{D}}_{r_1}$ is contained in the union of the $7$ disks stated in the lemma.

\medskip
(d) Assume that $\zeta\in\overline{\MU}_1$. By Lemma \ref{lema-a4}, we have
$\zeta\not\in\OD(a_4,\varepsilon_4)\cup\OD(\overline{a}_4,\varepsilon_4)\cup\OD(a_5,\varepsilon_5)\cup\OD(\overline{a}_5,\varepsilon_5)$. If $|Q(\zeta)|\geq \rho$, we have $\zeta\not\in\OD(-1,\varepsilon_7)$ by Part (b). Note that $Q(\gamma_{b1})=Q(\gamma_{b3})=\Gamma_b^Q=(0,cv]$ and $|cv|<17<R=125$. By Part (a), this implies that $\zeta\not\in\OD(\omega,\varepsilon_6)\cup\OD(\overline{\omega},\varepsilon_6)$. By Part (c), we have $\zeta\in\EC\setminus\overline{\mathscr{D}}_{r_1}$.

\medskip
(e) Assume that $\zeta\in\overline{\MU}_{12}'$. By Lemma \ref{lema-a4}, we have
$\zeta\not\in\OD(a_4,\varepsilon_4)\cup\OD(\overline{a}_4,\varepsilon_4)\cup\OD(a_5,\varepsilon_5)\cup\OD(\overline{a}_5,\varepsilon_5)$.
If $\rho\leq |Q(\zeta)|\leq R$. By Parts (a) and (b), we have $\zeta\not\in\OD(\omega,\varepsilon_6)\cup\OD(\overline{\omega},\varepsilon_6)\cup\OD(-1,\varepsilon_7)$.  By Part (c), we have $\zeta\in\C\setminus\overline{\mathscr{D}}_{r_1}$.

\medskip
(f) The proof is completely similar to Part (e). We omit the details.
\end{proof}

\section{Estimates on $Q$: Part II}

\begin{lem}\label{lema-Q}
One can write
\begin{equation}
Q(\zeta)  =\zeta+b_0+\frac{b_1}{\zeta}+Q_2(\zeta),
\end{equation}
with
\begin{equation}
\begin{split}
Q_2(\zeta) =
&~ \frac{2^4}{5^5}\cdot\frac{a_{1,1}\zeta-a_{0,1}}{\zeta(\zeta-\omega)(\zeta-\overline{\omega})}
 -\frac{2^6}{5^{10}}\cdot\frac{a_{1,2}\zeta+a_{0,2}}{(\zeta-\omega)^2(\zeta-\overline{\omega})^2}\\
&~ +\frac{2^{11}}{5^{14}}\cdot\frac{a_{1,3}\zeta+a_{0,3}}{(\zeta-\omega)^3(\zeta-\overline{\omega})^3}
-\frac{2^{12}}{5^{16}}\cdot\frac{a_{1,4}\zeta+a_{0,4}}{(\zeta-\omega)^4(\zeta-\overline{\omega})^4},
\end{split}
\end{equation}
where
\begin{equation}
b_0=\frac{2(13+32\sqrt{6})}{25}\Aeq{7.31}, \quad b_1=\frac{2029+256\sqrt{6}}{125}\Aeq{21.24},
\end{equation}
and
\begin{align}
a_{1,1}&=2(617+688\sqrt{6}), & a_{0,1}&=25(119+16\sqrt{6}), \\
a_{1,2}&=3889250 + 837000\sqrt{6}, & a_{0,2}&=2755539+487396\sqrt{6}, \\
a_{1,3}&=31356325 + 8965425 \sqrt{6}, & a_{0,3}&=66811702 + 23697378\sqrt{6}, \\
a_{1,4}&=102142212 + 38104768\sqrt{6}, & a_{0,4}&=240990025 + 94826600\sqrt{6}.
\end{align}
For $|\zeta|\geq r>1$, we have
\begin{equation}
\begin{split}
|Q_2(\zeta)|\leq Q_{2,max}(r) :=
&~ \frac{2^4}{5^5}\cdot\frac{a_{1,1}r+a_{0,1}}{r(r-1)^2}
+\frac{2^6}{5^{10}}\cdot\frac{a_{1,2}r+a_{0,2}}{(r-1)^4}\\
&~ +\frac{2^{11}}{5^{14}}\cdot\frac{a_{1,3}r+a_{0,3}}{(r-1)^6}
+\frac{2^{12}}{5^{16}}\cdot\frac{a_{1,4}r+a_{0,4}}{(r-1)^8}.
\end{split}
\end{equation}
\end{lem}

\begin{proof}
The formula of $Q_2$ can be obtained by a calculation from \eqref{equ:Q-another-form}. For the estimate of $|Q_2(\zeta)|$, it is sufficient to notice that $|(\zeta-\omega)(\zeta-\overline{\omega})|\geq (|\zeta|-1)^2$ for $|\zeta|\geq r>1$ and that $r\mapsto r/(r-1)^n$ is decreasing in $(1,+\infty)$ for any $n\geq 2$.
\end{proof}

For $z_0\in\C$ and $\theta>0$, let $\V(z_0,\theta)$ be defined in \eqref{equ:V-z0-theta}.

\begin{lem}\label{lema-Q-subset-cv}
$Q\big(\OV(11,\frac{\pi}{6})\big)\subset \V (17,\frac{\pi}{6}) \subset \V (cv,\frac{\pi}{6})$.
\end{lem}

\begin{proof}
Suppose $\zeta\in\OV(11,\frac{\pi}{6})$ and let $\zeta'=\zeta+6$. Since $\zeta'\in\OV(17,\frac{\pi}{6})$, it is sufficient to show that $|\arg(Q(\zeta)-\zeta')|=\big|\arg\big(\frac{b_1}{\zeta}+(b_0-6+Q_2(\zeta))\big)\big|<\frac{\pi}{6}$. Since $\zeta\in\OV(11,\frac{\pi}{6})$, we have $|\arg\zeta|<\frac{\pi}{6}$ and $|\arg\frac{b_1}{\zeta}|<\frac{\pi}{6}$.
On the other hand, by Lemma \ref{lema-Q}, we have
\renewcommand\theequation{\thesection.\arabic{equation}*}
\begin{align}
 Q_{2,max}(11)\Aeq{0.2998}\ly 0.4. \label{0.3-Q2max}
\end{align}
Note that $b_0-6>1$. By Lemma \ref{lema-basic-esti}, we have
\begin{equation}
|\arg(b_0-6+Q_2(\zeta))|=\Big|\arg\Big(1+\tfrac{Q_2(\zeta)}{b_0-6}\Big)\Big|\leq \arcsin Q_{2,max}(11)<\tfrac{\pi}{3}\cdot 0.4<\tfrac{\pi}{6}.
\end{equation}
Since both $\frac{b_1}{\zeta}$ and $b_0-6+Q_2(\zeta)$ are contained in $\V(0,\frac{\pi}{6})$, we obtain $|\arg(Q(\zeta)-\zeta')|<\tfrac{\pi}{6}$. Therefore, $Q(\OV(11,\frac{\pi}{6}))\subset \V (17,\frac{\pi}{6}) \subset \V (cv,\frac{\pi}{6})$.
\end{proof}


\begin{lem}\label{lema-injective}
\begin{enumerate}
\item We have
\begin{equation}
Q'(\zeta)=\Big(1-\tfrac{cp}{\zeta}\Big)^2\Big(1-\tfrac{cp'}{\zeta}\Big)^2\,
\frac{(1-\frac{1}{\zeta})^3(1+\frac{1}{\zeta})^5}{(1-\frac{\omega}{\zeta})^5(1-\frac{\overline{\omega}}{\zeta})^5};
\end{equation}

\item If $|\zeta|\geq r>cp\Aeq{4.073}$, then
\begin{equation}
|\log Q'(\zeta)|\leq LogDQ_{max}(r):=\frac{b_1}{r^2}+\frac{50}{r^3}+\frac{cp^4}{2r^3\big(r-cp\big)};
\end{equation}

\item If $|\zeta|>5.6$, then $\re Q'(\zeta)>0$. In particular, $Q$ is injective in the half plane $\{\zeta\in\C: \re(\zeta e^{-\ii\theta})>5.6\}$ for any $\theta\in\R$.
\end{enumerate}
\end{lem}

\begin{proof}
(a) This can be obtained by a direct calculation from \eqref{equ:Q-another-form}.

\medskip
(b) Using $-\log(1-x)=x+\frac{x^2}{2}+\frac{x^3}{3}+\sum_{n=4}^\infty \frac{x^n}{n}$, we have
\begin{equation}
\begin{split}
&~\log Q'(\zeta)
=2\log\Big(1-\tfrac{cp}{\zeta}\Big)+2\log\Big(1-\tfrac{cp'}{\zeta}\Big)+3\log(1-\tfrac{1}{\zeta}) \\
 &~\qquad\qquad\quad +5\log\Big(1+\tfrac{1}{\zeta}\Big)- 5\log\Big(1-\tfrac{\omega}{\zeta}\Big)- 5\log\Big(1-\tfrac{\overline{\omega}}{\zeta}\Big)\\
=&~-\tfrac{b_1}{\zeta^2}-\tfrac{\widetilde{b}_2}{\zeta^3}-2\sum_{n=4}^\infty \tfrac{(cp)^n+(cp')^n}{n\zeta^n}
-\sum_{n=4}^\infty\tfrac{3+5(-1)^n}{n\zeta^n}+5\sum_{n=4}^\infty\tfrac{\omega^n+\overline{\omega}^n}{n\zeta^n},
\end{split}
\end{equation}
where
\begin{equation}
\widetilde{b}_2=\tfrac{2}{3}(cp^3+(cp')^3-1)-\tfrac{5}{3}(\omega^3+\overline{\omega}^3)=\tfrac{64(617+688\sqrt{6})}{3125}\Aeq{47.15}.
\end{equation}
Note that $|3+5(-1)^n|\leq 8$ and $|\omega^n+\overline{\omega}^n|\leq 2$ for all $n\in\N$. If $|\zeta|\geq r>cp$, then
\begin{equation}
|\log Q'(\zeta)|
\leq \frac{b_1}{r^2}+\frac{\widetilde{b}_2}{r^3}+\frac{1}{2}\left(\frac{\Big(\frac{cp}{r}\Big)^4}{1-\frac{cp}{r}}
+\frac{\Big(\frac{cp'}{r}\Big)^4}{1-\frac{cp'}{r}}\right)+\Big(2+\frac{5}{2}\Big)\frac{\frac{1}{r^4}}{1-\frac{1}{r}}.
\end{equation}
The estimate then follows since $r>4$.

\medskip
(c) Note that $\arg Q'(\zeta)=\im\big(\log Q'(\zeta)\big)$. If $|\zeta|>5.6$, by Part (b) we have
\renewcommand\theequation{\thesection.\arabic{equation}*}
\begin{align}
|\arg Q'(\zeta)|\leq |\log Q'(\zeta)|<LogDQ_{max}(5.6)\Aeq{1.47}\ly \tfrac{\pi}{2}. \retainlabel{equ:arg-Q-deri-1}
\end{align}
Thus we have $\re Q'(\zeta)>0$. If two distinct points $\zeta_0$ and $\zeta_1$ can be connected by a segment in $\{\zeta\in\C:|\zeta|>5.6\}$, then by $\tfrac{Q(\zeta_1)-Q(\zeta_0)}{\zeta_1-\zeta_0}=\int_0^1 Q(\zeta_0+(\zeta_1-\zeta_0)t)dt$ we have $\re \tfrac{Q(\zeta_1)-Q(\zeta_0)}{\zeta_1-\zeta_0}>0$. Hence $Q(\zeta_0)\neq Q(\zeta_1)$. This implies that $Q$ is injective in the half plane $\{\zeta\in\C: \re(\zeta e^{-\ii\theta})>5.6\}$.
\end{proof}

\section{Estimates on $\varphi$}

In the following, we assume that $F=Q\circ\varphi^{-1}\in\MF_1^Q$ unless otherwise specified. Then $\varphi:\EC\setminus\overline{\mathscr{D}}\to\EC\setminus\{0\}$ is a normalized univalent map, where $\mathscr{D}=\D(a_0,r_0)$. Let
\begin{equation}\label{equ:zeta}
\zeta(w)=r_0 w+a_0: \EC\setminus\OD\to\EC\setminus\overline{\mathscr{D}}.
\end{equation}
The following lemma gives various estimates of $\varphi$. The proof is based on applying distortion theorems of univalent maps. We refer to \cite[\S5.G]{IS08}.

\begin{lem}\label{lema-esti-varphi}
Suppose $\varphi:\EC\setminus\overline{\mathscr{D}}\to\EC\setminus\{0\}$ is a normalized univalent map, i.e., $\varphi(\infty)=\infty$ and $\lim_{\zeta\to\infty}\frac{\varphi(\zeta)}{\zeta}=1$. Hence $\varphi$ can be written as
\begin{equation}
\varphi(\zeta)=\zeta+c_0+\varphi_1(\zeta)
\end{equation}
with $c_0\in\C$ and $\lim_{\zeta\to\infty}\varphi_1(\zeta)=0$. Then the following estimates hold:
\begin{enumerate}
\item $|c_0-c_{00}|\leq c_{01,max}$, where
\begin{equation}
c_{00}:=-a_0=0.06 \text{\quad and\quad} c_{01,max}:=2r_0=2.14;
\end{equation}
\item $Image(\varphi)\supset\{z\in\C:\,|z-(c_0+a_0)|>2 r_0\}\supset\{z\in\C:\,|z|>4r_0=4.28\}$;
\item $r_0|w|\left(1-\frac{1}{|w|}\right)^2\leq |\varphi(\zeta(w))|\leq r_0|w|\left(1+\frac{1}{|w|}\right)^2$  for $|w|>1$;
\item $\left|\arg \frac{\varphi(\zeta(w))}{w}\right|\leq\log\frac{|w|+1}{|w|-1}$ for $|w|>1$;
\item $|\varphi_1(\zeta)|\leq\varphi_{1,max}(r):=r_0\sqrt{-\log\big(1-(\frac{r_0}{r-|a_0|})^2\big)}$ for $|\zeta|\geq r>r_0+|a_0|=1.13$;
\item $|\log\varphi'(\zeta)|\leq LogD\varphi_{max}(r):=-\log\big(1-(\frac{r_0}{r-|a_0|})^2\big)$ for $|\zeta|\geq r>r_0+|a_0|=1.13$.
\end{enumerate}
\end{lem}

\begin{lem}\label{lema-phi-rho}
If $\zeta\in\C\setminus \mathscr{D}_{r_1}$ with $r_1=1.2$, then $|\varphi(\zeta)|>\rho$ and $\left|\arg\frac{\varphi(\zeta)}{\zeta}\right|<\pi$.
\end{lem}

\begin{proof}
Note that $\zeta(w)=r_0 w+a_0$ is a conformal map sending $\{w\in\C:|w|=r\}$ onto $\partial \mathscr{D}_r$, where $\mathscr{D}_r=\D(a_0,r_0 r)$ and $r>1$. Let $\zeta\in\C\setminus \mathscr{D}_{r_1}$.
Then one can write $\zeta=\zeta(w)$ with $|w|\geq r_1$. By Lemma \ref{lema-esti-varphi}(c), we have
\begin{equation}
|\varphi(\zeta)|
\geq r_0|w|\Big(1-\tfrac{1}{|w|}\Big)^2
\geq  1.07\times 1.2\,\Big(1-\tfrac{1}{1.2}\Big)^2=\tfrac{1.07}{30}>0.03=\rho.
\end{equation}
By Lemma \ref{lema-esti-varphi}(d), we have
$
\big|\arg\frac{\varphi(\zeta(w))}{w}\big|\leq\log\frac{1.2+1}{1.2-1}=\log 11.
$
By Lemma \ref{lema-basic-esti}(a), we have
$\big|\arg\tfrac{\zeta(w)}{w}\big|=\big|\arg\big(1+\tfrac{a_0}{r_0 w}\big)\big|
<\arcsin(|a_0|)=\arcsin(0.06).
$
Therefore,
\renewcommand\theequation{\thesection.\arabic{equation}*}
\begin{align}
\left|\arg\tfrac{\varphi(\zeta)}{\zeta}\right|
\leq  \left|\arg\tfrac{\varphi(\zeta(w))}{w}\right|+\left|\arg\tfrac{\zeta(w)}{w}\right|
< \log 11+\tfrac{0.06\pi}{3}\Aeq{2.46}\ly \pi. \retainlabel{246-pi}
\end{align}
The proof is finished.
\end{proof}

Let $\D(a_5,\varepsilon_5)$ and $\D(\omega,\varepsilon_6)$ be the disks defined in Lemmas \ref{lema-a4} and \ref{lema-E-r1} respectively. The following result will be used in \S\ref{sec-bounding-dom}.

\begin{lem}[see Figure \ref{Fig-Omega0}]\label{lema-arg-arg}
Let $h_0=\re\omega+\im\omega+\varepsilon_6=\frac{14\sqrt{6}}{25}+0.45\Aeq{1.82}$ and denote
\begin{equation}
\begin{split}
\Omega_0:=&~\{\zeta\in\C:\re \zeta>\re\omega \text{~or~}|\im \zeta|> h_0+1 \text{~or~} |\im\zeta|>-\re\zeta+h_0\} \\
&~\setminus\big(\OD(a_5,\varepsilon_5)\cup\OD(\overline{a}_5,\varepsilon_5)\cup\OD(\omega,\varepsilon_6)\cup\OD(\overline{\omega},\varepsilon_6)\big).
\end{split}
\end{equation}
If $\zeta\in\overline{\Omega}_0$, then $\varphi(\zeta)\not\in\R_-$.
\end{lem}

\begin{figure}[!htpb]
  \setlength{\unitlength}{1mm}
  \centering
  \includegraphics[width=0.96\textwidth]{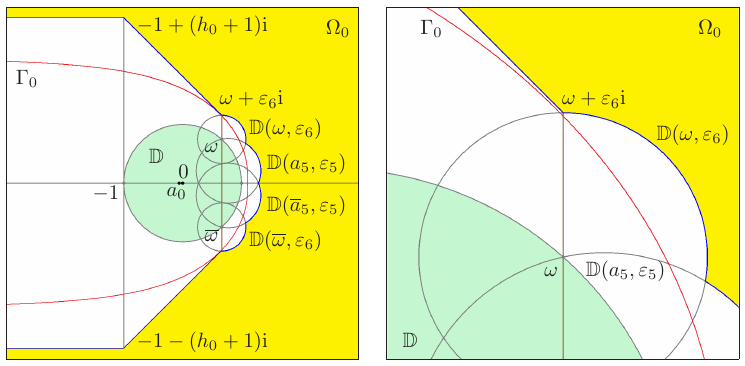}
  \caption{Left: The domain $\Omega_0$ (in yellow) and the curve $\Gamma_0$ (in red). To prove Lemma \ref{lema-arg-arg}, we show that it suffices to prove the property $\Gamma_0\cap\overline{\Omega}_0=\emptyset$. Right: A zoom of the left picture near $\omega+\varepsilon_6\ii$.}
  \label{Fig-Omega0}
\end{figure}

\begin{proof}
By Lemma \ref{lema-esti-varphi}(d), for $|w|>1$, then
$\left|\arg(\varphi(\zeta(w)))\right|\leq  |\arg w|+\left|\arg\tfrac{\varphi(\zeta(w))}{w}\right|\leq |\arg w|+\log\tfrac{|w|+1}{|w|-1}$.
For $\zeta\in\C\setminus\overline{\mathscr{D}}$, we write $\zeta=\zeta(w)=r_0 w+a_0$ with $|w|=r>1$ and $\theta=\arg w\in(-\pi,\pi]$.
We define
\begin{equation}
\Gamma_0:=\{\zeta(re^{\ii\theta})\in\C: \theta\in(-\pi,\pi) \text{ and } r>1 \text{ satisfy } |\theta|+\log\tfrac{r+1}{r-1}=\pi\}.
\end{equation}
It is easy to see that for the points in $\Gamma_0$, as $\theta$ moves from $0$ to $\pi$ monotonously, the modulus $r$ increases strictly from $\frac{e^\pi+1}{e^\pi-1}$ to $+\infty$. Indeed, the following map is strictly increasing on $[0,\pi)$:
\begin{equation}
r_2(\theta):=\tfrac{e^{\pi-\theta}+1}{e^{\pi-\theta}-1}.
\end{equation}

By the definition of $\Omega_0$, one can verify that the intersection of the line $\{\zeta(re^{\ii\theta})=a_0+r_0 re^{\ii\theta}:r>0\}$ with $\partial\Omega_0$ is a singleton
$a_0+r_0 r_3 (\theta) e^{\ii\theta}$ for all $\theta\in[0,\pi)$. The specific formula of $r_3(\theta)$ is a piecewise continuous function which can be seen in the following arguments. So in order to prove the lemma, it suffices to show that $r_2(\theta)<r_3(\theta)$ for all $\theta\in[0,\pi)$.

\medskip
Define $\alpha_1=\arctan\frac{\im a_5}{\re a_5-a_0}=\arctan\frac{1}{4}\Aeq{0.2449}$, $\alpha_2^-=0.54$, $\alpha_2^+=0.55$, $\alpha_3=\arctan\frac{\im\omega}{\re\omega-a_0}\Aeq{0.8017}$, $\alpha_4=\arctan\frac{\im\omega+\varepsilon_6}{\re\omega-a_0}\Aeq{1.0120}$, $\alpha_5=\pi-\arctan\frac{h_0+1}{a_0+1}\Aeq{1.8923}$. We divide the arguments into several cases.

\medskip
(\textbf{Case 1})  Let $\theta\in[0,\alpha_1]$. In order to find the formula of $r_3(\theta)$, it suffices to notice that $a_0+r_0 r_3 (\theta) e^{\ii\theta}=(a_0+r_0 r_3(\theta)\cos\theta, r_0 r_3(\theta)\sin\theta)$ lies on the circle $\partial\D(a_5,\varepsilon_5)$. A direct calculation shows that
\begin{equation}
r_3(\theta):=\widetilde{r}_5(\theta)=\tfrac{1}{r_0}\Big(B_5(\theta)+\sqrt{B_5(\theta)^2-C_5(\theta)}\Big),
\end{equation}
where $B_5(\theta)=\im a_5 \sin\theta+(\re a_5-a_0)\cos\theta$ and $C_5(\theta)=(\re a_5-a_0)^2+\im a_5^2-\varepsilon_5^2$.
By the definition of $\alpha_1$, it follows that $r_3(\theta)$ increases on $[0,\alpha_1]$.
Hence for any $\theta\in[0,\alpha_1]$, we have
\renewcommand\theequation{\thesection.\arabic{equation}*}
\begin{align}
r_3(\theta)-r_2(\theta)\geq \widetilde{r}_5(0)-r_2(\alpha_1)\Aeq{0.1435}\gy 0. \retainlabel{0.0069-Upsilon}
\end{align}

\medskip
(\textbf{Case 2})  Let $\theta\in[\alpha_1,\alpha_2^+]$. We denote $r_5^+(\theta):=\widetilde{r}_5(\theta)$.
Then $a_0+r_0 r_5^+(\alpha_2^+)e^{\ii\alpha_2^+}\in\D(\omega,\varepsilon_6)$ since
\renewcommand\theequation{\thesection.\arabic{equation}*}
\begin{align}
&~(a_0+r_0 r_5^+(\alpha_2^+)\cos\alpha_2^+ -\re\omega)^2 \\
+&~(r_0 r_5^+(\alpha_2^+)\sin\alpha_2^+ -\im\omega)^2-\varepsilon_6^2\Aeq{-0.0091}\ly 0. \label{0.0091-Upsilon}
\end{align}
Let $\zeta_{5,6}\in\partial\Omega_0\cap\partial\D(a_5,\varepsilon_5)\cap\partial\D(\omega,\varepsilon_6)$ and denote $\alpha_2:=\arctan\frac{\im\zeta_{5,6}}{\re\zeta_{5,6}-a_0}$. By \eqref{0.0091-Upsilon} we have $\alpha_2<\alpha_2^+$. So in order to prove $r_3(\theta)>r_2(\theta)$ for all $\theta\in[\alpha_1,\alpha_2]$, it suffices to show that $r_5^+(\theta)>r_2(\theta)$ for all $\theta\in[\alpha_1,\alpha_2^+]$.
Note that $r_5^+(\theta)$ decreases on $[\alpha_1,\alpha_2^+]$. For $\theta\in[\alpha_1,\alpha_2^+]$, we have
\renewcommand\theequation{\thesection.\arabic{equation}*}
\begin{align}
r_5^+(\theta)-r_2(\theta)\geq r_5^+(\alpha_2^+)-r_2(\alpha_2^+)\Aeq{0.0631}\gy 0. \retainlabel{0.0631-Upsilon}
\end{align}

\medskip
(\textbf{Case 3})  Let $\theta\in[\alpha_3,\alpha_4]$. Consider the condition that $a_0+r_0 r_3 (\theta) e^{\ii\theta}=(a_0+r_0 r_3(\theta)\cos\theta, r_0 r_3(\theta)\sin\theta)\in\partial\D(\omega,\varepsilon_6)$. A direct calculation shows that
\begin{equation}
r_3(\theta):=\widetilde{r}_6(\theta)=\tfrac{1}{r_0}\big(B_6(\theta)+\sqrt{B_6(\theta)^2-C_6(\theta)}\big),
\end{equation}
where $B_6(\theta)=\im \omega \sin\theta+(\re \omega-a_0)\cos\theta$ and $C_6(\theta)=(\re \omega-a_0)^2+\im \omega^2-\varepsilon_6^2$.
By the definition of $\alpha_3$, it follows that $r_3(\theta)$ decreases on $[\alpha_3,\alpha_4]$.
Hence for any $\theta\in[\alpha_3,\alpha_4]$, we have
\renewcommand\theequation{\thesection.\arabic{equation}*}
\begin{align}
r_3(\theta)-r_2(\theta)\geq \widetilde{r}_6(\alpha_4)-r_2(\alpha_4)\Aeq{0.0062}\gy 0. \retainlabel{0.0062-Upsilon}
\end{align}

\medskip
(\textbf{Case 4})  Let $\theta\in[\alpha_2^-,\alpha_3]$. We denote $r_6^-(\theta):=\widetilde{r}_6(\theta)$.
Then $a_0+r_0 r_6^-(\alpha_2^-)e^{\ii\alpha_2^-}\in\D(a_5,\varepsilon_5)$ since
\renewcommand\theequation{\thesection.\arabic{equation}*}
\begin{align}
&~(a_0+r_0 r_6^-(\alpha_2^-)\cos\alpha_2^- -\re a_5)^2 \\
+&~(r_0 r_6^-(\alpha_2^-)\sin\alpha_2^- -\im a_5)^2-\varepsilon_5^2\Aeq{-0.0031}\ly 0. \label{0.0091-Upsilon-1}
\end{align}
Let $\alpha_2$ be defined in (Case 2). By \eqref{0.0091-Upsilon-1} we have $\alpha_2^-<\alpha_2$. So in order to prove $r_3(\theta)>r_2(\theta)$ for all $\theta\in[\alpha_2,\alpha_3]$, it suffices to prove that $r_6^-(\theta)>r_2(\theta)$ for all $\theta\in[\alpha_2^-,\alpha_3]$.
Note that $r_6^-(\theta)$ increases on $[\alpha_2^-,\alpha_3]$. For $\theta\in[\alpha_2^-,\alpha_3]$, we have
\renewcommand\theequation{\thesection.\arabic{equation}*}
\begin{align}
r_6^-(\theta)-r_2(\theta)\geq r_6^-(\alpha_2^-)-r_2(\alpha_3)\Aeq{0.0152}\gy 0. \retainlabel{0.0152-Upsilon-1}
\end{align}

\medskip
(\textbf{Case 5})  Let $\theta\in[\alpha_4,\alpha_5]$. Consider the condition that $a_0+r_0 r_3 (\theta) e^{\ii\theta}=(a_0+r_0 r_3(\theta)\cos\theta, r_0 r_3(\theta)\sin\theta)$ lies on the line $\{\zeta\in\C:\re\zeta+\im\zeta=h_0\}$. Then
\begin{equation}
r_3(\theta)=\frac{h_0-a_0}{\sqrt{2}r_0\sin(\theta+\tfrac{\pi}{4})}.
\end{equation}
Note that $\tfrac{\pi}{4}<\alpha_4<\alpha_5<\tfrac{3\pi}{4}$. It follows that $r_3(\theta)$ increases on $[\alpha_4,\alpha_5]$.
Then for any $\theta\in[\theta',\theta'']\subset[\alpha_4,\alpha_5]$, we have $r_3(\theta)-r_2(\theta)\geq r_3(\theta')-r_2(\theta'')$. Define $\theta_{4,1}=\alpha_4$, $\theta_{4,k}=1+0.01k$ for $2\leq k\leq 10$, $\theta_{4,k}=1.1+0.02(k-10)$ for $11\leq k\leq 15$, $\theta_{4,k}=1.2+0.05(k-15)$ for $16\leq k\leq 23$ and $\theta_{4,24}=\alpha_5$.  One can verify that for $1\leq k\leq 23$, then
\renewcommand\theequation{\thesection.\arabic{equation}*}
\begin{align}
r_3(\theta_{4,k})-r_2(\theta_{4,k+1})\gy 0.002>0. \retainlabel{0.0062-Upsilon-2}
\end{align}
This implies that $r_3(\theta)>r_2(\theta)$ for all $\theta\in[\alpha_4,\alpha_5]$.

\medskip
(\textbf{Case 6})  Let $\theta\in[\alpha_5,\pi)$. Consider the condition that $a_0+r_0 r_3 (\theta) e^{\ii\theta}=(a_0+r_0 r_3(\theta)\cos\theta, r_0 r_3(\theta)\sin\theta)$ lies on the line $\{\zeta\in\C:\im\zeta=h_0+1\}$. Then
\begin{equation}
r_3(\theta)=\frac{h_0+1}{r_0\sin\theta}.
\end{equation}
Note that $\alpha_5>\tfrac{\pi}{2}$. Hence $r_3(\theta)$ increases on $[\alpha_5,\pi)$. Similar to (Case 5), for any $\theta\in[\theta',\theta'']\subset[\alpha_5,\pi)$, we have $r_3(\theta)-r_2(\theta)\geq r_3(\theta')-r_2(\theta'')$. Define $\theta_{5,1}=2.38$ and $\theta_{5,2}=2.6$. Then
\renewcommand\theequation{\thesection.\arabic{equation}*}
\begin{align}
r_3(\alpha_5)-r_2(\theta_{5,1})\Aeq{0.0277}\gy 0, ~ r_3(\theta_{5,1})-r_2(\theta_{5,2})\Aeq{0.0388}\gy 0. \retainlabel{0.0388-Upsilon-3}
\end{align}
This implies that $r_3(\theta)>r_2(\theta)$ for all $\theta\in[\alpha_5, 2.6]$.

If $\theta\in[2.6,\pi)$, then $t=\pi-\theta\in(0,\pi-2.6]\subset(0,0.6]$. We have
\begin{equation}
r_3(\theta)-r_2(\theta)= \frac{h_0+1}{r_0 \sin t}-\frac{2}{e^t-1}-1\geq \Big(\frac{h_0+1}{r_0}-2\Big)\frac{1}{t}-1
>\frac{2.8-2.14}{1.07\times 0.6}-1>0.
\end{equation}
Hence $r_3(\theta)>r_2(\theta)$ for all $\theta\in[2.6,\pi)$ and the proof is complete.
\end{proof}

\section{Lifting $Q$ and $\varphi$ to $X$} \label{sec-lift-Q}

Recall that $\MU_{12}'=\MU_1\cup\MU_{2-}\cup\MU_{3+}\cup\gamma_{b1}\cup\gamma_{b3}$ (see Figure \ref{Fig-dom-Q}).

\begin{defi}
Denote $Y_{j\pm}=(Q|_{\overline{\MU}_{j\pm}})^{-1}(\pi_X(X_{j\pm}))$ for $j=1,2$, except $Y_{2+}$ is defined by $Y_{2+}=(Q|_{\overline{\MU}_{3+}})^{-1}(\pi_X(X_{2+}))$ (see Figures \ref{Fig-Riemann-X} and \ref{Fig-dom-Q}). Let
\begin{equation}
Y:=Y_{1+}\cup Y_{1-}\cup Y_{2+}\cup Y_{2-},
\end{equation}
which is subset of $\MU_{12}'\cup\R_-\subset\C$. Define $\widetilde{Q}:Y\to X$ (whose well-definedness will be verified) by
\begin{equation}
\widetilde{Q}(\zeta)=(\pi|_{X_{j\pm}})^{-1}(Q(\zeta))\in X_{j\pm} \text{~for~} \zeta\in Y_{j\pm}.
\end{equation}
We define
\begin{equation}
\widetilde{Y}:=\C\setminus(\overline{\mathscr{D}}_{r_1}\cup\R_+\cup\OV(11,\tfrac{\pi}{6})).
\end{equation}
\end{defi}

\begin{proof}[Proof of Proposition \ref{prop-lift}]
(a) The proof is almost the same as that in \cite[\S5.H]{IS08}. It suffices to check the consistency along the boundaries of $Y_{j\pm}$ since $\widetilde{Q}$ maps $Y_{j\pm}$ homeomorphically onto $X_{j\pm}$ for $j=1,2$. We omit the details here.

\medskip
(b) By Lemma \ref{lema-E-r1}(d)(e) and Lemma \ref{lema-Q-subset-cv}, one can obtain $Y\subset\widetilde{Y}$ by a completely similar argument to \cite[Lemma 5.26]{IS08}. Hence it is sufficient to prove that $\varphi|_{\widetilde{Y}}$ can be lifted to a well-defined holomorphic map $\widetilde{\varphi}:\widetilde{Y}\to X$.

Let $\zeta\in\tuta Y$. By Lemma \ref{lema-phi-rho}, we have $\Big|\arg\frac{\varphi(\zeta)}{\zeta}\Big|<\pi$ and $|\varphi(\zeta)|>\rho$ for $\zeta\in\tuta Y$ since $\tuta Y\subset\C\setminus\overline{\mathscr{D}}_{r_1}$. We define $\tuta\varphi(\zeta)\in X$ such that $\pi_X(\tuta\varphi(\zeta))=\varphi(\zeta)$ and
\begin{equation}
\tuta\varphi(\zeta)\in
\left\{
\begin{array}{l}
X_{1+}\cup X_{2-} \quad \text{~~if~} \im \zeta\geq 0 \text{~and~} -\pi<\arg\tfrac{\varphi(\zeta)}{\zeta}\leq 0,\\
X_{1-}\cup X_{2+} \quad \text{~~if~} \im \zeta< 0 \text{~and~} 0\leq\arg\tfrac{\varphi(\zeta)}{\zeta}<\pi,\\
X_{1+}\cup X_{1-} \quad \text{~~otherwise}.
\end{array}
\right.
\end{equation}
We first prove that $\tuta\varphi$ is well-defined.

For the first case, suppose $\im \zeta\geq 0$ and $\varphi(\zeta)\in H=\{w\in\C:\arg\zeta-\pi<\arg w\leq\arg\zeta\}$.
By Lemma \ref{lema-esti-varphi}(a)(e), if $|\zeta|\geq 11$, then
\renewcommand\theequation{\thesection.\arabic{equation}*}
\begin{equation}\label{equ-varphi-est}
\begin{split}
|\varphi(\zeta)-\zeta|
\leq &~c_{00}+c_{01,max}+\varphi_{1,max}(11)\Aeq{2.304} \\
\ly &~(cv-11)\sin\tfrac{\pi}{6}\Aeq{2.973}.
\end{split}
\end{equation}
Note that $\zeta\in\C\setminus\OV(11,\frac{\pi}{6})$. We have $\varphi(\zeta)\not\in\OV(cv,\frac{\pi}{6})$ since the Euclidean distance between $\partial\OV(11,\frac{\pi}{6})$ and $\OV(cv,\frac{\pi}{6})$ is $(cv-11)\sin\tfrac{\pi}{6}$. We consider the following two subcases:
\begin{itemize}
\item If further $|\zeta|\geq 50$, since $\varphi_{1,max}(11)>\varphi_{1,max}(50)$, we have $\varphi(\zeta)\not\in\OV(cv,\frac{\pi}{6})$, $|\varphi(\zeta)-\zeta|<3$ and $|\varphi(\zeta)|\geq 47$. Since the distance between the two points $\partial\D(0,47)\cap\partial\V(cv,\tfrac{\pi}{6})$ is larger than $3$ and $\varphi(\zeta)\in H$, we have $\im \varphi(\zeta)\geq 0$. This implies that $\tuta\varphi(\zeta)$ should be defined in $X_{1+}$.
\item If further $|\zeta|< 50$, then $|\varphi(\zeta)|<53<R=125$ and $\tuta\varphi(\zeta)$ should be defined in $X_{1+}\cup X_{2-}$.
\end{itemize}
Therefore, $\tuta\varphi$ is well-defined for the first case.

Similar argument can be applied to the second case $\im \zeta< 0$ and $0\leq\arg\tfrac{\varphi(\zeta)}{\zeta}<\pi$. For the continuity of $\tuta\varphi$, it is sufficient to consider the switching case: $\im\zeta=0$ or $\arg\tfrac{\varphi(\zeta)}{\zeta}=0$. The continuity follows by definition. Once the continuity of $\tuta\varphi$ is obtained, it is holomorphic.
\end{proof}

\section{Estimates on $F$}

Recall that $b_0$ and $b_1$, $c_{00}$ and $c_{01,max}$ are constants introduced in Lemma \ref{lema-Q} and Lemma \ref{lema-esti-varphi} respectively.
\begin{lem}\label{lema-esti-F}
Suppose $r\geq r'=6.1$ and $\re (\zeta e^{-\ii\theta})>r$ with $\theta\in\R$. Then the following estimates hold for $z=\varphi(\zeta)$:
\begin{enumerate}
\item $F(z)-z\in\D\big(b_0-c_{00}+\frac{b_1 e^{-\ii\theta}}{2r},\beta_{max}(r)\big)$, where
\begin{equation}
\beta_{max}(r):=c_{01,max}+\frac{b_1}{2r}+Q_{2,max}(r)+\varphi_{1,max}(r);
\end{equation}

\item $Arg\Delta F_{min}(r,\theta)\leq\arg(F(z)-z)\leq Arg\Delta F_{max}(r,\theta)$, where
\begin{equation}
Arg\Delta F_{\left\{\substack{max \\ min}\right\}}(r,\theta):= -\arctan\big(\sigma_1(r,\theta)\big)
\pm\arcsin\left(\frac{\beta_{max}(r)}{\sigma_2(r,\theta)}\right)
\end{equation}
and
\begin{equation}
\begin{split}
\sigma_1(r,\theta)=&~\frac{\frac{b_1\sin\theta}{2r}}{b_0-c_{00}+\frac{b_1\cos\theta}{2r}},\\
\sigma_2(r,\theta)=&~\sqrt{\left(b_0-c_{00}\right)^2+\left(\tfrac{b_1}{2r}\right)^2+2\left(b_0-c_{00}\right)\left(\tfrac{b_1}{2r}\right)\cos\theta};
\end{split}
\end{equation}

\item $Abs\Delta F_{min}(r,\theta)\leq|F(z)-z|\leq Abs\Delta F_{max}(r,\theta)$, where
\begin{equation}
Abs\Delta F_{\left\{\substack{max \\ min}\right\}}(r,\theta):= \sigma_2(r,\theta)\pm\beta_{max}(r);
\end{equation}

\item $|\log F'(z)|\leq LogDF_{max}(r):=LogDQ_{max}(r)+LogD\varphi_{max}(r)$.
\end{enumerate}
\end{lem}

\begin{proof}
If $r\geq r'=6.1$, then
\renewcommand\theequation{\thesection.\arabic{equation}*}
\begin{align}
b_0-c_{00}-\frac{b_1}{2\cdot r'}\Aeq{5.5090}\gy \beta_{max}(r')\Aeq{5.5046}.\retainlabel{5.5046-IS}
\end{align}
This is the only place that needs to be verified in the proof of \cite[Lemma 5.27]{IS08}. We will not include a proof here since the proof is completely the same.
\end{proof}

\begin{lem}\label{lema-F-cv}
\begin{enumerate}
\item $F(\OV(cv,\frac{\pi}{6}))\subset\V(17,\frac{\pi}{6})\subset\V(cv,\frac{\pi}{6})$;
\item $F(cv)\in\D(25.5,3)$, $\OV(F(cv),\frac{7\pi}{20})\subset \V(22,\frac{7\pi}{20})$, $\OV(F(cv),\frac{13\pi}{20})\subset \V(22,\frac{13\pi}{20})$; and
\item $F^{\circ 2}(cv)\in\V(F(cv),\frac{7\pi}{20})$.
\end{enumerate}
\end{lem}

\begin{proof}
(a) By the estimate \eqref{equ-varphi-est}, if $\zeta\in\C\setminus\OV(11,\tfrac{\pi}{6})$, then $\varphi(\zeta)\not\in\OV(cv,\frac{\pi}{6})$. This implies that $\varphi^{-1}(\OV(cv,\frac{\pi}{6}))\subset\OV(11,\frac{\pi}{6})$. By Lemma \ref{lema-Q-subset-cv}, we have $F(\OV(cv,\frac{\pi}{6}))$ $=Q\circ\varphi^{-1}(\OV(cv,\frac{\pi}{6}))\subset Q(\OV(11,\frac{\pi}{6}))\subset\V(17,\frac{\pi}{6})\subset\V(cv,\frac{\pi}{6})$.

\medskip
(b) Based on $\varphi(\zeta)=\zeta+c_0+\varphi_1(\zeta)$ and Lemma \ref{lema-esti-varphi}(b), we write
\begin{equation}\label{equ:varphi-inv}
\varphi^{-1}(z)=z-c_0+\phi_1(z),
\end{equation}
where $\varphi^{-1}$ is a univalent map defined in $\{z\in\C:|z|>4r_0\}$ with $\lim_{z\to\infty}\phi_1(z)=0$. Similar to Lemma \ref{lema-esti-varphi}(e), for $|z|\geq r> 4r_0=4.28$, we have
\begin{equation}\label{equ:phi-1-est}
|\phi_1(z)|\leq\phi_{1,max}(r):=4 r_0\sqrt{-\log\Big(1-(\tfrac{4r_0}{r})^2\Big)}.
\end{equation}
Then by Lemma \ref{lema-esti-varphi}(a), if $|z|\geq cv-2.35$, we have
\renewcommand\theequation{\thesection.\arabic{equation}*}
\begin{equation}\label{equ-varphi-inv-est}
|\varphi^{-1}(z)-z|\leq c_{00}+c_{01,max}+\phi_{1,max}(cv-2.35)\Aeq{3.483}\ly 3.5.
\end{equation}

We improve this bound as follows.
Note that $\big|\log|\varphi'(\zeta)|\big|\leq |\log\varphi'(\zeta)|$ and $|\varphi'(\zeta)|=\exp(\log|\varphi'(\zeta)|)$. By Lemma \ref{lema-esti-varphi}(f), if $|\zeta|\geq r>r_0+|a_0|=1.13$, we have
\begin{equation}\label{equ:varphi-deriv}
\exp(-LogD\varphi_{max}(r))\leq|\varphi'(\zeta)|\leq\exp(LogD\varphi_{max}(r)).
\end{equation}
By \eqref{equ-varphi-est}, if $|\zeta|\geq cv-5.9>11$, then $|\varphi(\zeta)-\zeta| \leq c_{00}+c_{01,max}+\varphi_{1,max}(cv-5.9)<c_{00}+c_{01,max}+\varphi_{1,max}(11)<2.35$.
By \eqref{equ-varphi-inv-est} and Lemma \ref{lema-esti-varphi}(f), if $|z|\geq cv$, then
\renewcommand\theequation{\thesection.\arabic{equation}*}
\begin{align}
&|\varphi^{-1}(z)-z| =|\varphi^{-1}(z)-\varphi^{-1}(\varphi(z))|\\
\leq~&|\varphi(z)-z|\cdot\max_{\xi\in\overline{\D}(z,2.35)}|(\varphi^{-1})'(\xi)|=|\varphi(z)-z|\cdot\max_{\xi\in\overline{\D}(z,2.35)}\frac{1}{|\varphi'(\varphi^{-1}(\xi))|} \\
\leq~&\frac{2.35}{\exp\big(-LogD\varphi_{max}(cv-2.35-3.5)\big)}\Aeq{2.37}\ly 2.4. \label{2.3-2.4-cv}
\end{align}
Therefore, if $|z|\geq cv$, by Lemma \ref{lema-injective},
\renewcommand\theequation{\thesection.\arabic{equation}*}
\begin{align}
&~|F(z)-Q(z)|=|Q(\varphi^{-1}(z))-Q(z)| \\
<~&2.4\cdot\exp(LogDQ_{max}(cv-2.4)) \Aeq{2.708}\ly 2.75. \retainlabel{2.703-2.8}
\end{align}
Then we have
\renewcommand\theequation{\thesection.\arabic{equation}*}
\begin{align}
|F(cv)-25.5|< 2.75+|Q(cv)-25.5|\Aeq{2.835}\ly 3. \retainlabel{2.835-Fcv}
\end{align}
Since
\renewcommand\theequation{\thesection.\arabic{equation}*}
\begin{align}
(25.5-22)\sin(\tfrac{7\pi}{20})\Aeq{3.11}\gy 3, \retainlabel{22.13-22}
\end{align}
we have $\OV(F(cv),\frac{7\pi}{20})\subset \V(22,\frac{7\pi}{20})$ and $\OV(F(cv),\frac{13\pi}{20})\subset \V(22,\frac{13\pi}{20})$.

\medskip
(c) Note that $|F(cv)|\geq 25.5-3>cv$ by Part (b). According to \eqref{2.3-2.4-cv}, we have $\zeta:=\varphi^{-1}(F(cv))\in\D(F(cv), 2.4)$. Then $|\zeta|\geq |F(cv)|-2.4> 25.5-3-2.4>20$. Let $b_0$ and $b_1$ be constants in Lemma \ref{lema-Q}. Then
\begin{align}
\left|\tfrac{b_1}{\zeta}+Q_2(\zeta)\right|\leq \tfrac{b_1}{20}+Q_{2,max}(20)\Aeq{1.136}\ly 1.2. \retainlabel{20-b1-zeta}
\end{align}
Hence $F^{\circ 2}(cv)=Q(\zeta)\in\D(\zeta+b_0,1.2)\subset\D(F(cv)+b_0,3.6)$. Since $b_0\sin\frac{7\pi}{20}>7.3\cdot\sin\frac{\pi}{4}>\frac{7.3}{2}>3.6$, we have $F^{\circ 2}(cv)\in\V(F(cv),\frac{7\pi}{20})$.
\end{proof}

\section{Repelling Fatou coordinate $\widetilde{\Phi}_{rep}$ on $X$}~ \label{sec-rep}

\begin{proof}[Proof of Proposition \ref{prop-rep-Fatou}]
On the Riemann surface $X$, we have $F\circ\pi_X\circ g=Q\circ\varphi^{-1}\circ\pi_X\circ\tuta\varphi\circ\tuta Q^{-1}=Q\circ\tuta Q^{-1}=\pi_X$ (see Figure \ref{Fig-commu-diagram}). In a small neighborhood of $\infty$, $F$ has an inverse branch $\overline{g}(z)=z-(b_0-c_0)+o(1)$ as $z\to\infty$. By Lemma \ref{lema-basic-esti} and Lemma \ref{lema-esti-varphi}(a), we have
\begin{equation}
\big|\arg(b_0-c_0)\big|
=\left|\arg\left(1+\frac{c_{00}-c_0}{b_0-c_{00}}\right)\right|
\leq \frac{\pi}{3}\cdot\frac{2.14}{7.3-0.06}<\frac{\pi}{9}.
\end{equation}
Let $L>0$ be a large number such that $\overline{g}$ is defined and injective in $W=\C\setminus\OV(-L,\frac{\pi}{9})$, and moreover, it satisfies $|\arg(\overline{g}(z)-z)-\pi|<\frac{\pi}{9}$, $\overline{g}(W)\subset W$ and $\re \overline{g}(z)<\re z-(b_0-c_{00})+c_{01,max}+1<\re z-4$. The rest proof can be obtained by a word for word copy of \cite[\S5.J]{IS08}.
In the following we only verify that in $\{z\in\C:\re z<-2R\}$, $F=Q\circ\varphi^{-1}$ has a repelling Fatou coordinate $\Phi_{rep}$ which is injective.

\medskip
Recall that $Q(\zeta)=\zeta+b_0+\frac{b_1}{\zeta}+Q_2(\zeta)$. We denote $\zeta:=\varphi^{-1}((b_0-c_0)z)$. By \eqref{equ:varphi-inv}, we have
\begin{equation}
\widehat{F}(z):=\frac{F((b_0-c_0)z)}{b_0-c_0}=z+1+\frac{\phi_1((b_0-c_0)z)+\frac{b_1}{\zeta}+Q_2(\zeta)}{b_0-c_0}.
\end{equation}
In the following we prove that $|\widehat{F}(z)-(z+1)|<\frac{1}{4}$ and $|\widehat{F}'(z)-1|<\frac{1}{4}$ for $|z|\geq \widehat{r}=25$.

By Lemma \ref{lema-esti-varphi}(a), we have $|c_0-0.06|\leq 2.14$. Since $b_0\in[7.31,7.32]$, thus $b_0-c_0\in\OD(7.25,2.15)$ and $5< |b_0-c_0|< 9.5$. For $|z|\geq \widehat{r}$, we have $|(b_0-c_0)z|\geq 5\widehat{r}$. By \eqref{2.3-2.4-cv}, $|\zeta|\geq 5\widehat{r}-3=122$. By Lemma \ref{lema-Q} and \eqref{equ:phi-1-est}, we have
\renewcommand\theequation{\thesection.\arabic{equation}*}
\begin{align}
|\widehat{F}(z)-(z+1)|\leq \frac{\phi_{1,max}(125)+\frac{b_1}{122}+Q_{2,max}(122)}{5}\Aeq{0.064}\ly\frac{1}{4}. \retainlabel{0.2051-1-4}
\end{align}
By Lemma \ref{lema-esti-F}(d), we have
\renewcommand\theequation{\thesection.\arabic{equation}*}
\begin{align}
|\log \widehat{F}'(z)|=|\log F'((b_0-c_0)z)|\leq LogDF_{max}(5\widehat{r})\Aeq{0.0014}\ly 0.01. \retainlabel{0.0014-1-4}
\end{align}
This implies that $|\widehat{F}'(z)-1|<0.01e^{0.01}<\frac{1}{4}$.
Note that $\{z\in\C:\re \big((b_0-c_0)z)\leq -9.5\widehat{r}\}\subset\{z\in\C:|z|> \widehat{r}\}$. By \cite[Proposition 2.5.2]{Shi00a}, there exists a repelling Fatou coordinate $\widehat{\Phi}_{rep}:\{z\in\C:\re \big((b_0-c_0)z)<-9.5\widehat{r}\}\to\C$ such that $\widehat{\Phi}_{rep}$ is univalent and $\widehat{\Phi}_{rep}(\widehat{F}(z))=\widehat{\Phi}_{rep}(z)+1$.
Define $\Phi_{rep}(z):=\widehat{\Phi}_{rep}(\frac{z}{b_0-c_0})$. Note that $2R=10\widehat{r}>9.5\widehat{r}$. Then $\Phi_{rep}$ is univalent in $\{z\in\C:\re z<-9.5\widehat{r}\}$ and we have $\Phi_{rep}(F(z))=\Phi(z)+1$ for $\{z\in\C:\re z<-2R\}$.
\end{proof}

\section{Attracting Fatou coordinate $\Phi_{attr}$}\label{sec-Phi-attr}

\begin{defi}
For $0\leq \theta<\frac{\pi}{2}$, let $pr_+^\theta(z)=\re (ze^{-\ii\theta})$ and $pr_-^\theta(z)=\re (ze^{+\ii\theta})$ be the orthogonal projections to the line with angle $\pm\theta$ to the real axis respectively. Define two angles
\begin{equation}
\theta_1:=\tfrac{3\pi}{20} \quad\text{and}\quad \theta_2:=\tfrac{\pi}{4};
\end{equation}
some constants
\begin{align}
u_{1,\theta_1}&:=8.5, & u_{2,\theta_1}&:=6.1, \quad u_3:=22\cos \theta_1\Aeq{19.60},\quad u_4:=17.3, \\
u_{1,\theta_2}&:=9, & u_{2,\theta_2}&:=6.6;
\end{align}
and some half planes ($i=1,2$)
\begin{align}
H_{1,\theta_i}^\pm & :=\{z\in\C:pr_\pm^{\theta_i}(z)>u_{1,\theta_i}\}, &  H_{2,\theta_i}^\pm & :=\{z\in\C:pr_\pm^{\theta_i}(z)>u_{2,\theta_i}\}, \\
H_3^\pm & :=\{z\in\C:pr_\pm^{\theta_1}(z)\geq u_3\}, & H_4^\pm & :=\{z\in\C:pr_\pm^{\theta_1}(z)\geq u_4\}.
\end{align}
\end{defi}

\begin{lem}[{Attracting Fatou coordinate $\Phi_{attr}$}]\label{lema-attr}
We have
\begin{enumerate}
\item $\varphi(H_{2,\theta_i}^\pm)\supset H_{1,\theta_i}^\pm$ for $i=1,2$ and $\varphi(H_4^\pm)\supset H_3^\pm$. Hence $F$ is defined in $H_{1,\theta_i}^+\cup H_{1,\theta_i}^-$ for $i=1,2$;
\item $Q$ is injective in $H_{2,\theta_i}^\pm$. Therefore, $F$ is injective in $H_{1,\theta_i}^\pm$ for $i=1,2$;
\item  If $z\in H_{1,\theta_i}^\pm$ for $i=1,2$, then $|\arg(F(z)-z)|<\tfrac{\pi}{2}-\theta_i$, hence $F(H_{1,\theta_i}^\pm)\subset H_{1,\theta_i}^\pm$. In particular, $H_{1,\theta_i}^+\cup H_{1,\theta_i}^-=\V(u_{0,\theta_i},\tfrac{\pi}{2}+\theta_i)$ is forward invariant under $F$ and contained in the immediate parabolic basin of $\infty$, where
    \begin{equation}
    u_{0,\theta_1}=u_{1,\theta_1}/\cos\theta_1\Aeq{9.53}\quad\text{and}\quad u_{0,\theta_2}=u_{1,\theta_2}/\cos\theta_2\Aeq{12.72};
    \end{equation}
\item  There exists an attracting Fatou coordinate $\Phi_{attr}$ for $F$ in $\V(u_{0,\theta_1},\tfrac{13\pi}{20})\cup \V(u_{0,\theta_2},\tfrac{3\pi}{4})$ and it is injective in each of $H_{1,\theta_i}^\pm$.
\end{enumerate}
\end{lem}

In the following, we normalize the Fatou coordinate $\Phia$ such that $\Phia(cv)=1$.

\begin{proof}
(a) By Lemma \ref{lema-esti-varphi}(b), the half planes $H_{1,\theta_i}^\pm$ is contained in $Image(\varphi)$. If $\zeta\in\partial H_{2,\theta_1}^\pm$, by Lemma \ref{lema-esti-varphi}(e) we have
\renewcommand\theequation{\thesection.\arabic{equation}*}
\begin{align}
&~ pr_\pm^{\theta_1}(\varphi(\zeta))
=    pr_\pm^{\theta_1}(\zeta)+pr_\pm^{\theta_1}(c_{00})+pr_\pm^{\theta_1}(c_0-c_{00})+pr_\pm^{\theta_1}(\varphi_1(\zeta)) \notag \\
\leq &~u_{2,\theta_1}+c_{00}\cos\theta_1+c_{01,max}+\varphi_{1,max}(u_{2,\theta_1})\Aeq{8.484}\ly 8.5. \retainlabel{equ:u-2-2}
\end{align}
Similarly, we have
\renewcommand\theequation{\thesection.\arabic{equation}*}
\begin{align}
&~ pr_\pm^{\theta_2}(\varphi(\zeta))
=    pr_\pm^{\theta_2}(\zeta)+pr_\pm^{\theta_2}(c_{00})+pr_\pm^{\theta_2}(c_0-c_{00})+pr_\pm^{\theta_2}(\varphi_1(\zeta)) \notag \\
\leq &~u_{2,\theta_2}+c_{00}\cos\theta_2+c_{01,max}+\varphi_{1,max}(u_{2,\theta_2})\Aeq{8.958}\ly 9. \retainlabel{equ:8.161-8.2}
\end{align}
Hence $\varphi(\zeta)\not\in H_{1,\theta_i}^\pm$ for $i=1,2$. This implies that $\varphi^{-1}(H_{1,\theta_i}^\pm)$ is contained in one side of $\partial H_{2,\theta_i}^\pm$. On the other hand, if $\zeta\in H_{1,\theta_i}^\pm$ is a point with large real part, then $\varphi^{-1}(\zeta)\in H_{2,\theta_i}^\pm$. This implies that $\varphi^{-1}(H_{1,\theta_i}^\pm)\subset H_{2,\theta_i}^\pm$ and hence $\varphi(H_{2,\theta_i}^\pm)\supset H_{1,\theta_i}^\pm$.

If $\zeta\in\partial H_4^\pm$, then
\renewcommand\theequation{\thesection.\arabic{equation}*}
\begin{align}
pr_\pm^{\theta_1}(\varphi(\zeta))
\leq &~u_4+c_{00}\cos\theta_1+c_{01,max}+\varphi_{1,max}(u_4)\Aeq{19.55} \notag\\
\ly    &~ pr_\pm^{\theta_1}(22)\Aeq{19.60}. \retainlabel{equ:u_4-2}
\end{align}
By the same argument as above one can show that $\varphi(H_4^\pm)\supset H_3^\pm$.

\medskip
(b) By Lemma \ref{lema-injective}(c), $Q$ is an injection in $H_{2,\theta_i}^\pm$. Then $F=Q\circ\varphi^{-1}$ is injective in $H_{1,\theta_i}^\pm$ since $\varphi^{-1}(H_{1,\theta_i}^\pm)\subset H_{2,\theta_i}^\pm$.

\medskip
(c) If $z\in H_{1,\theta_1}^\pm$, then $\zeta=\varphi^{-1}(z)\in H_{2,\theta_1}^\pm$ by Part (a). By Lemma \ref{lema-esti-F}(b), we have
\renewcommand\theequation{\thesection.\arabic{equation}*}
\begin{align}
|\arg(F(z)-z)|
& ~\leq\max\{Arg\Delta F_{max}(u_{2,\theta_1},\pm\theta_1),-Arg\Delta F_{min}(u_{2,\theta_1},\pm\theta_1)\} \notag\\
& ~(\doteqdot\max\{0.7619\ldots,0.5827\ldots\})\ly 1< \tfrac{7\pi}{20}. \retainlabel{equ:ArgFminusz}
\end{align}
Similarly, if $z\in H_{1,\theta_2}^\pm$, then $\zeta=\varphi^{-1}(z)\in H_{2,\theta_2}^\pm$ by Part (a). By Lemma \ref{lema-esti-F}(b), we have
\renewcommand\theequation{\thesection.\arabic{equation}*}
\begin{align}
|\arg(F(z)-z)|
& ~\leq\max\{Arg\Delta F_{max}(u_{2,\theta_2},\pm\theta_2),-Arg\Delta F_{min}(u_{2,\theta_2},\pm\theta_2)\} \notag\\
& ~(\doteqdot\max\{0.7767\ldots,0.5069\ldots\})\ly \tfrac{\pi}{4}\Aeq{0.7853}. \retainlabel{equ:pi-0.7853}
\end{align}
This implies that each of $H_{1,\theta_i}^{\pm}$ is forward invariant under $F$ and hence $H_{1,\theta_i}^+\cup H_{1,\theta_i}^-=\V(u_{0,\theta_i},\frac{\pi}{2}+\theta_i)$ is also. The assertion that $\V(u_{0,\theta_i},\frac{\pi}{2}+\theta_i)$ is contained in the immediate parabolic basin of $\infty$ follows immediately.

\medskip
(d) This part can be proved as that of \cite[Proposition 5.5]{IS08}.
\end{proof}

\begin{lem}[{Estimates on $\Phi_{attr}$}]\label{lema-est-Phi}
We have
\begin{enumerate}
\item The attracting Fatou coordinate $\Phi_{attr}$ satisfies the following inequalities:
\renewcommand\theequation{\thesection.\arabic{equation}}
\begin{gather}
-\tfrac{3\pi}{20}<\arg\Phi_{attr}'(z)<\tfrac{\pi}{5} \text{~for~} z\in H_3^+ \text{~and~} -\tfrac{\pi}{5}<\arg\Phi_{attr}'(z)<\tfrac{3\pi}{20} \text{~for~} z\in H_3^-,\\
0.083<|\Phi_{attr}'(z)|<0.228 \text{~for~} z\in H_3^+\cup H_3^-=\OV(22,\tfrac{13\pi}{20});
\end{gather}

\item $\Phi_{attr}$ is injective in $H_3^+\cup H_3^-=\OV(22,\tfrac{13\pi}{20})\supset \OV(F(cv),\frac{13\pi}{20})$. There exists a domain $\MH_2$ such that $\Phi_{attr}$ is a homeomorphism from $\overline{\MH}_2$ onto $\{w\in\C:\re w\geq 2\}$, and $\MH_2$ satisfies $\OV(F(cv),\tfrac{7\pi}{20})\subset\MH_2\cup\{F(cv)\}\subset\overline{\MH}_2
    \subset\V(F(cv),\tfrac{13\pi}{20})\cup\{F(cv)\}$ and $F(cv)\in\partial\MH_2$.
\end{enumerate}
\end{lem}

\begin{proof}
By Lemma \ref{lema-F-cv}(b), we have $\OV(F(cv),\frac{13\pi}{20})\subset \OV(22,\frac{13\pi}{20})$ and it is easy to see that $H_3^+\cup H_3^-=\OV(22,\frac{13\pi}{20})$. The proof of this proposition is similar to \cite[Lemma 5.31]{IS08}. We only check the numerical results of Part (a).

\medskip

Suppose $z\in H_3^+$. Then $\zeta=\varphi^{-1}(z)\in H_4^+$, i.e., $\re (\zeta e^{-\ii\theta_1})\geq u_4=17.3$. We claim that
\renewcommand\theequation{\thesection.\arabic{equation}}
\begin{equation}\label{dom-1}
F(z)\in\D_{H_{1,\theta_1}^+}\big(z,s(r_4)\big)
\end{equation}
with $r_4=0.34$, where $s(r)=\log\tfrac{1+r}{1-r}$ is defined in Lemma \ref{lema-esti-re}(b). By applying Lemma \ref{lema-esti-re}(b) with $H=H_{1,\theta_1}^+$, $t=u_{1,\theta_1}$, $u=pr_+^{\theta_1}(z)-u_{1,\theta_1}$, $r=r_4$ and $\theta=\theta_1$, \eqref{dom-1} is equivalent to
\begin{equation}\label{dom-2}
F(z)-z\in\D\left(\frac{2ur_4^2e^{\ii\theta_1}}{1-r_4^2},\frac{2ur_4}{1-r_4^2}\right).
\end{equation}
This disk contains $0$ since $|r_4^2e^{\ii\theta_1}|<r_4$ and it is increasing with $u$. Therefore, we only need to check \eqref{dom-2} with the smallest $u$, i.e., $u_5=u_3-u_{1,\theta_1}=pr_+^{\theta_1}(22)-8.5$. As in the proof of \cite[Lemma 5.27 (a)]{IS08}, one can write $F(z)-z=\alpha+\beta$, where $\alpha=b_0-c_{00}+\frac{b_1 e^{-\ii\theta_1}}{2u_4}$ and $|\beta|\leq\beta_{max}(u_4)$. By a numerical calculation, we have
\renewcommand\theequation{\thesection.\arabic{equation}*}
\begin{align}
&~\left|\alpha-\frac{2u_5r_4^2e^{\ii\theta_1}}{1-r_4^2}\right|+\beta_{max}(u_4)-\frac{2u_5 r_4}{1-r_4^2} \notag\\
=&~\sqrt{\Big(b_0-c_{00}+\Big(\frac{b_1}{2u_4}-\frac{2u_5 r_4^2}{1-r_4^2}\Big)\cos \theta_1\Big)^2
+\Big(\frac{b_1}{2u_4}+\frac{2u_5 r_4^2}{1-r_4^2}\Big)^2\sin^2\theta_1} \notag\\
&~+\beta_{max}(u_4)-\frac{2u_5 r_4}{1-r_4^2}\Aeq{-0.1616}\ly 0. \retainlabel{equ:Alpha-Beta-2}
\end{align}
This implies that \eqref{dom-1} and \eqref{dom-2} are true.

\medskip
Applying Theorem \ref{thm-est-Fatou} to $\Phi_{attr}$ with $\Omega=H_{1,\theta_1}^+$ and $r=r_4$, by Lemma \ref{lema-esti-F} we have
\renewcommand\theequation{\thesection.\arabic{equation}*}
\begin{align}
\arg\Phi_{attr}'(z)
\leq&~ -\arg(F(z)-z)+\frac{1}{2}\left|\log F'(z)\right|+\frac{1}{2}\log\frac{1}{1-r_4^2} \notag\\
\leq&~ -Arg\Delta F_{min}(u_4,\theta_1)+\frac{1}{2}LogDF_{max}(u_4)-\frac{1}{2}\log(1-r_4^2) \notag\\
&~ \Aeq{0.5244} \ly \tfrac{\pi}{5}\Aeq{0.6283} \retainlabel{equ:PhiArg-sup}
\end{align}
and
\begin{align}
\arg\Phi_{attr}'(z)
\geq&~-\arg(F(z)-z)-\frac{1}{2}\left|\log F'(z)\right|-\frac{1}{2}\log\frac{1}{1-r_4^2} \notag\\
\geq&~-Arg\Delta F_{max}(u_4,\theta_1)-\frac{1}{2}LogDF_{max}(u_4)+\frac{1}{2}\log(1-r_4^2)\notag\\
&~\Aeq{-0.4530} \gy -0.46 \gy -\tfrac{3\pi}{20}\Aeq{-0.4712}.\label{equ:arg-lower}
\end{align}
The similar estimate can be applied to $z\in H_3^-$.

For $|\Phi_{attr}'(z)|$ on $H_3^+$ or $H_3^-$, by Theorem \ref{thm-est-Fatou} and Lemma \ref{lema-esti-F}, we have
\begin{align}
|\Phi_{attr}'(z)|
& \leq \exp\left(-\log |F(z)-z|+\frac{1}{2}\left|\log F'(z)\right|+\frac{1}{2}\log\frac{1}{1-r_4^2}\right) \notag\\
& \leq \frac{\exp\left(\tfrac{1}{2}LogDF_{max}(u_4)\right)}{Abs\Delta F_{min}(u_4,\theta_1)\sqrt{1-r_4^2}}\Aeq{0.2275}\ly 0.228 \retainlabel{equ:PhiAbs-sup}
\end{align}
and
\begin{align}
&~|\Phi_{attr}'(z)|
\geq \exp\left(-\log |F(z)-z|-\frac{1}{2}\left|\log F'(z)\right|-\frac{1}{2}\log\frac{1}{1-r_4^2}\right)\notag\\
\geq &~ \frac{\sqrt{1-r_4^2}}{Abs\Delta F_{max}(u_4,\theta_1)\exp\left(\tfrac{1}{2}LogDF_{max}(u_4)\right)}\Aeq{0.0839}\gy 0.083. \retainlabel{equ:PhiAbs-inf}
\end{align}
This finishes the proof.
\end{proof}

Recall that $R_1=108$ is defined just before Proposition \ref{prop-Phi}.

\begin{lem}\label{lema-W2}
Let $R_2=R_1-9=99$.
There are domains $D_2$, $D_2^\pm$ in
\begin{equation}
W_2:=\V(F(cv),\tfrac{13\pi}{20})\setminus\overline{\V}(F^{\circ 2}(cv),\tfrac{7\pi}{20})
\end{equation}
such that
\begin{equation}
\begin{split}
\Phia(D_2)=&~\{w\in\C:2<\re w<3,\,-\eta<\im w<\eta\}, \text{~and}\\
\Phia(D_2^\pm)=&~\{w\in\C:2<\re w<3,\,\pm\,\im w>\eta\},
\end{split}
\end{equation}
where $D_2\subset\D(F(cv),R_2)$ and $D_2^\pm\subset\{z\in\C:0.785<\pm\,\arg(z-F(cv))<\tfrac{13\pi}{20}\}$.
\end{lem}

\begin{proof}
Note that $\Phia:\overline{\MH}_2\to\{w\in\C:\re w\geq 2\}$ is a homeomorphism by Lemma \ref{lema-est-Phi}(b). We define
\begin{equation}
\begin{split}
D_2:=&~\Phia^{-1}(\{w\in\C:2<\re w<3,\,-\eta<\im w<\eta\}) \text{ and }\\
D_2^\pm:=&~\Phia^{-1}(\{w\in\C:2<\re w<3,\,\pm\,\im w>\eta\}).
\end{split}
\end{equation}
By Lemma \ref{lema-F-cv}(c), we have $F^{\circ 2}(cv)\in\V(F(cv),\frac{7\pi}{20})$. Hence $\OV(F^{\circ 2}(cv),\frac{7\pi}{20})\subset \V(F(cv),\frac{7\pi}{20})$.
By Lemma \ref{lema-F-cv}(b), we have $\OV(F(cv),\frac{7\pi}{20})\subset\overline{\V}(22,\frac{7\pi}{20})=H_3^+\cap H_3^-$.

We need to use the following result which appears in the proof of \cite[Lemma 5.31(b)]{IS08}:
Let $U$ be a domain in $\C$ and $f:U\to\C$ be a holomorphic map with $f'(z)\neq 0$ for all $z\in U$. Let $[z_1,z_2]\subset U$ be a non-degenerate closed segment. Then
\renewcommand\theequation{\thesection.\arabic{equation}}
\begin{equation}\label{equ:criterion-1}
\theta<\arg\frac{f(z_2)-f(z_1)}{z_2-z_1}<\theta'
\text{\quad if \quad}
\theta<\arg f'(z)<\theta'\leq\theta+\pi \text{ on } [z_1,z_2].
\end{equation}

Applying \eqref{equ:criterion-1} to $z_1=F^{\circ 2}(cv)$, $z_2\in\partial \V(F^{\circ 2}(cv),\frac{7\pi}{20})$ with $\im z_2>\im F^{\circ 2}(cv)$, $f=\Phia$, $\theta=-\frac{3\pi}{20}$ and $\theta'=\frac{3\pi}{20}$ (note that $\Phia(F^{\circ 2}(cv))=3$ and $\OV(F^{\circ 2}(cv),\frac{7\pi}{20})\subset H_3^+\cap H_3^-$), we have
\begin{equation}
-\tfrac{3\pi}{20}<\arg(\Phia(z_2)-3)-\tfrac{7\pi}{20}<\tfrac{3\pi}{20}.
\end{equation}
Similarly, if $z_2\in\partial \V(F^{\circ 2}(cv),\frac{7\pi}{20})$ with $\im z_2<\im F^{\circ 2}(cv)$, then we have
$-\tfrac{3\pi}{20}<\arg(\Phia(z_2)-3)+\tfrac{7\pi}{20}<\tfrac{3\pi}{20}$. Therefore, $|\arg(\Phia(z_2)-3)|<\frac{\pi}{2}$ for all
$z_2\in\partial \V(F^{\circ 2}(cv),\frac{7\pi}{20})\setminus\{F^{\circ 2}(cv)\}$ and hence $\re \Phia(z_2)>3$.
By Lemma \ref{lema-est-Phi}(b), $\MH_2\subset\V(F(cv),\tfrac{13\pi}{20})$. This implies that $D_2$, $D_2^+$ and $D_2^-$ are all contained in $W_2$.

\medskip
Similarly, if $|\arg(z-F(cv))|\leq 0.785<\tfrac{7\pi}{20}$ (hence $z\in\V(F(cv),0.785)\subset H_3^+\cap H_3^-$), by \eqref{equ:arg-lower} we have $|\arg(\Phia(z)-2)|<0.46+0.785=1.245$.
Then $\Phia(z)$ cannot be in $\{w\in\C:2<\re w<3,|\im w|>\eta\}$ since
\renewcommand\theequation{\thesection.\arabic{equation}*}
\begin{equation}\label{eta-geq}
\tan(1.245)\Aeq{2.960}\ly \eta=3.
\end{equation}
This implies that $D_2^+$ and $D_2^-$ are contained in $\{z\in\C:0.785<\pm\arg(z-F(cv))<\frac{13\pi}{20}\}$ respectively.

Finally, it remains to show $D_2\subset\D(F(cv),R_2)$. Note that the derivative of $\Phia^{-1}$ is bounded above by $\frac{1}{0.083}$ in $\Phia(D_2)=\{w\in\C:2<\re w<3,\,-\eta<\im w<\eta\}$ by Lemma \ref{lema-est-Phi}(a). Since $\Phia(D_2)\subset\D(2,\sqrt{1+\eta^2})$, we have $D_2\subset\D(F(cv),\sqrt{1+\eta^2}/0.083)$. Therefore, we only need to check that $\sqrt{1+\eta^2}/0.083<R_2=99$. Actually, this is true even for a bigger $\eta$ such as $\eta=8$ since
\renewcommand\theequation{\thesection.\arabic{equation}*}
\begin{equation}\label{eta-leq}
\sqrt{1+8^2}/0.083\Aeq{97.135}\ly 99.
\end{equation}
This finishes the proof.
\end{proof}

Recall that $\MU_1=\MU_1^Q=\MU_{1+}\cup\MU_{1-}\cup\gamma_{a1}$ is defined in \S\ref{sec-P-to-Q} (see Figure \ref{Fig-dom-Q}) and let $\MH_2$ be introduced in Lemma \ref{lema-est-Phi}(b).
The next work is to study the preimage of $W_2$ under $F=Q\circ\varphi^{-1}$.

\begin{lem}\label{lema-W1}
Let $W_1'$ be the component of $Q^{-1}(W_2)$ in $\MU_1$ and denote $\MH_1:=\varphi\big((Q|_{\MU_1})^{-1}(\MH_2)\big)$. Then
\begin{enumerate}
\item $\varphi$ is well defined on $\overline{W_1'}$ and $\varphi(\overline{W_1'})\setminus\{cv, F(cv)\}\subset W_1=\V(cv,\tfrac{3\pi}{4})\setminus\overline{\V}(F(cv),\tfrac{\pi}{4})$; and
\item $\overline{\V}(cv,\tfrac{\pi}{6})\subset\MH_1\cup\{cv\}\subset \overline{\MH}_1\subset \V(cv,\tfrac{3\pi}{4})\cup\{cv\}$.
\end{enumerate}
\end{lem}

\begin{proof}
(a) Recall that $\theta_1=\frac{3\pi}{20}$. We first claim that $\overline{W_1'}\subset\V(u_6/\cos\theta_1,\frac{13\pi}{20})$, where $u_6=10.7$. Note that $\V(F(cv)$, $\tfrac{13\pi}{20})\subset\V(22,\frac{13\pi}{20})$ by Lemma \ref{lema-F-cv}(b). To prove the claim, it is sufficient to show that if $\zeta\in \partial\V(u_6/\cos\theta_1,\frac{13\pi}{20})$, then $Q(\zeta)\in \C\setminus\OV(22,\tfrac{13\pi}{20})$. This is true since by Lemma \ref{lema-Q},
\renewcommand\theequation{\thesection.\arabic{equation}*}
\begin{equation}
\begin{split}
|Q(\zeta)-(\zeta+b_0)|
&\leq \frac{b_1}{u_6}+Q_{2,max}\big(u_6)\Aeq{2.306} \\
&\ly (22-b_0)\cos\theta_1-u_6\Aeq{2.388}. \retainlabel{2.22-2.25}
\end{split}
\end{equation}
Obviously, $\varphi$ is well defined on $\overline{W_1'}$.

If $\zeta\in \OV(u_6/\cos\theta_1,\frac{13\pi}{20})$, by Lemma \ref{lema-injective}(b), we have
\renewcommand\theequation{\thesection.\arabic{equation}*}
\begin{equation}\label{0.2803-0.29}
|\arg Q'(\zeta)|\leq LogDQ_{max}(u_6)\Aeq{0.2433}\ly 0.245.
\end{equation}
Applying \eqref{equ:criterion-1} to $z_1=F(cv)$, $z_2\in\partial \V(F(cv),\frac{13\pi}{20})$ with $\im z_2>\im F(cv)$, $f=(Q|_{\MU_1})^{-1}$, $\theta=-0.245$ and $\theta'=0.245$ (note that $(Q|_{\MU_1})^{-1}(z_1)=\varphi^{-1}(cv)$), we have
\begin{equation}
-0.245<\arg\big((Q|_{\MU_1})^{-1}(z_2)-\varphi^{-1}(cv)\big)-\tfrac{13\pi}{20}<0.245.
\end{equation}
If $\zeta\in \OV(u_6/\cos\theta_1,\frac{13\pi}{20})$, by Lemma \ref{lema-esti-varphi}(f), we have
\renewcommand\theequation{\thesection.\arabic{equation}*}
\begin{equation}\label{arg-0.03}
|\arg \varphi'(\zeta)|\leq LogD\varphi_{max}(u_6)\Aeq{0.0101}\ly 0.011.
\end{equation}
Applying \eqref{equ:criterion-1} again to $\widetilde{z}_1=\varphi^{-1}(cv)$, $\widetilde{z}_2=(Q|_{\MU_1})^{-1}(z_2)$, $\widetilde{f}=\varphi$, $\widetilde{\theta}=-0.011$ and $\widetilde{\theta}'=0.011$ (note that $\varphi(\widetilde{z}_1)=cv$ and $\varphi(\widetilde{z}_2)=F^{-1}(z_2)$, where $F^{-1}=\varphi\circ (Q|_{\MU_1})^{-1}$), we have
\begin{equation}
\tfrac{11\pi}{20}<\tfrac{13\pi}{20}-0.256<\arg\big(F^{-1}(z_2)-cv\big)<\tfrac{13\pi}{20}+0.256<\tfrac{3\pi}{4}.
\end{equation}
The case $\im z_2<\im F(cv)$ can be proved completely similarly.

By a similar argument, if $z_1=F^{\circ 2}(cv)$ and $z_2\in\partial \V(F^{\circ 2}(cv),\frac{7\pi}{20})\setminus\{z_1\}$, then
\begin{equation}
\tfrac{\pi}{4}<\tfrac{7\pi}{20}-0.256<\arg\big(F^{-1}(z_2)-F(cv)\big)<\tfrac{7\pi}{20}+0.256<\tfrac{9\pi}{20}.
\end{equation}
Hence $\varphi(\overline{W_1'})\setminus\{cv, F(cv)\}\subset \V(cv,\tfrac{13\pi}{20}+0.256)\setminus\overline{\V}(F(cv),\tfrac{7\pi}{20}-0.256)\subset W_1$ and this finishes the proof of Part (a).

\medskip
(b) Since $\V(F(cv),\tfrac{7\pi}{20}-0.256)\subset\V(F(cv),\frac{13}{20}\pi)$, the corresponding estimates in Part (a) still hold. Applying \eqref{equ:criterion-1} to $z_1=F(cv)$, $z_2\in\partial \V(F(cv),\tfrac{7\pi}{20}-0.256)$ with $\im z_2>\im F(cv)$, $f=F^{-1}=\varphi\circ (Q|_{\MU_1})^{-1}$, $\theta=-0.256$ and $\theta'=0.256$, we have
\begin{equation}
\tfrac{\pi}{6}<\tfrac{7\pi}{20}-0.256-0.256<\arg\big(F^{-1}(z_2)-cv\big)<\tfrac{7\pi}{20}.
\end{equation}
The case $\im z_2<\im F(cv)$ can be proved completely similarly.
Therefore, we have $\overline{\V}(cv,\tfrac{\pi}{6})\subset\MH_1\cup\{cv\}$ and $\overline{\MH}_1\subset \V(cv,\tfrac{3\pi}{4})\cup\{cv\}$ follows by Part (a).
\end{proof}

\begin{proof}[Proof of Proposition \ref{prop-Phi}]
(a) By Lemma \ref{lema-attr}, $\Phia$ is defined in $\V(u_0,\tfrac{3\pi}{4})$, where $u_0:=u_{0,\theta_2}=9/\cos \frac{\pi}{4}\Aeq{12.72}<13$. In particular, $\Phia$ is defined in $\V(cv,\tfrac{3\pi}{4})$. Moreover, $\V(13,\tfrac{3\pi}{4})$ is contained in the immediate parabolic basin of $\infty$. Hence $\{w\in\C:\re w>1\}\subset\Phia(\V(13,\frac{3\pi}{4}))$ follows by Lemmas \ref{lema-W2} and \ref{lema-W1}.

\medskip
(b) Define $\widetilde{D}_1=Q^{-1}(D_2)\cap\MU_1$ and $\widetilde{D}_1^\pm=Q^{-1}(D_2^\pm)\cap\MU_1$. We then define
$D_1=\varphi(\widetilde{D}_1)$ and $D_1^\pm=\varphi(\widetilde{D}_1^\pm)$.
By Lemma \ref{lema-W1}(a), $D_1\cup D_1^+\cup D_1^-$ is contained in $W_1$ and satisfies \eqref{equ:PhiaD}.

By \eqref{0.2803-0.29} and \eqref{arg-0.03}, applying \eqref{equ:criterion-1} to $z_1=F(cv)$, $z_2\in\partial \V(F(cv)$, $0.785)$ with $\im z_2>\im F(cv)$, $f=F^{-1}=\varphi\circ (Q|_{\MU_1})^{-1}$, $\theta=-0.256$ and $\theta'=0.256$, we have
$-0.256<\arg\big(F^{-1}(z_2)-cv\big)-0.785<0.256$.
Then $D_1^+\subset\{z\in\C:\tfrac{\pi}{6}<\arg(z-cv)<\tfrac{3\pi}{4}\}$ since
$0.785-0.256=0.529>\tfrac{\pi}{6}$.
The case $\im z_2<\im F(cv)$ can be proved completely similarly and $D_1^-\subset\{z\in\C:\tfrac{\pi}{6}<-\arg(z-cv)<\tfrac{3\pi}{4}\}$.

\medskip
If $|\zeta-cv|\geq R_2+6$, then $|\zeta|>50$. By Lemma \ref{lema-F-cv}(b), $|F(cv)-25.5|<3$. Note that $|b_0+cv-25.5|<1.5$ and $b_1<22$. If $|\zeta|\geq 50$, by \eqref{0.3-Q2max} we have
\renewcommand\theequation{\thesection.\arabic{equation}}
\begin{equation}\label{equ:Q-zeta-est}
|Q(\zeta)-\zeta-b_0|\leq \tfrac{b_1}{50}+Q_{2,max}(50)<\tfrac{22}{50}+0.4<1.
\end{equation}
Hence we have $|Q(\zeta)-F(cv)|\geq |\zeta-cv|-|Q(\zeta)-\zeta-b_0|-|F(cv)-25.5|-|b_0+cv-25.5|\geq |\zeta-cv|-5.5>R_2$. By Lemma \ref{lema-W2}, $D_2\subset\D(F(cv),R_2)$. This implies that
\begin{equation}
\widetilde{D}_1= (Q|_{\MU_1})^{-1}(D_2)\subset (Q|_{\MU_1})^{-1}(\D(F(cv),R_2))\subset\D(cv,R_2+6).
\end{equation}
Note that $\varphi(\{\zeta\in\C:|\zeta-cv|=R_2+6\})$ is a Jordan curve contained in $\D(cv,R_2+9)$ by \eqref{equ-varphi-est}. This implies that $D_1=\varphi(\widetilde{D}_1)\subset\D(cv,R_2+9)=\D(cv,R_1)$.
\end{proof}

\section{Locating domains $D_0$, $D_0'$, $D_0''$, $D_{-1}$, $D_{-1}'''$ and $D_{-1}''''$}\label{sec-bounding-dom}

In this section, we denote $pr_\pm:=pr_\pm^{\theta_2}$ for simplicity, where $\theta_2=\frac{\pi}{4}$. Recall that $\Omega_0$ is the set defined in Lemma \ref{lema-arg-arg}.

\begin{lem}\label{lema-W0}
\begin{enumerate}
\item Let $\widetilde{W}_0:=\{\zeta:\re \zeta>cp \text{~or~} pr_+(\zeta)>6.5\text{~or~} pr_-(\zeta)>6.5\}$ and $\widetilde{W}_0'$ be the component of $Q^{-1}(\V(cv,\frac{3\pi}{4}))$ in $\MU_1$. Then $\widetilde{W}_0'\subset\widetilde{W}_0$;
\item $\varphi(\widetilde{W}_0)\subset W_0:=\{z\in\C:\re z>1.7 \text{~or~} pr_+(z)>4.2\text{~or~} pr_-(z)>4.2\}$;
\item If $\zeta\in Q^{-1}(W_0)\cap (\MU_{2+}\cup\MU_{3-}\cup\gamma_{b2})$, then $\zeta\not\in \OD(a_5,\varepsilon_5)\cup\OD(\overline{a}_5,\varepsilon_5)$;
\item $\big(Q^{-1}(W_0)\cap\MU_{123}\big)\setminus\big(\OD(\omega,\varepsilon_6)\cup\OD(\overline{\omega},\varepsilon_6)\big)\subset \Omega_0$.
\end{enumerate}
\end{lem}

See Figure \ref{Fig-W-0}. We postpone the proof of Lemma \ref{lema-W0} until the end of this section.

\begin{figure}[!htpb]
  \setlength{\unitlength}{1mm}
  \centering
  \includegraphics[width=0.96\textwidth]{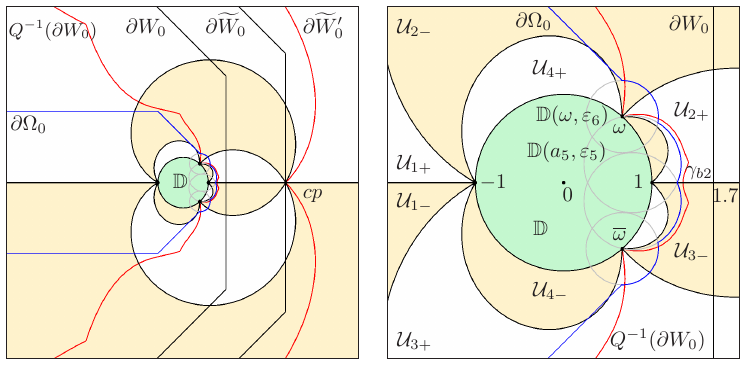}
  \caption{Left: The domains $\widetilde{W}_0'$, $\widetilde{W}_0$, $W_0$, $Q^{-1}(W_0)$, $\Omega_0$ and their boundaries; Right: A zoom of the left picture near the unit disk.}
  \label{Fig-W-0}
\end{figure}

\begin{defi}
Note that $Q$ maps each of $\MU_1$, $\MU_2$ and $\MU_3$ homeomorphically onto $\C\setminus(-\infty,cv]$. Define
\begin{align}
\widetilde{\MH}_0&=(Q|_{\MU_1})^{-1}(\MH_1), & \widetilde{D}_0&=(Q|_{\MU_1})^{-1}(D_1),  & \widetilde{D}_0^+&=(Q|_{\MU_1})^{-1}(D_1^+),\\
\widetilde{D}_0'&=(Q|_{\MU_2})^{-1}(D_1), & \widetilde{D}_0''&=(Q|_{\MU_3})^{-1}(D_1). &
\end{align}
The domain $\widetilde{\MH}_0$ and hence $\widetilde{D}_0$ and $\widetilde{D}_0^+$ are contained in $\C\setminus\overline{\mathscr{D}}_{r_1}$ because of Lemma \ref{lema-E-r1}(d) and $\overline{\MH}_1\subset\OV(cv,\frac{3\pi}{4})\subset\C\setminus\OD(0,\rho)$. The domains $\widetilde{D}_0'$ and $\widetilde{D}_0''$ are contained in $\C\setminus\overline{\mathscr{D}}_{r_1}$ by Lemma \ref{lema-D0}(b) below. Hence we can define
\begin{equation}
\MH_0=\varphi(\widetilde{\MH}_0), ~D_0=\varphi(\widetilde{D}_0),~D_0^+=\varphi(\widetilde{D}_0^+),
~D_0'=\varphi(\widetilde{D}_0'),~D_0''=\varphi(\widetilde{D}_0'').
\end{equation}
It is easy to see that $F(\MH_0)=\MH_1$ and $\Phia$ extends naturally to $\overline{\MH}_0$ such that it is a homeomorphism onto $\{w\in\C:\re w\geq 0\}$. Moreover, $\Phia(D_0)=\{w\in\C:0<\re w<1,\,|\im w|<\eta\}$ and $D_0\subset\MH_0\setminus\overline{\MH}_1$. In particular, by Lemma \ref{lema-W1}(b), we have $\OV(cv,\frac{\pi}{6})\subset\overline{\MH}_1$ and hence $D_0\cap\OV(cv,\frac{\pi}{6})=\emptyset$. By Lemma \ref{lema-W0}(a)(b), $D_0$ is contained in $W_0$ since $D_1\subset\V(cv,\frac{3\pi}{4})$. This implies that $D_0$ must be contained in $\C\setminus((-\infty,0]\cup[cv,+\infty))$.
\end{defi}

Note that $Q$ maps $\MU_{1+}\cup\,\MU_{2-}\cup\gamma_{b1}$, $\MU_{2+}\cup\,\MU_{3-}\cup\gamma_{b2}$ and $\MU_{1-}\cup\,\MU_{3+}\cup\gamma_{b3}$ homeomorphically onto the double slitted complex plane $\C\setminus((-\infty,0]\cup[cv,+\infty))$. Thus we can define
\begin{align}
\widetilde{D}_{-1}&=(Q|_{(\MU_{1+}\cup\,\MU_{2-}\cup\gamma_{b1})})^{-1}(D_0), \quad
\widetilde{D}_{-1}'''=(Q|_{(\MU_{2+}\cup\,\MU_{3-}\cup\gamma_{b2})})^{-1}(D_0), \text{ and}\\
\widetilde{D}_{-1}''''&=(Q|_{(\MU_{1-}\cup\,\MU_{3+}\cup\gamma_{b3})})^{-1}(D_0).
\end{align}
These domains are contained in $\C\setminus\overline{\mathscr{D}}_{r_1}$ by Lemma \ref{lema-D0}(b) below. Then one can define
\begin{equation}
D_{-1}=\varphi(\widetilde{D}_{-1}),\quad D_{-1}'''=\varphi(\widetilde{D}_{-1}''') \text{\quad and\quad} D_{-1}''''=\varphi(\widetilde{D}_{-1}'''').
\end{equation}
By the construction, $F$ maps each of $D_0$, $D_0'$ and $D_0''$  homeomorphically onto $D_1$, and maps each of $D_{-1}$, $D_{-1}'''$ and $D_{-1}''''$ homeomorphically onto $D_0$.

\medskip
Let $R=125$ and $R_1=108$ be the constants introduced in \S\ref{sec-definitions}.

\begin{lem}\label{lema-D0}
\textup{(a)} $\widetilde{D}_0\subset \widetilde{W}_0\cap\D(10,R_1+2)$ and $D_0\subset W_0\cap\D(10,R_1+5)$;

\textup{(b)} $\widetilde{D}_0\cup \widetilde{D}_0'\cup \widetilde{D}_0''\cup \widetilde{D}_{-1}\cup \widetilde{D}_{-1}'''\cup \widetilde{D}_{-1}''''\subset \Omega_0\cap\D(0,R_1+12)\cap\MU_{123}\cap(\C\setminus\overline{\mathscr{D}}_{r_1})$;

\textup{(c)} $D_0\cup D_0'\cup D_0''\cup D_{-1}\cup D_{-1}'''\cup D_{-1}''''\subset\D(0,R_1+15)$.
\end{lem}

\begin{proof}
(a) If $|\zeta-10|\geq R_1+2$, then $|\zeta|\geq R_1-8>50$. If $|\zeta|\geq 50$, similar to \eqref{equ:Q-zeta-est} we have $|Q(\zeta)-\zeta-b_0|<1$. Note that $|10+b_0-cv|<1$. Therefore, we have $|Q(\zeta)-cv|=|(Q(\zeta)-\zeta-b_0)+(\zeta-10)+(10+b_0-cv)|\geq |\zeta-10|-|Q(\zeta)-\zeta-b_0|-|10+b_0-cv|>R_1$. By Proposition \ref{prop-Phi}(b), $D_1\subset \D(cv,R_1)$. Therefore
\begin{equation}
\widetilde{D}_0\cup \widetilde{D}_0'\cup \widetilde{D}_0''\subset Q^{-1}(D_1)\subset Q^{-1}(\D(cv,R_1))\subset\D(10,R_1+2).
\end{equation}
Note that $\varphi(\{\zeta\in\C:|\zeta-10|=R_1+2\})$ is a Jordan curve contained in $\D(10,R_1+5)$ by \eqref{equ-varphi-est}. This implies that $D_0\cup D_0'\cup D_0''\subset\D(10,R_1+5)$. By Lemma \ref{lema-W0}(a) and (b), we have $\widetilde{D}_0\subset \widetilde{W}_0$ and $D_0\subset W_0$.

\medskip

(b) By a similar argument as Part (a), we have $\widetilde{D}_{-1}\cup\widetilde{D}_{-1}'''\cup\widetilde{D}_{-1}''''$ $\subset Q^{-1}(D_0)$ $\subset Q^{-1}(\D(10$, $R_1+5))$ $\subset\D(3,R_1+7)$ and $D_{-1}\cup D_{-1}'''\cup D_{-1}''''\subset\D(3,R_1+10)$. Let $\zeta\in\widetilde{D}_0\cup \widetilde{D}_0'\cup \widetilde{D}_0''\cup \widetilde{D}_{-1}\cup \widetilde{D}_{-1}'''\cup \widetilde{D}_{-1}''''$. By the above, we have $\zeta\in\D(10,R_1+2)\cup\D(3,R_1+7)\subset\D(0,R_1+12)$. By definition, $\zeta\in\MU_{123}$.

By Proposition \ref{prop-Phi}(b) and Lemma \ref{lema-W1}, $D_1\subset W_1\subset\V(cv,\frac{3\pi}{4})\subset W_0$. Hence we have $\zeta\in Q^{-1}(D_0\cup D_1)\subset Q^{-1}(W_0)$. Since $R_1+cv<108+17=R$, we have
\begin{equation}\label{relation-R-R1}
D_0\cup D_1\subset W_0\cap(\D(10,R_1+5)\cup\D(cv,R_1))\subset\D(0,R)\setminus\OD(0,\rho).
\end{equation}
By Lemma \ref{lema-W0}(c), we have $\zeta\not\in\OD(a_5,\varepsilon_5)\cup\OD(\overline{a}_5,\varepsilon_5)$.
By Lemma \ref{lema-E-r1}(f), we have $\zeta\in \C\setminus\overline{\mathscr{D}}_{r_1}$. By Lemmas \ref{lema-a4}(a) and \ref{lema-E-r1}(a), we have $\zeta\not\in\OD(\omega,\varepsilon_6)\cup\OD(\overline{\omega},\varepsilon_6)$. According to Lemma \ref{lema-W0}(d), this implies that $\zeta\in\Omega_0$.

\medskip
(c) By \eqref{equ-varphi-est}, the Jordan curve $\varphi(\{\zeta\in\C:|\zeta|=R_1+12\})$ is contained in $\D(0,R_1+15)$. Therefore, by Part (b), we have $D_0\cup D_0'\cup D_0''\cup D_{-1}\cup D_{-1}'''\cup D_{-1}''''\subset\D(0,R_1+15)$.
\end{proof}

\begin{proof}[Proof of Proposition \ref{prop-F}]
We only need to prove (e). Lemma \ref{lema-D0}(c) shows that $\overline{D}_0\cup \overline{D}_0'\cup \overline{D}_0''\cup \overline{D}_{-1}\cup \overline{D}_{-1}'''\cup \overline{D}_{-1}''''\setminus\{cv\}\subset\D(0,R)$ since $R_1+15=123<R=125$.

Let $\zeta\in closure(\widetilde{D}_0\cup \widetilde{D}_0'\cup \widetilde{D}_0''\cup \widetilde{D}_{-1}\cup \widetilde{D}_{-1}'''\cup \widetilde{D}_{-1}'''')\subset\overline{\MU}_{123}$. By Lemma \ref{lema-D0}(b) and Lemma \ref{lema-phi-rho}, we have $\zeta\in\C\setminus \mathscr{D}_{r_1}$ and hence $|\varphi(\zeta)|>\rho$. By Lemma \ref{lema-D0}(b) and Lemma \ref{lema-arg-arg}, we have $\zeta\in\overline{\Omega}_0$ and hence $\varphi(\zeta)\not\in\R_-$.

Let $z\in\overline{D}_0\cup \overline{D}_0'\cup \overline{D}_0''\cup \overline{D}_{-1}\cup \overline{D}_{-1}'''\cup \overline{D}_{-1}''''$. Then $F(z)\in\overline{\MH}_0$ and $0\leq\re \Phia(F(z))\leq 2$. By Lemma \ref{lema-W1}(b), if $z'\in\OV(cv,\frac{\pi}{6})\setminus\{cv\}$, we have $\re \Phia(z')>1$ and hence $\re \Phia(F(z'))>2$. Therefore, $z$ cannot lie in $\OV(cv,\frac{\pi}{6})\setminus\{cv\}$. This ends the proof of the first statement in Proposition \ref{prop-F}(e).

Finally, we need to prove that $\overline{D}_0^+\subset\pi_X(X_{1+})$. Similar to the argument in last paragraph, we have $z\not\in\overline{\mathbb{V}}(cv,\frac{\pi}{6})$ if $z\in\overline{D}_0^+$. By Proposition \ref{prop-Phi}(b) and the definition of $W_1$, we have $D_1^+\subset W_1\subset\mathbb{V}(cv,\frac{3\pi}{4})$. By Lemma \ref{lema-W0}(a) and (b), we have $\widetilde{D}_0^+\subset \widetilde{W}_0$ and $\overline{D}_0^+\subset\overline{W}_0$. Therefore, $\overline{D}_0^+\cap\overline{\D}(0,\rho)=\emptyset$. This implies that $\overline{D}_0^+\subset\pi_X(X_{1+})$.
\end{proof}

\vskip0.2cm
The rest of this section is devoted to the proof of Lemma \ref{lema-W0}.

\begin{proof}[{Proof of Lemma \ref{lema-W0}}]
(a) The boundary $\partial \widetilde{W}_0$ consists of $\ell_0^\pm:\zeta=cp\pm \ii t \,(0\leq t\leq t_0)$ and $\ell_1^\pm:\zeta=cp\pm \ii t_0+s\,e^{\pm\frac{3\pi\ii}{4}}(s\geq 0)$, where
$t_0=6.5\cdot\sqrt{2}-cp \Aeq{5.118}$. It is easy to see that the end point $\ell_0^\pm(t_0)$ coincides with the initial point $\ell_1^\pm(0)$ and one can check that $\ell_1^\pm$ are subsets of the lines $\{\zeta\in\C:pr_\pm(\zeta)=6.5\}$ respectively.

\medskip
(\textbf{Case a.1}) Take $\zeta=cp+\ii t\,(0<t\leq t_0)$ on $\ell_0^+\setminus\{cp\}$. From Lemma \ref{lema-cp-cv-Q} (see also Figure \ref{Fig-dom-Q}) we have
\begin{equation}
Q(\zeta)-cv=\frac{(\zeta-cp)^3(\zeta-cp')^3(\zeta-\nu_1)(\zeta-\overline{\nu}_1)(\zeta-\nu_2)(\zeta-\overline{\nu}_2)}{\zeta(\zeta-\omega)^4(\zeta-\overline{\omega})^4}.
\end{equation}
Then $\arg(Q(\zeta)-cv)=\frac{3}{2}\pi+\widetilde{\theta}(\zeta)$, where
\begin{align}
\widetilde{\theta}(\zeta):=
&~3\arg(\zeta-cp')+\arg(\zeta-\nu_1)+\arg(\zeta-\overline{\nu}_1)+\arg(\zeta-\nu_2)+\arg(\zeta-\overline{\nu}_2) \\
&-\arg\zeta-4\arg(\zeta-\omega)-4\arg(\zeta-\overline{\omega}).
\end{align}
We define
\begin{align}
\theta_1(t)&=3\arctan(\tfrac{t}{cp-cp'}), & \theta_2(t)&=\arctan(\tfrac{t-\im\nu_1}{cp-\re\nu_1}), & \theta_3(t)&=\arctan(\tfrac{t+\im\nu_1}{cp-\re\nu_1}), \\
\theta_4(t)&=\arctan(\tfrac{t-\im\nu_2}{cp-\re\nu_2}), & \theta_5(t)&=\arctan(\tfrac{t+\im\nu_2}{cp-\re\nu_2}), & \theta_6(t)&=\arctan(\tfrac{t}{cp}), \\
\theta_7(t)&=4\arctan(\tfrac{t-\im\omega}{cp-\re\omega}), & \theta_8(t)&=4\arctan(\tfrac{t+\im\omega}{cp-\re\omega}),
\end{align}
and
\begin{align}
\Theta(t',t''):=\theta_1(t')+\theta_2(t')+\theta_3(t')+\theta_4(t')+\theta_5(t')-\theta_6(t'')-\theta_7(t'')-\theta_8(t'').
\end{align}
For $\zeta=cp+\ii t$ with $t\geq 0$, we have $\widetilde{\theta}(\zeta)=\Theta(t,t)$. Note that $\theta_j$ increases on $[0,+\infty)$ for all $1\leq j\leq 8$. So for any $t\in[t',t'']\subset[0,+\infty)$, we have $\Theta(t',t'')\leq \widetilde{\theta}(\zeta)\leq\Theta(t'',t')$.

Define $t_{0,k}=0.5k$ for $0\leq k\leq 8$, $t_{0,k}=0.4k+0.8$ for $9\leq k\leq 10$ and $t_{0,k}=0.2k+2.8$ for $11\leq k\leq 12$. Then for $0\leq k\leq 11$ we have
\renewcommand\theequation{\thesection.\arabic{equation}*}
\begin{align}
-\tfrac{3\pi}{4}\Aeq{-2.3561}\ly\Theta(t_{0,k},t_{0,k+1})\leq \Theta(t_{0,k+1},t_{0,k})\ly\tfrac{\pi}{2}\Aeq{1.5707}. \retainlabel{1.5707-pi-2}
\end{align}
Therefore, $-\tfrac{3\pi}{4}<\widetilde{\theta}(\zeta)<\tfrac{\pi}{2}$ for $\zeta\in\ell_0^+\setminus\{cp\}$. Hence if $\zeta\in\ell_0^+\setminus\{cp\}$, then
\renewcommand\theequation{\thesection.\arabic{equation}}
\begin{equation}\label{equ:arg-zeta-0}
\tfrac{3\pi}{4}<\arg(Q(\zeta)-cv)<2\pi.
\end{equation}
By a similar calculation, we have $-2\pi<\arg(Q(\zeta)-cv)<-\tfrac{3\pi}{4}$ if $\zeta\in\ell_0^-\setminus\{cp\}$.

\medskip
(\textbf{Case a.2}) Take $\zeta=cp+ \ii t_0+s\,e^{\pm\frac{3\pi\ii}{4}}(s\geq 0)$ on $\ell_1^+$. In the following we prove that $pr_+(Q(\zeta))<pr_+(cv)$. From Lemma \ref{lema-Q}, we have
\begin{equation}
Q(\zeta)=\zeta+b_0+\frac{b_1}{\zeta}+\frac{2^4}{5^5}\cdot\frac{a_{1,1}}{(\zeta-\omega)(\zeta-\overline{\omega})}+Q_3(\zeta),
\end{equation}
where
\begin{equation}
\begin{split}
|Q_3(\zeta)|\leq Q_{3,max}(r) :=
&~ \frac{2^4}{5^5}\cdot\frac{a_{0,1}}{r(r-1)^2}
+\frac{2^6}{5^{10}}\cdot\frac{a_{1,2}r+a_{0,2}}{(r-1)^4}\\
&~ +\frac{2^{11}}{5^{14}}\cdot\frac{a_{1,3}r+a_{0,3}}{(r-1)^6}
+\frac{2^{12}}{5^{16}}\cdot\frac{a_{1,4}r+a_{0,4}}{(r-1)^8}
\end{split}
\end{equation}
for $|\zeta|\geq r>1$.

Note that $pr_+(\zeta)=6.5\cos\frac{\pi}{4}$, $pr_+(b_0)=b_0\cos\frac{\pi}{4}$ and $\arctan\big(\frac{t_0}{cp}\big)\leq \arg\zeta< \frac{3\pi}{4}$. By a numerical calculation, we have
\renewcommand\theequation{\thesection.\arabic{equation}*}
\begin{align}
\arctan\big(\tfrac{t_0}{cp}\big) \Aeq{0.8985}\gy \tfrac{\pi}{4} \Aeq{0.7853}. \retainlabel{equ:arct1}
\end{align}
Then $-\pi< \arg\Big(\frac{e^{-\ii \pi/4}}{\zeta}\Big)<-\frac{\pi}{2}$ and hence $pr_+(\frac{1}{\zeta})< 0$. Note that $\arg(\zeta-\omega)< \frac{3\pi}{4}$, $\arg(\zeta-\overline{\omega})< \frac{3\pi}{4}$ and
\renewcommand\theequation{\thesection.\arabic{equation}*}
\begin{align}
\arg(\zeta-\omega)&\geq\arctan(\tfrac{t_0-\im\omega}{cp-\re\omega})\Aeq{0.9082}\gy \tfrac{\pi}{4}\Aeq{0.7853}, \retainlabel{0.785-arg-zeta}\\
\arg(\zeta-\overline{\omega})&\geq\arctan(\tfrac{t_0+\im\omega}{cp-\re\omega})\Aeq{1.0442}\gy \tfrac{\pi}{4}.  \retainlabel{1.044-arg-zeta}
\end{align}
We have $-\frac{7\pi}{4}< \arg\Big(\frac{e^{-\ii \pi/4}}{(\zeta-\omega)(\zeta-\overline{\omega})}\Big)<-\frac{3\pi}{4}$ and $pr_+(\frac{1}{(\zeta-\omega)(\zeta-\overline{\omega})})\leq \frac{1}{(|\zeta|-1)^2}\cos \frac{\pi}{4}$.
Therefore, we have
\renewcommand\theequation{\thesection.\arabic{equation}*}
\begin{align}
pr_+(Q(\zeta))
\leq&~ \big(6.5+b_0\big)\cos\tfrac{\pi}{4}+0+\tfrac{2^4}{5^5}\tfrac{a_{1,1}}{(6.5-1)^2}\cos\tfrac{\pi}{4}+Q_{3,max}(6.5) \notag\\
&~\Aeq{10.73}\ly pr_+(cv)=cv\cos\tfrac{\pi}{4}\Aeq{11.98}. \retainlabel{equ:prQ1}
\end{align}
This implies that
\renewcommand\theequation{\thesection.\arabic{equation}}
\begin{equation}\label{equ:arg-zeta-1}
Q(\ell_1^+)\subset\{z\in\C:\tfrac{3\pi}{4}<\arg(z-cv)<\pi\}.
\end{equation}
By the symmetry of the dynamics of $Q$, we have $Q(\ell_1^-)\subset\{z\in\C:-\pi<\arg(z-cv)<-\frac{3\pi}{4}\}$.

Note that $Q$ maps $[cp,+\infty)$ homeomorphically to $[cv,+\infty)$. By \eqref{equ:arg-zeta-0} and \eqref{equ:arg-zeta-1}, we have $Q^{-1}(\V(cv,\tfrac{3\pi}{4})\cap\BH_+)\cap\MU_1\subset \widetilde{W}_0\cap\BH_+$. By the dynamical symmetry of $Q$, we have $Q^{-1}(\V(cv,\tfrac{3\pi}{4}))\cap\MU_1\subset \widetilde{W}_0$.

\medskip

(b) Suppose $\zeta\in \widetilde{W}_0$. Then $\re \zeta\geq cp$ or $pr_\pm(\zeta)\geq 6.5$.
If $\re \zeta\geq cp$, then
\renewcommand\theequation{\thesection.\arabic{equation}*}
\begin{align}
\re \varphi(\zeta)\geq cp+c_{00}-c_{01,max}-\varphi_{1,max}(cp)\Aeq{1.703}\gy 1.7. \retainlabel{equ:varphi-1}
\end{align}
If $pr_\pm(\zeta)\geq 6.5$, then
\renewcommand\theequation{\thesection.\arabic{equation}*}
\begin{align}
&~ pr_\pm(\varphi(\zeta))=pr_\pm(\zeta)+pr_\pm(c_{00})+pr_\pm(c_0-c_{00})+pr_\pm(\varphi_1(\zeta)) \notag\\
\geq &~6.5+c_{00}\cos\tfrac{\pi}{4}-c_{01,max}-\varphi_{1,max}(6.5)\Aeq{4.223}\gy 4.2.\retainlabel{equ:pmvarphi1}
\end{align}
Therefore, in both cases we have $\varphi(\zeta)\in W_0$.

\medskip
(c) Since this part and Lemma \ref{lema-a4}(b) are related to the position of $\D(a_5,\varepsilon_5)$ and $\D(\overline{a}_5,\varepsilon_5)$, we postpone the proof of this part to Appendix \ref{sec:appendix}.

\medskip

(d) Recall that $h_0=\frac{14\sqrt{6}}{25}+0.45\Aeq{1.82}$ and $\Omega_0$ are defined in Lemma \ref{lema-arg-arg}. Denote $\ell_2^\pm:\zeta=x\pm (h_0-x)\ii \,(-1\leq x\leq\re\omega)$ and $\ell_3^\pm:\zeta=x\pm (h_0+1)\ii \, (x\leq -1)$. By Part (c) and the definition of $\Omega_0$, it suffices to prove that $Q(\zeta)\in \C\setminus \overline{W}_0$ when $\zeta\in\ell_2^\pm\cup\ell_3^\pm$.

Denote $y_2^\pm(x):=\pm(h_0-x)$ and $y_3^\pm(x):=\pm(h_0+1)$.
For $j=2,3$, we consider
\begin{align}
\theta_{j,1}^\pm(x)&=6\arctan_\star\Big(\tfrac{y_j^\pm(x)}{x+1}\Big),~\theta_{j,2}^\pm(x)=4\arctan_\star\Big(\tfrac{y_j^\pm(x)}{x-1}\Big),\\
~\theta_{j,3}^\pm(x)&=\arctan_\star\Big(\tfrac{y_j^\pm(x)}{x}\Big), \\
\theta_{j,4}^\pm(x)&=4\arctan_\star\Big(\tfrac{y_j^\pm(x)-\im\omega}{x-\re\omega}\Big),\quad \theta_{j,5}^\pm(x)=4\arctan_\star\Big(\tfrac{y_j^\pm(x)+\im\omega}{x-\re\omega}\Big),
\end{align}
where $\arctan_\star(t)=\arctan(t)$ if $t\in[0,+\infty]$ and $\arctan_\star(t)=\arctan(t)+\pi$ if $t\in[-\infty,0)$, and
\renewcommand\theequation{\thesection.\arabic{equation}}
\begin{equation}\label{equ:Theta-j-pm}
\Theta_j^\pm(x,x'):=\theta_{j,1}^\pm(x')+\theta_{j,2}^\pm(x')-\theta_{j,3}^\pm(x)-\theta_{j,4}^\pm(x)-\theta_{j,5}^\pm(x).
\end{equation}
We define $\xi_{j,i}^\pm(x)$, where $1\leq i\leq 5$, $j=2,3$, and $\Xi_j^\pm(x,x')$ as in \eqref{equ:Xi-j-pm}.
Then
\begin{equation}
|Q(\zeta)|=\Xi_j^\pm(x,x) \text{\quad and \quad} \arg Q(\zeta)=\Theta_j^\pm(x,x)
\end{equation}
for $\zeta=x+y_j^\pm(x)\ii\in\ell_j^\pm$, where $j=2,3$.
Since the dynamics of $Q$ is symmetric about $\R$, we only consider $Q(\ell_2^+\cup\ell_3^+)$ and the same estimates can be obtained for $Q(\ell_2^-\cup\ell_3^-)$.

\medskip
(\textbf{Case d.1}) Take $\zeta=x+ (h_0-x)\ii\in\ell_2^+$ with $0\leq x\leq \re\omega=\tfrac{8\sqrt{6}-3}{25}\Aeq{0.6638}$. Note that $\theta_{2,i}^+(x)$ decreases on $[-1,\re\omega]$ for $1\leq i\leq 5$. Hence
\begin{equation}
\Theta_2^+(x,x')\leq \arg Q(\zeta)\leq \Theta_2^+(x',x), \quad \text{ for all }-1\leq x\leq\re\zeta\leq x'\leq \re\omega.
\end{equation}
Denote $x_{1,k}=0.04k$ for $0\leq k\leq 16$ and $x_{1,17}=\re\omega$. One can verify that for $0\leq k\leq 16$,
\renewcommand\theequation{\thesection.\arabic{equation}*}
\begin{align}
-\tfrac{5\pi}{4}&\ly\Theta_2^+(x_{1,k},x_{1,k+1})\leq\Theta_2^+(x_{1,k+1},x_{1,k})\ly -\tfrac{3\pi}{4}. \retainlabel{pi-4-d1}
\end{align}
This implies that $Q(\zeta)\in\C\setminus \overline{W}_0$ for $\zeta=x+ (h_0-x)\ii\in\ell_2^+$ for $0\leq x\leq \re\omega$.

\medskip
(\textbf{Case d.2}) Take $\zeta=x+ (h_0-x)\ii\in\ell_2^+$ with $-1\leq x\leq 0$.
Note that $\xi_{2,i}^+(x)$ decreases on $[-1,0]$ for $1\leq i\leq 5$. Therefore, we have
\begin{equation}
\Xi_2^+(x,x')\leq |Q(\zeta)|\leq \Xi_2^+(x',x), \quad \text{ for all }-1\leq x\leq\re\zeta\leq x'\leq 0.
\end{equation}
Define $x_{2,k}=-1+0.025k$ for $0\leq k\leq 4$ and $x_{2,k}=-1.1+0.05k$ for $5\leq k\leq 22$. We consider the following two functions:
\begin{equation}
\begin{split}
h_2^+(x,x') & :=\Xi_2^+(x',x)\cos \Theta_2^+(x',x), \text{ and}\\
v_2^+(x,x') & :=\Xi_2^+(x',x)\sin \Theta_2^+(x',x)-\big(h_2^+(x,x')-\tfrac{21}{5}\sqrt{2}\big).
\end{split}
\end{equation}
One can verify that for $0\leq k\leq 3$,
\renewcommand\theequation{\thesection.\arabic{equation}*}
\begin{align}
-\tfrac{\pi}{2} &\ly \Theta_2^+(x_{2,k},x_{2,k+1})\leq\Theta_2^+(x_{2,k+1},x_{2,k})\ly 0, \retainlabel{x-2-2-k-1} \\
h_2^+(x_{2,k},x_{2,k+1})&\ly 1.7 \text{\quad and\quad} -\Xi_2^+(x_{2,k+1},x_{2,k})-(1.7-\tfrac{21}{5}\sqrt{2})\gy 0. \retainlabel{x-2-2-k-2}
\end{align}
For $4\leq k\leq 21$, we have
\renewcommand\theequation{\thesection.\arabic{equation}*}
\begin{align}
-\pi\ly \Theta_2^+(x_{2,k},x_{2,k+1})&\leq\Theta_2^+(x_{2,k+1},x_{2,k})\ly -\tfrac{\pi}{4}\Aeq{-0.7853}, \quad\qquad \retainlabel{x-2-2-k-3} \\
h_2^+(x_{2,k},x_{2,k+1})&\ly 1.7 \text{\quad and\quad} v_2^+(x_{2,k},x_{2,k+1})\gy 0. \retainlabel{x-2-2-k-4}
\end{align}
This implies that $Q(\zeta)\in\C\setminus \overline{W}_0$ for $\zeta=x+ (h_0-x)\ii\in\ell_2^+$ with $-1\leq x\leq 0$.

\medskip
(\textbf{Case d.3}) Take $\zeta=x+(h_0+1)\ii\in\ell_3^+$ with $-8\leq x\leq-1$. Note that both $\theta_{3,i}^+(x)$ and $\xi_{3,i}^+(x)$ decrease in $(-\infty,-1]$ for $1\leq i\leq 5$. If $x\leq \re\zeta\leq x'\leq -1$, then
\begin{equation}
\Theta_3^+(x,x')\leq \arg Q(\zeta)\leq \Theta_3^+(x',x) \text{\quad and\quad}\Xi_3^+(x,x')\leq |Q(\zeta)|\leq \Xi_3^+(x',x).
\end{equation}
We define $x_{3,k}=-1-0.04k$ for $0\leq k\leq 20$ and $x_{3,k}=2.2-0.2k$ for $21\leq k\leq 29$. Note that $x_{3,k}>x_{3,k+1}$ for $0\leq k\leq 28$. One can verify that if $0\leq k\leq 7$, then
\renewcommand\theequation{\thesection.\arabic{equation}*}
\begin{align}
-\tfrac{\pi}{2} \ly &~\Theta_3^+(x_{3,k+1},x_{3,k})\leq \Theta_3^+(x_{3,k},x_{3,k+1})\ly 0, \retainlabel{x-3-2-k-5} \\
&~\Xi_3^+(x_{3,k},x_{3,k+1})\ly \tfrac{21}{5}\sqrt{2}-1.7\Aeq{4.2396} \text{\quad and} \retainlabel{x-3-2-k-7} \\
&~\Xi_3^+(x_{3,k},x_{3,k+1}) \cos\Theta_3^+(x_{3,k},x_{3,k+1})\ly 1.7. \retainlabel{x-3-2-k-6}
\end{align}
If $8\leq k\leq 28$, we have
\renewcommand\theequation{\thesection.\arabic{equation}*}
\begin{align}
\Xi_3^+(x_{3,k},x_{3,k+1})\ly 1.7. \retainlabel{1.7-Xi-3-x3}
\end{align}
This implies that $Q(\zeta)\in\C\setminus \overline{W}_0$ for $\zeta=x+ (h_0+1)\ii\in\ell_3^+$ with $x\in[-3.6,-1]$.

\medskip
For $30\leq k\leq 40$, we define $x_{3,k}=-(3.6+0.4(k-29))=8-0.4k$. Consider the following two functions:
\begin{equation}
\begin{split}
h_3^+(x,x') & :=\Xi_3^+(x,x')\cos \Theta_3^+(x',x), \text{ and}\\
v_3^+(x,x') & :=\Xi_3^+(x,x')\sin \Theta_3^+(x',x)-\big(\tfrac{21}{5}\sqrt{2}-h_3^+(x,x')\big).
\end{split}
\end{equation}
One can verify that if $29\leq k\leq 39$, then
\renewcommand\theequation{\thesection.\arabic{equation}*}
\begin{align}
\tfrac{\pi}{4}\Aeq{0.7853}&\ly \Theta_3^+(x_{3,k+1},x_{3,k})\leq \Theta_3^+(x_{3,k},x_{3,k+1})\ly \pi, \quad\qquad \retainlabel{x-3-2-k-33} \\
h_3^+(x_{3,k},x_{3,k+1})&\ly 1.7 \text{\quad and\quad} v_3^+(x_{3,k},x_{3,k+1})\ly 0. \retainlabel{x-3-2-k-44}
\end{align}
This implies that $Q(\zeta)\in\C\setminus \overline{W}_0$ for $\zeta=x+ (h_0+1)\ii\in\ell_3^+$ with $x\in[-8,-3.6]$.

\medskip
(\textbf{Case d.4}) Take $\zeta=x+(h_0+1)\ii\in\ell_3^+$ with $x\leq-8$. We claim that $Q(\zeta)\in\C\setminus\OV(1.6+(h_0+1)\ii,\frac{3\pi}{4})$. Denote $\widetilde{Q}(\zeta):=Q(\zeta)-(1.6+(h_0+1)\ii)$. It suffices to prove that $|\im \widetilde{Q}(\zeta)|<-\re \widetilde{Q}(\zeta)$. By Lemma \ref{lema-Q}, if $\zeta=x+(h_0+1)\ii$ with $x\leq -8$, we have $\widetilde{Q}(\zeta)=x-1.6+b_0+\tfrac{b_1}{x+(h_0+1)\ii}+Q_2(\zeta)$.
Note that $|\im\tfrac{b_1}{x+(h_0+1)\ii}|+\re\tfrac{b_1}{x+(h_0+1)\ii}<0$ since $h_0+1<|x|$. We have
\begin{align}\retainlabel{0.1202-im}
|\im \widetilde{Q}(\zeta)|+\re\widetilde{Q}(\zeta)\leq 2Q_{2,max}(8)-8-1.6+b_0\Aeq{-0.9396}\ly 0.
\end{align}
This implies that $Q(\zeta)\in\C\setminus \overline{W}_0$ for $\zeta=x+ (h_0+1)\ii\in\ell_3^+$ with $x\leq -8$  and the proof is finished.
\end{proof}

\section{Proofs of Proposition \ref{prop-E_F-P} and the Main Theorem}~\label{sec-para-renorm}

In this section, we give the proof of Proposition \ref{prop-E_F-P} and hence the Main Theorem.

\begin{proof}[{Proof of Proposition \ref{prop-E_F-P}}]
By Proposition \ref{prop-F}, the sets $\overline{D}_0$, $\overline{D}_0'$, $\overline{D}_0''$, $\overline{D}_{-1}$, $\overline{D}_{-1}'''$, $\overline{D}_{-1}''''$ and $\overline{D}_0^+$ are contained in $\pi_X(X_{1+}\cup X_{2- })$. Hence they can also be regarded as subsets of $X_{1+}\cup X_{2-}$. For $n=1,2,\ldots$, define
\begin{align}
\qquad D_{-n-1}:=&~g^n(D_{-1}),  & D_{-n}':=&~g^n(D_0'), & D_{-n}'':=&~g^n(D_0''), \\
\qquad D_{-n-1}''':=&~g^n(D_{-1}'''),  & D_{-n-1}'''':=&~g^n(D_{-1}''''), & D_{-n}^+:=&~g^n(D_0^+).
\end{align}
The Fatou coordinate $\Phia$ extends naturally to $\widetilde{\Phi}_{attr}$ on these domains and their closures. Let
\begin{equation}
\begin{split}
\SD =&~\{w\in\C:0<\re w<1 \text{ and } |\im w|<\eta\}, \quad\text{and} \\
\SD^+  =&~\{w\in\C:0<\re w<1 \text{ and } \im w>\eta\}.
\end{split}
\end{equation}
We name the boundary segments of $\SD$ and $\SD^+$ as follows (see Figure \ref{Fig-Psi0}):
\begin{align}
\qquad \partial_\pm^l\SD &= 0+[0,\pm\,\eta]\ii, & \partial_\pm^r\SD &= 1+[0,\pm\,\eta]\ii, & \partial_\pm^h\SD &= [0,1]\pm\eta\ii,\\
\qquad \partial^h\SD^+ &= \partial_+^h\SD, & \partial^l\SD^+ &= 0+[\eta,+\infty]\ii, & \partial^r\SD^+ &= 1+[\eta,+\infty]\ii.
\end{align}

\begin{figure}[!htpb]
  \setlength{\unitlength}{1mm}
  \centering
  \includegraphics[width=0.77\textwidth]{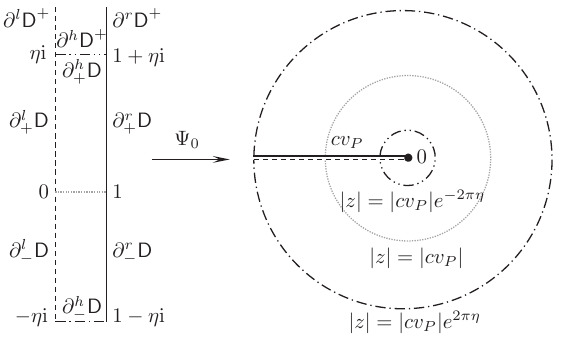}
  \caption{The domain and the range of $\Psi_0(z)=cv_P\, e^{2\pi\ii z}$. These two pictures are just topologically correct but not conformally precise.}
  \label{Fig-Psi0}
\end{figure}

Note that $\Phia(z)-1$ maps $D_1$ and $D_1^+$ homeomorphicially onto $\SD$ and $\SD^+$ including their boundaries respectively. Hence we can name the boundary segments of $D_1$ and $D_1^+$ by $\partial_+^l D_1$ and $\partial^h D_1^+$, etc., according to their images under $\Phia(z)-1$. In the following, the same naming convention will be applied to domains that are mapped homeomorphically onto $D_1$ and $D_1^+$ by the iterates of $F$ and by $Q$.

\medskip
According to the definitions of $D_n$, $D_n'$, $D_n''$, $D_{n-1}'''$, $D_{n-1}''''$ and $D_n^+$, where $n=1$, $-1$, $-2, \ldots$, we have the following relations (see Figure \ref{Fig-chessboard}):
\begin{itemize}
\item $g(D_0)=D_{-1}$ and $g(D_1^+)=D_0^+$;
\item $\overline{D}_n\cap \overline{D}_{n-1}=\partial_+^l D_n=\partial_+^r D_{n-1}$, $\overline{D}_{n-1}\cap \overline{D}_n'=\partial_-^r D_{n-1}=\partial_-^l D_n'$;
\item $\overline{D}_n'\cap\overline{D}_{n-1}'''=\partial_+^l D_n'=\partial_+^r D_{n-1}'''$, $\overline{D}_{n-1}'''\cap\overline{D}_n''=\partial_-^r D_{n-1}'''=\partial_-^l D_n''$;
\item $\overline{D}_n''\cap\overline{D}_{n-1}''''=\partial_+^l D_n''=\partial_+^r D_{n-1}''''$, $\overline{D}_{n-1}''''\cap\overline{D}_n=\partial_-^r D_{n-1}''''=\partial_-^l D_n$;

\item The following sets are equal to each other, which is a singleton:  $\overline{D}_n\cap \overline{D}_n'\cap \overline{D}_n''$, $\overline{D}_{n-1}\cap \overline{D}_{n-1}'''\cap \overline{D}_{n-1}''''$, $\overline{D}_n\cap \overline{D}_n'$, $\overline{D}_n\cap \overline{D}_{n-1}'''$, $\overline{D}_n\cap \overline{D}_n''$, $\overline{D}_{n-1}\cap \overline{D}_{n-1}'''$, $\overline{D}_{n-1}\cap \overline{D}_n''$, $\overline{D}_{n-1}\cap \overline{D}_{n-1}''''$, $\overline{D}_n'\cap \overline{D}_n''$, $\overline{D}_n'\cap \overline{D}_{n-1}''''$, $\overline{D}_{n-1}'''\cap \overline{D}_{n-1}''''$;

\item $\overline{D}_n\cap \overline{D}_n^+=\partial_+^h D_n=\partial^h D_n^+$, $\overline{D}_n^+\cap \overline{D}_{n-1}^+=\partial^l D_n^+=\partial^r D_{n-1}^+$;
\item $\overline{D}_n\cap \overline{D}_{n-1}^+=\overline{D}_{n-1}\cap \overline{D}_{n}^+ = \text{a point}$.
\end{itemize}

Define (see Figures \ref{Fig-U-eta-P-log} and \ref{Fig-dom-P}):
\begin{align}
\MU & = \MU_{1+}^P\cup\MU_{1-}^P\cup\gamma_{c1}^P,  &\MU' & = \MU_{2-}^P\cup\MU_{4+}^P\cup\gamma_{c4}^P, & \MU'' & = \MU_{3-}^P\cup\MU_{5+}^P\cup\gamma_{c5}^P,\\
\MU''' & =\MU_{2+}^P\cup\MU_{5-}^P\cup\gamma_{c2}^P,  & \MU'''' & = \MU_{3+}^P\cup\MU_{4-}^P\cup\gamma_{c3}^P.
\end{align}
It is easy to see that $P$ maps each of these $5$ domains onto $\C\setminus(-\infty,0]$. The map $\Psi_0(z)=cv_P\,e^{2\pi\ii z}$ defined in Proposition \ref{prop-E_F-P} satisfies (see Figure \ref{Fig-Psi0})
\begin{align}
  \Psi_0(\SD) & = \{z\in\C: |cv_P|\,e^{-2\pi\eta}<|z|<|cv_P|\,e^{2\pi\eta}\}\setminus(-\infty,0], \text{ and}\\
  \Psi_0(\SD^+) & = \{z\in\C: |z|<|cv_P|\,e^{-2\pi\eta}\}\setminus(-\infty,0].
\end{align}

We now define $\Psi_1$ in the interior of $D_n$ etc., as following, where $n\leq 0$:
\begin{equation}
\Psi_1:=
\left\{
\begin{aligned}
&(P|_{\MU})^{-1}\circ\Psi_0\circ\widetilde{\Phi}_{attr}  \text{~~~~~~~on~} D_n\cup D_n^+\\
&(P|_{\MU'})^{-1}\circ\Psi_0\circ\widetilde{\Phi}_{attr}  \text{~~~~~~\,on~} D_n'\\
&(P|_{\MU''})^{-1}\circ\Psi_0\circ\widetilde{\Phi}_{attr}  \text{~~~~~~on~} D_n''\\
&(P|_{\MU'''})^{-1}\circ\Psi_0\circ\widetilde{\Phi}_{attr}  \text{~~~~~\,on~} D_{n-1}'''\\
&(P|_{\MU''''})^{-1}\circ\Psi_0\circ\widetilde{\Phi}_{attr}  \text{~~~~~\,on~} D_{n-1}''''.
\end{aligned}
\right.
\end{equation}
Then $\Psi_1$ is a homeomorphism from each domain onto its image, and it extends continuously to the closure of each domain. Note that $\Psi_1$ is holomorphic in the interior of each domain since $\Psi_1$ is a branch of $P^{-1}\circ\Psi_0\circ\widetilde{\Phi}_{attr}$ on each domain. In order to prove that $\Psi_1$ is holomorphic in the whole domain, we just need to show that on the common boundary of any two domains above, the two extensions are consistent.

Now we check the consistent condition. If $z\in D_n$ tends to $\partial_+^l D_n$, then $\Psi_0\circ\widetilde{\Phi}_{attr}(z)$ tends to $[cv_P,0)=\Gamma_a^P\cup\{cv_P\}$ from the lower side. Hence $\Psi_1(z)\in\MU$ tends to $[cp_P,0)=\gamma_{a1}^P$ from the lower side (see Figures \ref{Fig-U-eta-P-log}, \ref{Fig-dom-P} and \ref{Fig-Psi0}). If $z\in D_{n-1}$ tends to $\partial_+^r D_{n-1}$ from another side, then $\Psi_0\circ\widetilde{\Phi}_{attr}(z)$ tends to $[cv_P,0)=\Gamma_a^P\cup\{cv_P\}$ from the upper side, and $\Psi_1(z)\in\MU$ tends to $[cp_P,0)=\gamma_{a1}^P$ from the upper side. The map $\Psi_1$ matches completely along $\overline{D}_n\cap \overline{D}_{n-1}=\partial_+^l D_n=\partial_+^r D_{n-1}$ since $P$ is a homeomorphism in a neighborhood of each $\gamma_{ai}^P$. It follows that $\Psi_1$ is holomorphic there. Similarly, one can check the rest easily. The map $\Psi_1$ is defined in $U=\text{the interior of }\bigcup_{n=-\infty}^0 (\overline{D}_n\cup \overline{D}_n'\cup \overline{D}_n''\cup \overline{D}_{n-1}'''\cup \overline{D}_{n-1}''''\cup \overline{D}_n^+)$. Now it is obvious that $P\circ\Psi_1=\Psi_0\circ\widetilde{\Phi}_{attr}$ and $\Psi_1: U\to U_\eta^P\setminus\{0\}=V'\setminus\{0\}$ is a surjection. The assertions (b), (c) and (d) in Proposition \ref{prop-E_F-P} now are straightforward and the Main Theorem is proved.
\end{proof}

\section{Remarks on the constants}\label{sec:rmk-cst}

The important constants are $\eta$, $\rho$, $R$ and $r_1$.
The number $\eta$ is determined by \eqref{eta-geq} and \eqref{eta-leq}. There, $\eta$ can be chosen between $3$ and $8$. By \eqref{eta-leq}, the upper bound of $\eta$ is effected by $R_2$ but actually, it is effected by $R$ since $R_2$ and $R_1$ are related by $R_2+9=R_1$ (see Lemma \ref{lema-W2}), and $R_1$, $R$ need to satisfy $R_1+cv<R$ (see \eqref{relation-R-R1}). However, $R$ cannot be too large, the upper bound of $R$ is determined by Lemma \ref{lema-E-r1}(a)(c) since $\overline{\mathscr{D}}_{r_1}$ needs to be covered by $7$ disks and as two of them, the disks $\D(\omega,\varepsilon_6)$ and $\D(\overline{\omega},\varepsilon_6)$ are affected by $R$ (the radius $\varepsilon_6$ decreases as $R$ increases). Similarly, $\rho$ cannot be too small by Lemma \ref{lema-E-r1}(b)(c) and cannot be too large by Lemma \ref{lema-phi-rho}. In fact, Lemmas \ref{lema-E-r1} and \ref{lema-phi-rho} imply that the constant $r_1$ can neither be too large nor too small.

\medskip
The choice of the domain $\Omega_0$ is also crucial. On the one hand, this domain cannot be too large since for every $\zeta\in\Omega_0$, we need to guarantee that $\varphi(\zeta)\not\in\R_-$ (see Lemma \ref{lema-arg-arg}). On the other hand, it cannot be too small since it should cover the set $\big(Q^{-1}(W_0)\cap\MU_{123}\big)\setminus\big(\OD(\omega,\varepsilon_6)\cup\OD(\overline{\omega},\varepsilon_6)\big)$ (see Lemma \ref{lema-W0}(d)). The suitable choice of the location of $\D(a_5,\varepsilon_5)$ guarantees the above properties of $\Omega_0$.

\medskip
The constant $u_{2,\theta_1}$ defined at the beginning of \S\ref{sec-Phi-attr} is determined by Lemma \ref{lema-esti-F} and also Lemma \ref{lema-injective}(c).

\appendix
\section{The position of $\D(a_4,\varepsilon_4)$ and $\D(a_5,\varepsilon_5)$}\label{sec:appendix}

In this appendix, we give the proofs of the statements of the positions of $\D(a_4,\varepsilon_4)$ and $\D(a_5,\varepsilon_5)$ in Lemma \ref{lema-a4} and Lemma \ref{lema-W0}(c). This can be observed clearly from Figure \ref{Fig-a4-a5}.

\begin{figure}[!htpb]
  \setlength{\unitlength}{1mm}
  \centering
  \includegraphics[width=0.96\textwidth]{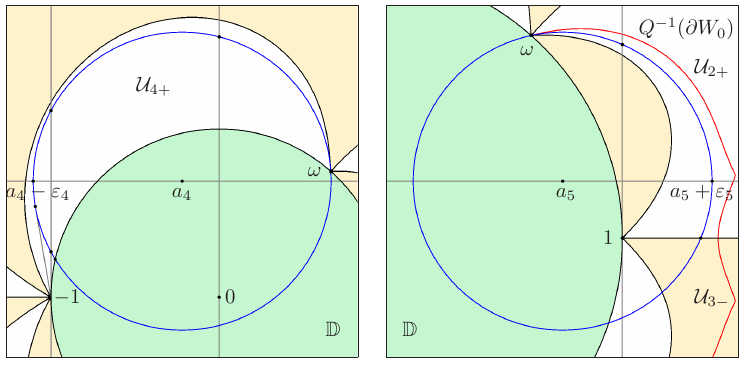}
  \caption{The position of $\D(a_4,\varepsilon_4)$ and $\D(a_5,\varepsilon_5)$, and some special points are marked. It can be seen clearly  that $\OD(a_4,\varepsilon_4)\setminus\OD\subset\MU_{4+}$, $\big(\partial\D(a_5,\varepsilon_5)\cap\BH_+\big)\setminus\OD\subset \MU_{2+}$ and $Q^{-1}(\partial W_0)\cap \MU_{2+}\cap \OD(a_5,\varepsilon_5)=\emptyset$.}
  \label{Fig-a4-a5}
\end{figure}

For $j=4,5$, let $y_j^\pm(x)$, $\xi_{j,i}^\pm(x)$, $\Xi_j^\pm(x,x')$ be defined in \eqref{equ:y-j-pm}, \eqref{equ:Xi-j-pm}, and let $\theta_{j,i}^\pm(x)$, $\Theta_j^\pm(x,x')$ be defined as in \eqref{equ:Theta-j-pm}, where $1\leq i\leq 5$. Then we have $|Q(\zeta)|=\Xi_j^\pm(x,x)$ and $\arg Q(\zeta)=\Theta_j^\pm(x,x)$ for $\zeta=x+y_j^\pm(x)\ii$, where $j=4,5$.

\begin{proof}[{Proof of Lemma \ref{lema-a4}(a)}]
Recall that $\omega=\tfrac{8\sqrt{6}-3}{25}+\tfrac{6\sqrt{6}+4}{25}\ii$. The intersection of $\partial\D$ and $\partial\D(a_4,\varepsilon_4)$ with $\varepsilon_4=|a_4-\omega|$ consists of two points, one is $\omega$ and the other is $\zeta_{4,2}^-= x_{4,2}^-+y_{4,2}^-\ii$, where
\renewcommand\theequation{\Alph{section}.\arabic{equation}*}
\begin{align}
x_{4,2}^-= &~\frac{2\,\re a_4\,\im a_4\,\im \omega +((\re a_4)^2-(\im a_4)^2)\,\re \omega}{(\re a_4)^2+(\im a_4)^2} \Aeq{-0.9742}, \retainlabel{zeta44:0.9742} \\
y_{4,2}^-= &~\frac{2\,\re a_4\,\im a_4\,\re \omega  -((\re a_4)^2-(\im a_4)^2)\,\im \omega}{(\re a_4)^2+(\im a_4)^2} \Aeq{0.2255}. \retainlabel{zeta44:0.2255}
\end{align}
The arc $\partial\D(a_4,\varepsilon_4)\setminus\D$ can be divided into $5$ subarcs:
\begin{align}
\ell_{4,p}^-(x)&:=x+y_4^-(x)\ii, \quad x_{4,p}^-\leq x< x_{4,p+1}^-,~p=0,1, \text{ and}\\
\ell_{4,q}^+(x)&:=x+y_4^+(x)\ii, \quad x_{4,q}^+\leq x< x_{4,q+1}^+,~q=0,1,2,
\end{align}
where $x_{4,0}^\pm=\re a_4-\varepsilon_4\Aeq{-1.1057}$,  $x_{4,1}^\pm=-1$, $x_{4,2}^+=0$ and $x_{4,3}^+=\re \omega$ $\Aeq{0.6638}$.
We denote
\begin{equation}
\widetilde{\Theta}_4^\pm(x,x'):=\Theta_4^\pm(x,x')+2\pi.
\end{equation}
To prove this part, it is sufficient to show that $0<\Theta_4^\pm(x,x)=\arg Q(\zeta)+2\pi<\pi$ for all $\zeta=x+y_4^\pm(x)\ii\in\ell_{4,0}^-\cup\ell_{4,1}^-\cup\ell_{4,0}^+\cup \ell_{4,1}^+\cup \ell_{4,2}^+$.

\medskip

Let $y=s_{4,1}(x+1)$ $(s_{4,1}<0)$ be the line which is tangent to $\partial \D(a_4,\varepsilon_4)$ with smaller slope. A direct calculation shows that
\renewcommand\theequation{\Alph{section}.\arabic{equation}*}
\begin{align}
s_{4,1}=&-\frac{(1+\re a_4)\im a_4+\varepsilon_4\sqrt{(1+\re a_4)^2+(\im a_4)^2-\varepsilon_4^2}}{\varepsilon_4^2-(1+\re a_4)^2} \\
&\Aeq{-5.8104} \retainlabel{5.8104-k-41}
\end{align}
and the real part of the tangent point $\widetilde{\zeta}_{4,1}$ is
\renewcommand\theequation{\Alph{section}.\arabic{equation}*}
\begin{align}
\widetilde{x}_{4,1}=\frac{\re a_4+(\im a_4) s_{4,1}-s_{4,1}^2}{1+s_{4,1}^2}\Aeq{-1.0928} \retainlabel{1.0928-k-41}
\end{align}
Let $y=s_{4,2}(x-1)$ $(s_{4,2}<0)$ be the line which is tangent to $\partial \D(a_4,\varepsilon_4)$. A direct calculation shows that
\begin{equation}
s_{4,2}=-\frac{(1-\re a_4)\im a_4+\varepsilon_4\sqrt{(1-\re a_4)^2+(\im a_4)^2-\varepsilon_4^2}}{(1-\re a_4)^2-\varepsilon_4^2}
\end{equation}
and the real part of the tangent point is
\renewcommand\theequation{\Alph{section}.\arabic{equation}*}
\begin{align}\label{equ:tilde-x-4-2}
\widetilde{x}_{4,2}=\frac{\re a_4+(\im a_4) s_{4,2}+s_{4,2}^2}{1+s_{4,2}^2}\Aeq{0.6051}.
\end{align}
Obviously, the tangent points of $\partial\D(a_4,\varepsilon_4)$ and the line passing through $\overline{\omega}$ are not contained in $\partial\D(a_4,\varepsilon_4)\setminus\D$ since $\im a_4=0.69<\im\omega\Aeq{0.7478}$.

\medskip
(\textbf{Case 1}) We first assume that $\zeta(x)=x+y_4^-(x)\ii\in\ell_{4,0}^-$ with $x\in[x_{4,0}^-,-1)$. Define $x_{4,0,1}^-=-1.096$ such that
$x_{4,0}^-\Aeq{-1.1057}<x_{4,0,1}^-<\widetilde{x}_{4,1}\Aeq{-1.0928}<x_{4,1}^-=-1$.
It is easy to see that $\arg (\zeta(x)+1)$, $\arg (\zeta(x)-1)$, $\arg \zeta(x)$, $\arg (\zeta(x)-\omega)$ and $\arg (\zeta(x)-\overline{\omega})$ increase on $[x_{4,0}^-, \widetilde{x}_{4,1}]$. Hence if $x\in[x_{4,0}^-,x_{4,0,1}^-]$, then
\renewcommand\theequation{\Alph{section}.\arabic{equation}*}
\begin{align}
\arg Q(\zeta)+2\pi & \leq \widetilde{\Theta}_4^-(x_{4,0}^-,x_{4,0,1}^-)\Aeq{3.1341} \ly \pi, \text{\quad and} \retainlabel{3.1341-argQ}\\
\arg Q(\zeta)+2\pi & \geq \widetilde{\Theta}_4^-(x_{4,0,1}^-,x_{4,0}^-)\Aeq{2.2450}\gy  0.  \retainlabel{2.2450-argQ}
\end{align}
If $x\in[x_{4,0,1}^-,-1)$, then $\tfrac{\pi}{2}< \arg (\zeta(x)+1)\leq \pi+\arctan(s_{4,1})$, $\arg (\zeta(x)-1)$, $\arg \zeta(x)$, $\arg (\zeta(x)-\omega)$ and $\arg (\zeta(x)-\overline{\omega})$ increase on $[x_{4,0,1}^-,-1)$. We have
\renewcommand\theequation{\Alph{section}.\arabic{equation}*}
\begin{align}
\arg Q(\zeta)+2\pi \leq
&~ 6\arctan(s_{4,1})+\theta_{4,2}^-(-1)-\theta_{4,3}^-( x_{4,0,1}^-) -\theta_{4,4}^-( x_{4,0,1}^-) \\
&~-\theta_{4,5}^-( x_{4,0,1}^-)+3\pi \Aeq{3.0845}\ly\pi \retainlabel{3.0845-argQ}
\end{align}
and
 \renewcommand\theequation{\Alph{section}.\arabic{equation}*}
 \begin{align}
\arg Q(\zeta)+2\pi \geq
&~ 6\cdot(-\tfrac{\pi}{2})+\theta_{4,2}^-(x_{4,0,1}^-)-\theta_{4,3}^-(-1) -\theta_{4,4}^-(-1) \\
&~-\theta_{4,5}^-(-1)+3\pi \Aeq{0.3000}\gy 0 \retainlabel{0.3000-argQ}
\end{align}
This implies that $0< \arg Q(\zeta)+2\pi<\pi$ for $\zeta\in\ell_{4,0}^-$.

\medskip
(\textbf{Case 2}) To prove that $\ell_{4,1}^-\subset\MU_{4+}$, we consider the line $y_{4,1}^-(x)=s_{4,1}(x+1)$ which is tangent to $\partial \D(a_4,\varepsilon_4)$, where $\widetilde{x}_{4,1}\leq x< -1$, and denote
\begin{align}
\Theta_{4,1}^-(x,x'):=
&~6\arctan(s_{4,1})+4\arctan\Big(\tfrac{y_{4,1}^-(x')}{x'-1}\Big)-\arctan\Big(\tfrac{y_{4,1}^-(x)}{x}\Big) \\
&-4\arctan\Big(\tfrac{y_{4,1}^-(x)-\im\omega}{x-\re\omega}\Big)-4\arctan\Big(\tfrac{y_{4,1}^-(x)+\im\omega}{x-\re\omega}\Big)+3\pi.
\end{align}
Define $x_{4,1,1}^-=-1.04$. Similar to (Case 1) we have
 \renewcommand\theequation{\Alph{section}.\arabic{equation}*}
\begin{align}
\Theta_{4,1}^-(\widetilde{x}_{4,1},x_{4,1,1}^-)&\Aeq{3.0860}\ly\pi, & \Theta_{4,1}^-(x_{4,1,1}^-,\widetilde{x}_{4,1})&\Aeq{1.1461}\gy  0, \retainlabel{1.1461-Theta-4-1}\\
\Theta_{4,1}^-(x_{4,1,1}^-,-1)&\Aeq{2.1557}\ly\pi, & \Theta_{4,1}^-(-1,x_{4,1,1}^-)&\Aeq{0.5688}\gy  0. \retainlabel{0.5688-Theta-4-1}
\end{align}
Therefore, if $\zeta$ lies on the line $y_{4,1}^-(x)=s_{4,1}(x+1)$ with $\widetilde{x}_{4,1}\leq x< -1$, then $0<\arg Q(\zeta)+2\pi<\pi$.

\medskip
(\textbf{Case 3}) Assume that $\zeta(x)=x+y_4^+(x)\ii\in\ell_{4,0}^+$ with $x\in[x_{4,0}^+,-1)$. It is easy to see that $\arg (\zeta(x)+1)$, $\arg (\zeta(x)-1)$, $\arg \zeta(x)$, $\arg (\zeta(x)-\omega)$ and $\arg (\zeta(x)-\overline{\omega})$ decrease on $[x_{4,0}^+,-1]$. Hence we have
$\widetilde{\Theta}_4^+(x,x')\leq \arg Q(\zeta)+2\pi\leq \widetilde{\Theta}_4^+(x',x)$ for any $x_{4,0}^+\leq x\leq\re\zeta\leq x'\leq -1$.
Define $x_{4,0,0}^+=x_{4,0}^+$, $x_{4,0,k}^+=-1.11+0.005k$ for $1\leq k\leq 6$ and $x_{4,0,k}^+=-1.14+0.01k$ for $7\leq k\leq 14$. One can verify that for $0\leq k\leq 13$, then
\renewcommand\theequation{\Alph{section}.\arabic{equation}*}
 \begin{align}
\widetilde{\Theta}_4^+(x_{4,0,k+1}^+,x_{4,0,k}^+)\ly\pi \text{\quad and\quad} \widetilde{\Theta}_4^+(x_{4,0,k}^+,x_{4,0,k+1}^+)\gy  0, \retainlabel{1.1210-Theta-4-1}
\end{align}
Hence we have $0<\arg Q(\zeta)+2\pi<\pi$ for $\zeta\in\ell_{4,0}^+$.

\medskip
(\textbf{Case 4}) Assume that $\zeta(x)=x+y_4^+(x)\ii\in\ell_{4,1}^+$ with $x\in[-1,0)$. Similar to (Case 3), $\arg (\zeta(x)+1)$, $\arg (\zeta(x)-1)$, $\arg \zeta(x)$, $\arg (\zeta(x)-\omega)$ and $\arg (\zeta(x)-\overline{\omega})$ decrease on $[-1,0]$. Define $x_{4,1,k}^+=-1+0.01k$ for $0\leq k\leq 20$, $x_{4,1,k}^+=-1.2+0.02k$ for $21\leq k\leq 40$ and $x_{4,1,k}^+=-2+0.04k$ for $41\leq k\leq 50$. Then one can verify that for $0\leq k\leq 49$,
\renewcommand\theequation{\Alph{section}.\arabic{equation}*}
 \begin{align}
\widetilde{\Theta}_4^+(x_{4,1,k+1}^+,x_{4,1,k}^+)\ly\pi \text{\quad and\quad} \widetilde{\Theta}_4^+(x_{4,1,k}^+,x_{4,1,k+1}^+)\gy  0, \retainlabel{0000-Theta-4-1}
\end{align}
This implies that $0<\arg Q(\zeta)+2\pi<\pi$ for $\zeta\in\ell_{4,1}^+$.

\medskip
(\textbf{Case 5}) Assume that $\zeta(x)=x+y_4^+(x)\ii\in\ell_{4,2}^+$ with $x\in[0,\re\omega)$. It is easy to see that $\arg (\zeta(x)+1)$, $\arg \zeta(x)$, $\arg (\zeta(x)-\omega)$ and $\arg (\zeta(x)-\overline{\omega})$ decrease on $[0,\re\omega)$ while $\arg (\zeta(x)-1)$ decreases on $[0,\widetilde{x}_{4,2}]$ and increase on $[\widetilde{x}_{4,2},\re\omega)$, where $\widetilde{x}_{4,2}$ is defined in \eqref{equ:tilde-x-4-2}. Define $x_{4,2,k}^+=0.05k$ for $0\leq k\leq 11$, $x_{4,2,12}^+=0.575$ and $x_{4,2,13}^+=\widetilde{x}_{4,2}$. Then one can verify that if $0\leq k\leq 12$, then
\renewcommand\theequation{\Alph{section}.\arabic{equation}*}
 \begin{align}
\widetilde{\Theta}_4^+(x_{4,2,k+1}^+,x_{4,2,k}^+)\ly\pi \text{\quad and\quad} \widetilde{\Theta}_4^+(x_{4,2,k}^+,x_{4,2,k+1}^+)\gy  0. \retainlabel{0101-Theta-4-1}
\end{align}
We define
\begin{equation}
\widehat{\Theta}_4^+(x,x'):=\theta_{4,1}^+(x')+\theta_{4,2}^+(x)-\theta_{4,3}^+(x)-\theta_{4,4}^+(x)-\theta_{4,5}^+(x)+2\pi.
\end{equation}
Hence $\widehat{\Theta}_4^+(x,x')\leq \arg Q(\zeta)+2\pi\leq \widehat{\Theta}_4^+(x',x)$ for any $\widetilde{x}_{4,2}\leq x\leq\re\zeta\leq x'<\re\omega$.
Define $x_{4,2,k}^+=0.625+0.005(k-14)$ for $14\leq k\leq 20$, $x_{4,2,k}^+=0.659+0.001(k-21)$ for $21\leq k\leq 25$ and $x_{4,2,26}^+=0.6635$. One can verify that for $13\leq k\leq 25$, then
 \renewcommand\theequation{\Alph{section}.\arabic{equation}*}
 \begin{align}
\widehat{\Theta}_4^+(x_{4,2,k+1}^+,x_{4,2,k}^+)\ly\pi \text{\quad and\quad} \widehat{\Theta}_4^+(x_{4,2,k}^+,x_{4,2,k+1}^+)\gy  0. \retainlabel{0101-Thetahat-4-1}
\end{align}
If $x\in[0.6635,\re \omega)$, then
\renewcommand\theequation{\Alph{section}.\arabic{equation}*}
\begin{align}
&~\arg Q(\zeta)+2\pi \leq
\theta_{4,1}^+(0.6635)+\theta_{4,2}^+(\re\omega)-\theta_{4,3}^+(\re\omega)\\
&\qquad -4\Big(\arctan\big(\tfrac{\im\omega-0.69}{\re\omega+0.22}\big)+\tfrac{\pi}{2}\Big)
-4\cdot \tfrac{\pi}{2}+2\pi \Aeq{3.1331}\ly\pi \retainlabel{3.1331-argQ}
\end{align}
and
 \renewcommand\theequation{\Alph{section}.\arabic{equation}*}
\begin{align}
\arg Q(\zeta)+2\pi\geq \widehat{\Theta}_4^+(0.6635,\re \omega)\Aeq{3.0939}\gy 0. \retainlabel{3.0939-argQ}
\end{align}
Hence we have $0<\arg Q(\zeta)+2\pi<\pi$ for $\zeta\in\ell_{4,2}^+$ and the proof is finished.
\end{proof}

\begin{proof}[{Proof of Lemma \ref{lema-W0}(c) and Lemma \ref{lema-a4}(b)}]
The arc $(\partial\D(a_5,\varepsilon_5)\setminus\D)\cap\BH_+$ consists of two end points, one is $\omega$ and the other is $\zeta_{5,1}^-= x_{5,1}^-$, where
\renewcommand\theequation{\Alph{section}.\arabic{equation}*}
\begin{align}
x_{5,1}^-=&~\re a_5+\sqrt{1+(\re a_5)^2-2\,\re a_5\re\omega-2\,\im a_5\,\im\omega} \\
&~\Aeq{1.2886}. \retainlabel{1.2886-x51}
\end{align}
Such arc can be divided into $3$ subarcs:
\begin{align}
\ell_{5,p}^+(x)&:=x+y_5^+(x)\ii, \quad x_{5,p}^+< x\leq x_{5,p+1}^+,~p=0,1, \\
\ell_{5,1}^-(x)&:=x+y_5^-(x)\ii, \quad x_{5,1}^-< x\leq x_{5,2}^-,
\end{align}
where $x_{5,0}^+=\re \omega\Aeq{0.6638}$, $x_{5,1}^+=1$ and $x_{5,2}^\pm=\re a_5+\varepsilon_5\Aeq{1.3302}$.
Recall that $W_0$ is defined in Lemma \ref{lema-W0}(b). By the dynamical symmetry of $Q$ about the real line, it suffices to prove that $Q(\zeta)\in (\C\setminus \overline{W}_0)\cap\BH_+$ when $\zeta\in\ell_{5,0}^+\cup\ell_{5,1}^+\cup\ell_{5,1}^-$.

\medskip
(\textbf{Case 1}) Assume that $\zeta(x)=x+y_5^+(x)\ii\in\ell_{5,0}^+$ with $x\in(\re\omega,1]$. It is easy to see that $\arg (\zeta(x)+1)$, $\arg (\zeta(x)-1)$, $\arg \zeta(x)$, $\arg (\zeta(x)-\omega)$ and $\arg (\zeta(x)-\overline{\omega})$ decrease on $(\re\omega, 1]$. Hence we have $\Theta_5^+(x,x')\leq \arg Q(\zeta)\leq \Theta_5^+(x',x)$ for any $\re\omega< x\leq\re\zeta\leq x'\leq 1$.
Define $x_{5,0,k}^+=0.66+0.03k$ for $1\leq k\leq 11$ and $x_{5,0,12}^+=1$. One can verify that for $1\leq k\leq 11$, then
\renewcommand\theequation{\Alph{section}.\arabic{equation}*}
 \begin{align}
\tfrac{3\pi}{4}\Aeq{2.3561}\ly\Theta_5^+(x_{5,0,k}^+,x_{5,0,k+1}^+)\leq \Theta_5^+(x_{5,0,k+1}^+,x_{5,0,k}^+)\ly \pi. \retainlabel{10-equ-Yan-1}
\end{align}
Moreover, if $x\in (\re\omega,x_{5,0,1}^+]$, we have
 \renewcommand\theequation{\Alph{section}.\arabic{equation}*}
\begin{align}
\arg Q(\zeta)\leq \Theta_5^+( x_{5,0,1}^+,\re\omega)\Aeq{2.7110}\ly\pi \retainlabel{2.7110-argQ}
\end{align}
and
\renewcommand\theequation{\Alph{section}.\arabic{equation}*}
\begin{align}
\arg Q(\zeta) \geq&~
\theta_{5,1}^+(x_{5,0,1}^+)+\theta_{5,2}^+(x_{5,0,1}^+)-\theta_{5,3}^+(\re\omega)-4\arctan\big(\tfrac{\re(a_5-\omega)}{\im(\omega-a_5)}\big)\\
&~ -4\cdot\tfrac{\pi}{2}\Aeq{2.3813}\gy \tfrac{3\pi}{4}\Aeq{2.3561}. \retainlabel{2.3561-argQ}
\end{align}
Hence we have $\tfrac{3\pi}{4}<\arg Q(\zeta)<\pi$ for $\zeta\in\ell_{5,0}^+$. This implies that $Q(\zeta)\in (\C\setminus \overline{W}_0)\cap\BH_+$ when $\zeta\in\ell_{5,0}^+$.

\medskip
(\textbf{Case 2}) Assume that $\zeta(x)=x+y_5^+(x)\ii\in\ell_{5,1}^+$ with $x\in(1,x_{5,2}^+]$. It is easy to see that $\arg (\zeta(x)+1)$, $\arg (\zeta(x)-1)$, $\arg \zeta(x)$, $\arg (\zeta(x)-\omega)$ and $\arg (\zeta(x)-\overline{\omega})$ decrease on $(1,x_{5,2}^+]$. Define $x_{5,1,k}^+=1+0.04k$ with $k=0,1,2$. Then one can verify that
if $k=0,1$, then
\renewcommand\theequation{\Alph{section}.\arabic{equation}*}
 \begin{align}
\tfrac{3\pi}{4}\Aeq{2.3561}\ly\Theta_5^+(x_{5,1,k}^+,x_{5,1,k+1}^+)\leq \Theta_5^+(x_{5,1,k+1}^+,x_{5,1,k}^+)\ly \pi. \retainlabel{10-equ-Yan-25}
\end{align}

Let $\iota_{5,1}$, $\iota_{5,2}$ and $\iota_{5,3}$  be the lines passing through $a_5$ and $-1$, $a_5$ and $0$, $a_5$ and $\overline{\omega}$,  with slopes $s_{5,1}=\frac{\im a_5}{\re a_5+1}$, $s_{5,2}=\frac{\im a_5}{\re a_5}$ and $s_{5,3}=\frac{\im (a_5+\omega)}{\re(a_5-\omega)}$ respectively. We use $\zeta(\widetilde{x}_{5,i})$, where $i=1,2,3$, to denote one of the points in $\iota_{5,i}\cap\partial\D(a_5,\varepsilon_5)$ with larger imaginary part. Then $\widetilde{x}_{5,1}=\re a_5+\varepsilon_5/\sqrt{1+s_{5,1}^2}\Aeq{1.3264}$, $\widetilde{x}_{5,2}=\re a_5+\varepsilon_5/\sqrt{1+s_{5,2}^2}\Aeq{1.3113}$ and $\widetilde{x}_{5,3}=\re a_5+\varepsilon_5/\sqrt{1+s_{5,3}^2}\Aeq{0.8462}$.
This implies that $|\zeta(x)+1|$ increases on $(1,\widetilde{x}_{5,1}]$ and decreases on $[\widetilde{x}_{5,1},x_{5,2}^+]$, $|\zeta(x)|$ increases on $(1,\widetilde{x}_{5,2}]$ and decreases on $[\widetilde{x}_{5,2},x_{5,2}^+]$. Moreover, $|\zeta(x)-1|$, $|\zeta(x)-\overline{\omega}|$ decrease on $(1,x_{5,2}^+]$, and $|\zeta(x)-\omega|$ increases on $(1,x_{5,2}^+]$.
Based on this, we define:
\begin{equation}
\widetilde{\Xi}_5^+(x,x'):=
\left\{
\begin{aligned}
&\frac{\xi_{5,1}^+(x')\cdot \xi_{5,2}^+(x)}{\xi_{5,3}^+(x)\cdot\xi_{5,4}^+(x)\cdot\xi_{5,5}^+(x')}, \text{\quad where } [x,x']\subset[1.08,\widetilde{x}_{5,2}], \\
&\frac{\xi_{5,1}^+(x')\cdot \xi_{5,2}^+(x)}{\xi_{5,3}^+(x')\cdot\xi_{5,4}^+(x)\cdot\xi_{5,5}^+(x')}, \text{\quad  where } [x,x']\subset[\widetilde{x}_{5,2},\widetilde{x}_{5,1}], \\
&\frac{\xi_{5,1}^+(x)\cdot \xi_{5,2}^+(x)}{\xi_{5,3}^+(x')\cdot\xi_{5,4}^+(x)\cdot\xi_{5,5}^+(x')}, \text{\quad  where } [x,x']\subset[\widetilde{x}_{5,1},x_{5,2}^+].
\end{aligned}
\right.
\end{equation}
If $[x,x']$ is contained in $[1.08,\widetilde{x}_{5,2}]$, or $[\widetilde{x}_{5,2},\widetilde{x}_{5,1}]$ or $[\widetilde{x}_{5,1},x_{5,2}^+]$, then $|Q(\zeta)|\leq \widetilde{\Xi}_5^+(x,x')$.
Consider following two functions:
\begin{equation}
\begin{split}
h_5^+(x,x') & :=\widetilde{\Xi}_5^+(x,x')\cos \Theta_5^+(x,x'), \text{ and}\\
v_5^+(x,x') & :=\widetilde{\Xi}_5^+(x,x')\sin \Theta_5^+(x,x')-\big(\tfrac{21}{5}\sqrt{2}-h_5^+(x,x')\big).
\end{split}
\end{equation}
We define $x_{5,1,k}^+=1.1+0.02(k-3)$ with $3\leq k\leq 10$, $x_{5,1,k}^+=1.25+0.01(k-11)$ with $11\leq k\leq 16$, $x_{5,1,17}^+=1.305$ and $x_{5,1,18}^+=\widetilde{x}_{5,2}$; $x_{5,1,19}^+=1.317$, $x_{5,1,20}^+=1.323$ and $x_{5,1,21}^+=\widetilde{x}_{5,1}$; $x_{5,1,22}^+=1.329$ and $x_{5,1,23}^+=x_{5,2}^+$. One can verify that if $2\leq k\leq 22$, then
\renewcommand\theequation{\Alph{section}.\arabic{equation}*}
\begin{align}
\tfrac{\pi}{4}\Aeq{0.7853}&\ly \Theta_5^+(x_{5,1,k}^+,x_{5,1,k+1}^+)\leq \Theta_5^+(x_{5,1,k+1}^+,x_{5,1,k}^+)\ly \pi, \quad\qquad \retainlabel{x-5-2-k-33} \\
h_5^+(x_{5,1,k}^+,x_{5,1,k+1}^+)&\ly 1.7 \text{\quad and\quad} v_5^+(x_{5,1,k}^+,x_{5,1,k+1}^+)\ly 0. \retainlabel{x-5-2-k-44}
\end{align}
This implies that $Q(\zeta)\in (\C\setminus \overline{W}_0)\cap\BH_+$ when $\zeta\in\ell_{5,1}^+$.

\medskip
(\textbf{Case 3}) Assume that $\zeta(x)=x+y_5^-(x)\ii\in\ell_{5,1}^-$ with $x\in(x_{5,1}^-,x_{5,2}^-]$. We first consider the points which are close to $x_{5,1}^-$. The point $Q(x_{5,1}^-)\Aeq{0.9514}$ lies on the positive axis and one can choose a univalent branch of $Q^{-1}$ in $\D(Q(x_{5,1}^-),0.7)$ since the critical values of $Q$ are $0$, $cv(\doteqdot 16.94\ldots)$ and $\infty$ (see Lemma \ref{lema-cp-cv-Q}(c)). According to Koebe's distortion theorem, it implies that $Q$ is univalent in $\D(x_{5,1}^-,r_{5,1})$ and $Q(\D(x_{5,1}^-,r_{5,1}))\subset\D(Q(x_{5,1}^-),0.7)$, where
\renewcommand\theequation{\Alph{section}.\arabic{equation}*}
\begin{align}
r_{5,1}=\tfrac{1}{4}\cdot\tfrac{0.7}{|Q'(x_{5,1}^-)|}\Aeq{0.0176}. \retainlabel{0.0176-r51}
\end{align}
Define $x_{5,1,1}^-=1.293$. Then
\renewcommand\theequation{\Alph{section}.\arabic{equation}*}
\begin{align}
x_{5,1,1}^- -x_{5,1}^-\Aeq{0.0043} & \ly \tfrac{0.0175}{\sqrt{2}}\Aeq{0.0123} \text{\quad and} \retainlabel{0.0123-x511} \\
 y_5^-(x_{5,1,1}^-)\Aeq{0.0109} & \ly \tfrac{0.0175}{\sqrt{2}}\Aeq{0.0123}. \retainlabel{0.0123-y5-x511}
\end{align}
This implies that $Q(\zeta)\in (\C\setminus \overline{W}_0)\cap\BH_+$ when $\zeta(x)=x+y_5^-(x)\ii\in\ell_{5,1}^-$ with $x\in[x_{5,1}^-,x_{5,1,1}^-]$.

\medskip
Let $\iota_{5,4}$  be the line passing through $\overline{\omega}$ and $x_{5,1}^-$  with slope $s_{5,4}=\frac{\im \omega}{\re (x_{5,1}^--\omega)}$. We claim that $\iota_{5,4}\cap\partial\D(a_5,\varepsilon_5)$ consists of two different points. Indeed, denote $\widetilde{x}_{5,4}:=1.2<x_{5,1}^-\Aeq{1.2886}$ and we have
\renewcommand\theequation{\Alph{section}.\arabic{equation}*}
\begin{align}
(\widetilde{x}_{5,4}-\re a_5)^2+(s_{5,4}(\widetilde{x}_{5,4}-x_{5,1}^-)-\im a_5)^2-\varepsilon_5^2\Aeq{-0.0264}\ly 0. \retainlabel{0.0264-x54}
\end{align}
Based on this, it is easy to see that $\arg (\zeta(x)+1)$, $\arg (\zeta(x)-1)$, $\arg \zeta(x)$, $\arg (\zeta(x)-\omega)$ and $\arg (\zeta(x)-\overline{\omega})$ increase on $[x_{5,1,1}^-,x_{5,2}^-]$. On the other hand, $|\zeta(x)+1|$, $|\zeta(x)-1|$, $|\zeta(x)|$ and $|\zeta(x)-\overline{\omega}|$ increase and $|\zeta(x)-\omega|$ decreases on $[x_{5,1,1}^-,x_{5,2}^-]$.
This implies that for any $[x,x']\subset[x_{5,1,1}^-,x_{5,2}^-]$, then $\Theta_5^-(x',x)\leq \arg Q(\zeta)\leq \Theta_5^-(x,x')$ and $\widetilde{\Xi}_5^-(x',x)\leq |Q(\zeta)|\leq \widetilde{\Xi}_5^-(x,x')$, where
\begin{equation}
\widetilde{\Xi}_5^-(x,x'):=\frac{\xi_{5,1}^-(x')\cdot \xi_{5,2}^-(x')}{\xi_{5,3}^-(x)\cdot\xi_{5,4}^-(x')\cdot\xi_{5,5}^-(x)}.
\end{equation}

Define $x_{5,1,2}^-=1.297$, $x_{5,1,k}^-=1.3+0.005(k-3)$ for $3\leq k\leq 8$, $x_{5,1,9}^-=1.327$, $x_{5,1,10}^-=1.329$  and $x_{5,1,11}^-=x_{5,2}^-$. One can verify that for $1\leq k\leq 10$, then
\renewcommand\theequation{\Alph{section}.\arabic{equation}*}
\begin{align}
0 \ly &~\Theta_5^-(x_{5,1,k+1}^-,x_{5,1,k}^-)\leq \Theta_5^-(x_{5,1,k}^-,x_{5,1,k+1}^-)\ly \pi, \retainlabel{x-55-2-k-5} \\
&~\widetilde{\Xi}_5^-(x_{5,1,k}^-,x_{5,1,k+1}^-)\ly \tfrac{21}{5}\sqrt{2}-1.7\Aeq{4.2396} \text{\quad and} \retainlabel{x-55-2-k-7} \\
&~\widetilde{\Xi}_5^-(x_{5,1,k}^-,x_{5,1,k+1}^-) \cos\Theta_5^-(x_{5,1,k+1}^-,x_{5,1,k}^-)\ly 1.7. \retainlabel{x-55-2-k-6}
\end{align}
This implies that $Q(\zeta)\in (\C\setminus \overline{W}_0)\cap\BH_+$ when $\zeta(x)=x+y_5^-(x)\ii\in\ell_{5,1}^-$ with $x\in[x_{5,1,1}^-,x_{5,2}^-]$ and the proof is finished.
\end{proof}


\section*{Acknowlegdements}

The project of this paper began in October 2014, when the author was a visiting scholar at Institut de Math\'{e}matiques de Toulouse. The author would like to thank its hospitality during his one-year visit. He is grateful to Arnaud Ch\'{e}ritat for helpful comments, Davoud Cheraghi for huge encouragements, modifications on an earlier version and helpful suggestions on the presentation of the paper, and Hiroyuki Inou for useful suggestions on the numerical calculations.
He is also very grateful to the referee for careful reading and helpful suggestions and the editor for kind comments.

\bibliographystyle{amsalpha}
\bibliography{E:/Latex-model/Ref1}

\end{document}